\documentclass[11pt]{article} 

\usepackage{graphicx}
\usepackage[english]{babel}
\usepackage{amssymb,amsmath,amsthm}  
\usepackage{amsfonts}
\usepackage{color}
\usepackage{hyperref}
\numberwithin{equation}{section}

\usepackage{anysize}
\marginsize{3.5cm}{3.5cm}{2.2cm}{2.2cm}

\theoremstyle{plain}
\newtheorem{theorem}{Theorem}[section]
\newtheorem{proposition}[theorem]{Proposition}
\newtheorem{lemma}[theorem]{Lemma}

\newtheorem{definition}{Definition}[section]

\theoremstyle{definition}
\newtheorem{remarkbase}[theorem]{Remark}
\newenvironment{remark}{\pushQED{\qed} \remarkbase}{\popQED\endremarkbase}

\let\OLDthebibliography\thebibliography
\renewcommand\thebibliography[1]{
  \OLDthebibliography{#1}
  \setlength{\parskip}{0pt}
  \setlength{\itemsep}{0pt plus 0.3ex} }

\newcommand{\mA}{\mathcal{A}}
\newcommand{\mB}{\mathcal{B}}

\newcommand{\mD}{\mathcal{D}}
\newcommand{\mE}{\mathcal{E}}
\newcommand{\mF}{\mathcal{F}}
\newcommand{\mG}{\mathcal{G}}
\newcommand{\mH}{\mathcal{H}}

\newcommand{\mL}{\mathcal{L}}

\newcommand{\mN}{\mathcal{N}}

\newcommand{\mS}{\mathcal{S}}

\newcommand{\mR}{\mathcal{R}}
\newcommand{\mT}{\mathcal{T}}

\renewcommand{\a}{\alpha}
\renewcommand{\b}{\beta}
\newcommand{\g}{\gamma}
\renewcommand{\d}{\delta}

\newcommand{\e}{\varepsilon}
\newcommand{\ph}{\varphi}
\newcommand{\vphi}{\varphi}
\newcommand{\lm}{\lambda}
\newcommand{\Om}{\Omega}
\newcommand{\om}{\omega}
\newcommand{\p}{\pi}
\newcommand{\s}{\sigma}
\renewcommand{\t}{\tau}

\renewcommand{\t}{\tau }

\newcommand{\teta}{\theta}

\newcommand{\be}{\begin{equation}}
\newcommand{\ee}{\end{equation}}

\newcommand{\ii}{{\mathrm i}}

\newcommand{\gr}{\nabla}

\newcommand{\la}{\langle}
\newcommand{\ra}{\rangle}

\newcommand{\R}{\mathbb R}
\newcommand{\C}{\mathbb C}
\newcommand{\Z}{\mathbb Z}
\newcommand{\N}{\mathbb N}
\newcommand{\T}{\mathbb T}

\newcommand{\inv}{^{-1}}

\newcommand{\pa}{\partial}

\newcommand{\fracchi}{{\mathfrak{I}}}

\newcommand{\Lipg}{{\mathrm{Lip}(\g)}}
\newcommand{\lip}{\mathrm{lip}}

\begin{document}

\title{\textbf{KAM for autonomous quasi-linear perturbations of mKdV}}
\date{}
\author{Pietro Baldi, Massimiliano Berti, Riccardo Montalto}

\maketitle

\begin{small}
\noindent
\textbf{Abstract.} 
We prove the existence  of Cantor families of small amplitude, 
linearly  stable,  quasi-periodic solutions of \emph{quasi-linear} (also called strongly nonlinear) 
autonomous Hamiltonian differentiable perturbations of  the mKdV equation.
The proof is based on a weak version of the Birkhoff normal form algorithm and 
a nonlinear Nash-Moser iteration. 
The analysis of the linearized operators at each step of the iteration 
is achieved by pseudo-differential operator techniques and a linear KAM reducibility scheme. 

\smallskip


\end{small}

\section{Introduction and main result}

In the paper \cite{bbm} we proved the first existence result of 
quasi-periodic solutions for autonomous quasi-linear PDEs 
(also called  ``{\it strongly nonlinear}'' in \cite{k1}), in particular 
of small amplitude quasi-periodic  solutions of the KdV equation subject to  
a Hamiltonian quasi-linear perturbation.
The approach developed in \cite{bbm} (see also \cite{BBM1})
is of wide applicability for quasi-linear PDEs in 1 space dimension. 
In this paper we take the opportunity to explain the general strategy of \cite{bbm} 
applied to a model which is slightly simpler than KdV.

\smallskip

We consider the cubic, focusing or defocusing, mKdV equation
\begin{equation}\label{eq:1}
u_t + u_{xxx} + \varsigma \, \pa_x(u^3) + {\mathcal N}_4 (x, u, u_x, u_{xx}, u_{xxx}) = 0\,, 
\quad \varsigma = \pm1, 
\end{equation}
under periodic boundary conditions $ x \in \T := \R / 2 \p \Z$, 
where 
\begin{equation}\label{qlpert}
{\mathcal N}_4 (x, u, u_x, u_{xx}, u_{xxx}) := - \partial_x \big[ (\partial_u f)(x, u,u_x)  - \partial_{x} ((\partial_{u_x} f)(x, u,u_x)) \big]   
\end{equation}
is the most general quasi-linear Hamiltonian (local) nonlinearity. 
Note that $ {\cal N}_4 $ contains as many derivatives as the linear vector field $ \pa_{xxx} $.
It is a {\it quasi-linear} perturbation because $ {\cal N}_4 $ depends linearly on the highest derivative $ u_{xxx} $ multiplied by a  coefficient 
which is a nonlinear function of the lower order derivatives $ u, u_x, u_{xx} $. 
The equation \eqref{eq:1} is the Hamiltonian PDE
\be\label{KdV:HS}
u_t =  X_H (u) \, , \quad 
 X_H (u) := \partial_x \nabla H (u) \, , 
\ee
where $ \nabla H $ denotes the $L^2(\T_x)$ gradient of the Hamiltonian 
\begin{equation} \label{Ham in intro}
H(u) = \frac12 \int_\T u_x^2 \, dx - \frac{\varsigma}{4} \int_\T u^4 \, dx 
+ \int_\T f(x, u,u_x) \, dx 
\end{equation}
on the real phase space   
\begin{equation}\label{def phase space}
H^1_0 (\T_x) := \Big\{ u(x ) \in H^1 (\T, \R) \ : \ \int_{\T} u(x) \, dx = 0  \Big\} 
\end{equation}
endowed with the non-degenerate symplectic form
\be\label{2form KdV}
\Omega (u, v) := \int_{\T} (\partial_x^{-1 } u) \, v \, dx \, ,   \quad 
\forall u, v \in H^1_0 (\T_x) \, , 
\ee
where $ \partial_x^{-1} u $ is the periodic primitive of $ u $ with zero average. 
The phase space $ H^1_0 (\T_x) $ is invariant for the evolution of \eqref{eq:1} 
because the integral $ \int_{\T} u(x) \, dx  $ is a prime integral (the mass). 
For simplicity we fix its value to $ \int_{\T} u(x) \, dx = 0 $. 
We recall that the Poisson bracket 
between two functions $ F $, $ G : H^1_0(\T_x) \to \R$ is defined as
\begin{equation}\label{Poisson bracket}
\{ F, G \}(u) := \Om (X_F(u), X_G(u) ) = \int_{\T} \gr F(u) \pa_x \gr G (u) dx \, . 
\end{equation}

We assume that the ``Hamiltonian density'' $f$ is of class  
$C^q (\T \times \R \times \R; \R ) $ for some $ q $ large enough
(otherwise, as it is well known, we cannot expect the existence of smooth invariant KAM tori).  
We also assume that  $ f $ vanishes of order five around $ u = u_x = 0$, namely 
\begin{equation}\label{order5}
|f(x,u,v)| \leq C (|u| + |v|)^5 \quad \forall (u,v) \in \R^2 \, ,  \ |u|+|v| \leq 1.
\end{equation}
As a consequence 
the nonlinearity $ {\mathcal N}_4 $ vanishes of order $ 4 $ at $ u = 0 $ and 
 \eqref{eq:1}
may be seen, close to the origin, as  a ``small'' perturbation of the cubic mKdV equation
\begin{equation}\label{KdVmKdV} 
u_t + u_{xxx} + 3 \varsigma u^2 u_x = 0 \, . 
\end{equation}
Such equation is known to be completely integrable. 
Actually it is mapped into KdV by a Miura transform, 
and it may be described by global analytic action-angle variables,
as it was proved by Kappeler-Topalov \cite{KaT}. 
We also remark that, among the generalized KdV equations $ u_t + u_{xxx} \pm \pa_x (u^p) = 0$, $ p \in \N $,  
the only known completely integrable ones are the KdV $ p=2$ and the cubic mKdV $ p = 3 $.

\smallskip

It is a natural question to know whether the periodic, quasi-periodic or almost periodic solutions of \eqref{KdVmKdV} persist under small perturbations. 
This is the content of KAM theory. 
It is a difficult problem because of small divisors resonance phenomena,
which are especially strong in presence of quasi-linear perturbations like ${\cal N}_4$.

In this paper (as well as in \cite{bbm}) we restrict the analysis to the search 
of small amplitude solutions.
It is also a very interesting question to investigate possible extensions of this 
result to perturbations of finite gap solutions. 
A difficulty which 
arises in the search of small amplitude solutions 
is that the mKdV equation \eqref{eq:1} is a {\it completely resonant} PDE at $ u = 0 $, 
namely the linearized equation at the origin is the linear Airy equation 
$$ 
u_t + u_{xxx} = 0 
$$ 
which possesses only the $ 2 \pi $-periodic in time, real solutions
\begin{equation}\label{Airyper}
u(t,x) = \sum_{j \in \Z \setminus \{0\} } u_j e^{\ii j^3 t} e^{\ii jx } , \quad 
u_{- j} = {\bar u}_j .
\end{equation}
Thus the existence of small amplitude quasi-periodic solutions of \eqref{eq:1} is 
entirely due to  the nonlinearity. 
Indeed, the nonlinear term $\varsigma \pa_x (u^3)$ is the one that produces the main modulation 
of the frequency vector of the solution with respect to its amplitude 
(the well-known frequency-to-action map, or frequency-amplitude relation, or ``twist'', 
see \eqref{mappa freq amp})
and that allows to ``tune'' the action parameters $\xi$ so that the 
frequencies becomes rationally independent and diophantine. 
Note that the mKdV equation  \eqref{eq:1} does not depend on other external parameters
which may influence the frequencies. This is a further difficulty in the study of  autonomous PDEs
with respect to the forced cases studied in \cite{BBM}. Actually, 
in \cite{BBM} we considered non-autonomous quasi-linear (and fully nonlinear) 
perturbations of the Airy  equation and we used the forcing frequencies as  
independent parameters. 

\smallskip

The core of the matter is to understand the perturbative effect 
of the quasi-linear term $ {\cal N}_4 $ over \emph{infinite times}. 
By \eqref{order5}, close to the origin, the quartic term $ {\cal N}_4 $ 
is smaller than the pure cubic mKdV \eqref{KdVmKdV}. 
Therefore, when we restrict the equation to finitely many 
space-Fourier indices $ |j| \leq C $, we essentially enter 
in the range of applicability of finite dimensional KAM theory 
close to an elliptic equilibrium. 
The new problem is to understand what happens to the dynamics 
on the high frequencies $ |j| \to + \infty $, since $ {\cal N}_4 $ is a 
nonlinear differential operator of the {\it same} order (i.e.\ 3) 
as the constant coefficient linear (and integrable) vector field $ \pa_{xxx} $. 

\smallskip

Does such a strongly nonlinear perturbation give rise to the 
formation of singularities for a solution in finite time,  
as it happens for the quasi-linear wave equations
considered by Lax \cite{Lax} and Klainerman-Majda \cite{KM}? 
Or, on the contrary, does the KAM phenomenon persist nevertheless 
for the mKdV equation \eqref{eq:1}?  
The answer to these questions has been controversial for several years. 
For example, Kappeler-P\"oschel \cite{KaP} (Remark 3, page 19) wrote: 
``{It would be interesting to obtain perturbation results which also include terms of higher order, at least in the region where the KdV approximation is valid. However, results of this type are still out of reach, if true at all}''.

\smallskip

We think that these are very important dynamical questions to be investigated, 
especially because many of the equations 
arising in Physics are quasi-linear or even fully nonlinear. 

\smallskip

The main result of this paper 
proves that 
the KAM phenomenon actually persists, at least close to the origin,  
for quasi-linear Hamiltonian perturbations of mKdV 
(the same result is proved in \cite{bbm} for KdV). 
More precisely, Theorem \ref{thm:mKdV} proves the existence of Cantor families of 
small amplitude, linearly stable, quasi-periodic solutions of the mKdV equation
\eqref{eq:1} subject to quasi-linear Hamiltonian perturbations. 
It is not surprising that the same result applies for both the focusing 
and the defocusing mKdV because we are looking for small amplitude solutions. 
Thus the different sign $ \varsigma = \pm 1 $ only affects the branch of the bifurcation. 

From a dynamical point of view, note that the parameters $\xi$ selected by the KAM Theorem  \ref{thm:mKdV} give rise to solutions 
of \eqref{eq:1}-\eqref{qlpert} which are global in time.  
This is  interesting information because, as far as we know, 
there are no results of global or even local solutions of the Cauchy problem 
for \eqref{eq:1}-\eqref{qlpert}, and such PDEs are in general believed to be ill-posed in Sobolev spaces
(for a rough result of local well-posedness for \eqref{eq:1}-\eqref{qlpert} see \cite{Baldi-Floridia-Haus}).

The iterative procedure we are going to present is able to 
select many parameters $\xi$ which give rise to quasi-periodic solutions
(hence defined for all times). 
This procedure works for parameters  
belonging to a finite dimensional Cantor like set 
which becomes asymptotically dense at the origin.

How can this kind of result be achieved? 
The proof of Theorem \ref{thm:mKdV} -- which we shall discuss 
in more detail later -- is based on an iterative Nash-Moser scheme. 
As it is well known, the main step of this procedure is to invert the linearized operators 
obtained at each step of the iteration and to prove that the inverse operators, 
albeit they lose derivatives (because of small divisors), 
satisfy tame estimates in high Sobolev norms. 
The linearized equations are non-autonomous linear PDEs 
which depend quasi-periodically on time. 
The key point of this paper (and \cite{bbm}) is that, 
using the symplectic decoupling of \cite{BB13}, 
some techniques of pseudo-differential operators 
adapted to the symplectic structure, 
and a linear Birkhoff normal form analysis, 
we are able to construct, for most diophantine frequencies, 
a time dependent (quasi-periodic) change of variables 
which conjugates each linearized equation into another one 
that is diagonal and has constant coefficients, 
that is, in ``normal form''. 
This means that, in the new coordinates, we have \emph{integrated the equations}. 
Then we easily invert the linearized operator (recall that the inverse loses derivatives 
because of small divisors) and we conjugate it back 
to solve the linear equation in the original set of variables. 
We remark that these quasi-periodic Floquet changes of variable 
map Sobolev spaces of arbitrarily high norms into itself and satisfy tame estimates.  
Hence the inverse operator also loses derivatives, but it satisfies tame estimates as well. 

In the dynamical systems literature, this strategy is called ``reducibility''  of the equation 
and it is a quasi-periodic KAM perturbative extension of Floquet theory 
(Floquet theory deals with periodic solutions of finite dimensional systems).  
The difficulty to make it work in the present setting is due to the 
quasi-linear character of the nonlinearity in \eqref{eq:1}.

\smallskip

Before stating precisely our main result we shortly present some related literature. 
In the last years a big interest has been devoted to understand the effect of derivatives in the nonlinearity in KAM theory.   
For {\it unbounded} perturbations the first KAM results 
have been proved by Kuksin \cite{K2} and Kappeler-P\"oschel \cite{KaP}
for KdV (see also Bourgain \cite{B96}), and more recently by  Liu-Yuan \cite{LY}, Zhang-Gao-Yuan \cite{ZGY} for derivative 
NLS, and by Berti-Biasco-Procesi \cite{BBiP1}-\cite{BBiP2} for derivative NLW.
For a recent survey of known results for KdV, we refer to \cite{K13}. 
Actually all these results still concern semi-linear perturbations. 

The KAM theorems in \cite{K2}, \cite{KaP} prove the persistence 
of the finite-gap solutions of the integrable KdV 
under semilinear Hamiltonian perturbations $ \e \partial_{x} (\partial_u f) (x, u) $, 
namely when the density  $ f $ is independent of $ u_x $, 
so that \eqref{qlpert} is a differential operator of order $ 1 $. 
The key idea in \cite{K2} is to exploit the fact that the frequencies of KdV 
grow as $ \sim j^3 $ and the difference $ |j^3  - i^3| \geq \frac12 (j^2 + i^2) $, $i \neq j $, 
so that KdV gains (outside the diagonal) two derivatives.
This approach also works for Hamiltonian pseudo-differential perturbations of order 2 (in space),
using the improved Kuksin's lemma proved by Liu-Yuan in \cite{LY}. 
However it does {\it not} work for the general quasi-linear perturbation in \eqref{qlpert},
which is a nonlinear differential operator of the {\it same} order 
as the constant coefficient linear operator $ \partial_{xxx}$.

\smallskip

Now we state precisely the main result of the paper. 
The solutions we find are, at the first order of amplitude, 
localized in Fourier space on finitely many {``tangential sites''}  
\begin{equation}  \label{tang sites}
S^+ := \{ \bar\jmath_1, \ldots, \bar\jmath_\nu \}\,, \quad 
S := \{ \pm j : j \in S^+ \}\,, \quad 
{\bar \jmath}_i \in \N \setminus \{0\}  \quad \forall i =1, \ldots, \nu.
\end{equation}
The set $ S $ is required to be even because the solutions $ u $ 
of \eqref{eq:1} have to be real valued. 
Moreover, we also assume the following explicit ``non-degeneracy'' hypothesis on $S$: 
\begin{equation} \label{scelta siti}
\frac{2}{2\nu-1} \, \sum_{i=1}^\nu \bar \jmath_i^{\,2} \, \notin \, 
\Big\{ j^2 + kj + k^2 : \, j,k \in \Z \setminus S, \ \,  j \neq k \Big\}.
\end{equation}

\begin{theorem}[KAM for quasi-linear perturbations of mKdV] 
\label{thm:mKdV}
Given $ \nu \in \N $,  
let $ f  \in C^q $ (with $ q := q(\nu) $ large enough) satisfy \eqref{order5}.
Then, for all the  tangential sites $ S $ as in \eqref{tang sites} 
satisfying \eqref{scelta siti}, 
the mKdV equation \eqref{eq:1} possesses small amplitude quasi-periodic solutions 
with diophantine frequency vector $\om := \om(\xi) = (\om_j)_{j \in S^+} \in \R^\nu$ of the form 
\begin{equation}\label{solution u}
u(t,x) 
= \sum_{j \in S^+} 2 \sqrt{\xi_j} \, \cos( \om_j t + j x) + o( \sqrt{|\xi|} ), 
\end{equation}
where  
\begin{equation}\label{omj in thm}
\om_j := j^3 + 3 \varsigma \big[ \xi_j - 2 \big( \sum_{j' \in S^+} \xi_{j'} \big) \big] j, 
\quad j \in S^+,
\end{equation}
for a ``Cantor-like'' set of small amplitudes $ \xi \in \R^\nu_+ $ with density $ 1 $ at $ \xi = 0 $.
The term $o(\sqrt{|\xi|})$ in \eqref{solution u} 
is a function $u_1(t,x) = \tilde u_1(\om t, x)$, 
with $\tilde u_1$ in the Sobolev space $H^s(\T^{\nu+1},\R)$ of periodic functions, 
and Sobolev norm $\| \tilde u_1 \|_s = o(\sqrt{|\xi|})$ as $\xi \to 0$,
for some $s < q$. 
These quasi-periodic solutions are linearly stable.  

If the density $  f(u, u_x) $ is independent on $ x $, 
a similar result holds for {\it all} the choices of the tangential sites, 
without assuming \eqref{scelta siti}.
\end{theorem}

This result is deduced from Theorem \ref{main theorem}.  It was announced also in \cite{BBM1}-\cite{bbm}
under the stronger condition on the tangential sites 
\begin{equation} \label{too strong}
\frac{2}{2\nu-1} \, \sum_{i=1}^\nu \bar \jmath_i^{\,2} \, \notin \Z \, . 
\end{equation}
Let us make some comments. 
\begin{enumerate}
\item 
In the case $\nu = 1$ (time-periodic solutions), the condition \eqref{scelta siti} is always satisfied. 
Indeed, suppose, by contradiction, that there exist integers $\bar\jmath_1 \geq 1$, $j,k \in \Z$ such that 
\begin{equation} \label{rel}
2 \bar \jmath_1^{\,2} = j^2 + jk + k^2.
\end{equation}
Then $j^2 + jk + k^2$ is even, and therefore both $j$ and $k$ are even, 
say $j = 2n$, $k = 2m$ with $n,m \in \Z$. 
Hence $2 \bar \jmath_1^{\,2} = 4(n^2 + nm + m^2)$, and this implies that $\bar\jmath_1$ is even, 
say $\bar\jmath_1 = 2p$ for some positive integer $p$. It follows that 
$2 p^2 = n^2 + nm + m^2$, namely $p,n,m$ satisfy 
\eqref{rel}.
Then, iterating the argument, 
we deduce that $\bar\jmath_1$ can be divided by $2$ infinitely many times in $\N$, 
which is impossible. 

\smallskip
\item 
When the density $ f(u, u_x )$ is independent of $ x $, the $L^2$-norm 
\begin{equation} \label{def M}
M(u) := \int_\T u^2 \, dx = \| u \|_{L^2(\T)}^2
\end{equation}
is a prime integral of the Hamiltonian equation \eqref{eq:1}. 
Hence the solutions of \eqref{eq:1} are in one-to-one correspondence with those of the 
Hamiltonian equation
\begin{equation} \label{eq:Kv}
v_t = \pa_x \gr K(v) \ \quad \text{with} \ \quad 
K := H + \lm M^2 \, , \ \lm \in \R \, . 
\end{equation}
More precisely,  if $u(t,x)$ is a solution of \eqref{eq:1}, 
then  $v(t,x) := u(t, x-ct)$, with $c := -4\lm M(u)$, 
is a solution of \eqref{eq:Kv}. 
Vice versa, if $v(t,x)$ solves \eqref{eq:Kv}, 
then the function $u(t,x) := v(t, x+ct)$, with $c := -4\lm M(v)$, 
is a solution of \eqref{eq:1}
($M(v)$ is also a prime integral of the equation  \eqref{eq:Kv}).

The advantage of looking for quasi-periodic solutions of \eqref{eq:Kv} 
is that, for $ \lm = 3\varsigma /4 $, 
the fourth order Birkhoff normal form of $ K $ is diagonal (remark \ref{rem:lm M2-1}) and therefore 
no conditions on the tangential sites $ S $ are required (remark \ref{rem:lm M2-7}).

\smallskip
\item 
The diophantine frequency vector  $ \om(\xi) = (\om_j)_{j \in S^+} \in \R^\nu$ 
of the quasi-periodic solutions of Theorem \ref{thm:mKdV} 
is $ O(|\xi|) $-close as $ \xi \to 0 $  (see \eqref{omj in thm}) 
to the integer vector of the unperturbed linear frequencies 
\begin{equation}\label{bar omega}
\bar\om := (\bar\jmath_1^3, \ldots, \bar\jmath_\nu^3) \in \N^\nu \, . 
\end{equation}
This makes perturbation theory more difficult. This is the difficulty 
due to  the fact that the mKdV equation  
\eqref{eq:1} is completely resonant at $ u = 0 $. 

\smallskip

\item 
As shown by \eqref {solution u} the expected quasi-periodic solutions are mainly supported in Fourier space on the tangential sites $  S $.
The dynamics of the Hamiltonian PDE \eqref{eq:1} restricted (and projected) to the
symplectic  subspaces 
\begin{equation}\label{splitting S-S-bot}
H_S := \Big\{ v = \sum_{j \in S} u_j e^{\ii jx}   \Big\} \, , 
\quad 
H_S^\bot := \Big\{ z = \sum_{j \in S^c} u_j e^{\ii jx} \in H^1_0(\T_x)  \Big\} ,
\end{equation}
where  $S^c := \{ j \in \Z \setminus \{ 0 \} : j \notin S \}$,  
is quite different. We call $ v $  the {\it tangential} variable and $ z $ the {\it normal} one. 
On $ H_S $ the dynamics is mainly governed by a finite dimensional integrable system (see 
Proposition \ref{prop:weak BNF}), 
and we find it convenient to describe the dynamics in this subspace
by introducing   action-angle variable, see section \ref{sec:4}. 
On the infinite dimensional subspace $ H_S^\bot $ the solution will stay forever close to the elliptic equilibrium $ z = 0 $.  
\end{enumerate}

\medskip

In Theorem \ref{thm:mKdV} it is stated that the quasi-periodic solutions 
are linearly stable. 
This information is not only an important complement of the result, 
but also  an essential ingredient for the existence proof.  
Let us explain better what we mean. 
By the general procedure in \cite{BB13} we  prove that, around each invariant torus, 
there exist symplectic coordinates (see \eqref{trasformazione modificata simplettica}) 
$$ 
(\psi, \eta, w) \in \T^\nu \times \R^\nu \times  H_{S}^\bot 
$$ 
in which the mKdV Hamiltonian  \eqref{Ham in intro} assumes the normal form
\begin{align}\label{weak-KAM-normal-form}
K (\psi, \eta , w) & =  
\om \cdot \eta  + \frac12 K_{2 0}(\psi) \eta \cdot \eta +  \big( K_{11}(\psi) \eta , w \big)_{L^2(\T)} 
+ \frac12 \big(K_{02}(\psi) w , w \big)_{L^2(\T)} \nonumber \\
& \quad + K_{\geq 3}(\psi, \eta, w)  
\end{align}
where $ K_{\geq 3} $ collects the terms at least cubic in the variables $ (\eta, w )$, see
remark \ref{rem:KAM normal form}. 
In these coordinates the quasi-periodic solution reads 
$ t \mapsto (\om t , 0, 0 ) $ and the corresponding linearized equations are
\begin{equation}\label{sistema lineare dopo}
\begin{cases}
\dot \psi = K_{20}(\om t) \eta + K_{11}^T (\om t ) w  
\\
\dot \eta = 0 
\\
\dot w - \pa_x K_{0 2}(\om t ) w = \partial_x  K_{11}(\om t) \eta  \, .
\end{cases} 
\end{equation}
Thus the actions $ \eta (t) = \eta (0) $ do not evolve in time and 
the third equation reduces to the forced PDE
\begin{equation}\label{San pietroburgo modi normali}
\dot{w} = \pa_x K_{02}(\omega t)[w] + \pa_x K_{11}(\omega t)[ \eta_0] \, .
\end{equation}
Ignoring the forcing term $\pa_x K_{11}(\omega t)[ \eta_0]$ for a moment, 
we note that the equation $\dot{w} = \pa_x K_{02}(\omega t)[w]$ is,  
up to a finite dimensional remainder (Proposition \ref{prop:lin}),   
the restriction to $ H_{S}^\bot $ of the ``variational equation''
$$
h_t = \pa_x \, (\pa_u \nabla H)( u(\om t, x) ) [h ] = X_{K}(h) \, ,  
$$
where $ X_K $ is the KdV Hamiltonian vector field with quadratic Hamiltonian 
$ K = \frac12  ((\pa_u \nabla H)(u)[h], h)_{L^2(\T_x)} $ 
$= \frac12 (\pa_{uu} H)(u)[h, h] $. 
This is a linear PDE with quasi-periodically time-dependent coefficients of the form 
 \be\label{Lom KdVnew-simpler-0}
h_t =   \partial_{xx} (a_1 (\om t, x)  \partial_x h)  +
\partial_x ( a_0 (\om t, x) h ) \, .
\ee
In section \ref{operatore linearizzato sui siti normali} we prove 
the reducibility of the linear operator $ \dot w - \pa_x K_{0 2}(\om t ) w $, 
which conjugates \eqref{San pietroburgo modi normali}
to the diagonal system (see \eqref{Lfinale})
\begin{equation}\label{san pietroburgo siti normali ridotta}
\partial_t v  = - \ii {\cal D}_\infty  v + f (\omega t) 
\end{equation}
where ${\cal D}_\infty := {\rm Op} \{ \mu_j^\infty\}_{j \in S^c}$ 
is a Fourier multiplier operator acting in $ H^s_\bot $, 
$$
\mu_j^\infty := \ii (-m_3 j^3 + m_1 j) + r_j^\infty \in \ii \R \, , \quad j \in S^c \, , 
$$
with $ m_3 = 1 + O(\e^3) $, $ m_1 = O(\e^2) $, $ \sup_{j \in S^c} r_j^\infty = o(\e^2) $, 
see \eqref{espressione autovalori}, \eqref{autofinali}. 
The eigenvalues $ \mu_j^\infty $ are the {\it Floquet} exponents of the quasi-periodic solution. 
The solutions of the scalar  non-homogeneous equations 
$$
{\dot v}_j +  \mu_j^\infty v_j = f_j (\om t)  \, , \quad j \in S^c \, , 
\quad \mu_j^\infty \in \ii \R \, , 
$$
are
$$
v_j (t) = c_j e^{ \mu_j^\infty t} + {\tilde v}_j (t) \,, \quad \text{where} \quad 
{\tilde v}_j (t)  := \sum_{l \in \Z^\nu} \frac{f_{jl} \, e^{\ii \om \cdot l t } }{ \ii \om \cdot l + \mu_j^\infty } 
$$ 
(recall that the first Melnikov conditions \eqref{prime di melnikov} hold at a solution).
As a consequence, the Sobolev norm of the solution of \eqref{san pietroburgo siti normali ridotta} satisfies
$$
\| v(t) \|_{H^s_x} \leq C \| v(0)  \|_{H^s_x} \, , \quad \forall t \in \R \, , 
$$ 
i.e. it does not increase in time. 

\smallskip

We now describe in detail the strategy of proof of Theorem \ref{thm:mKdV}. 
Many of the arguments that we use 
are quite general and of wide applicability to other PDEs.
Nevertheless, we think that a unique abstract KAM theorem applicable to {\it all} quasi-linear PDEs 
can not be expected. 
Indeed the suitable pseudo-differential operators that are required to conjugate the
highest order of the linearized operator to constant coefficients highly depend on the PDE at hand,
see the discussion after \eqref{L6 qualitativo}. 

There are two main issues in the proof: 
\begin{enumerate}
\item {\bf Bifurcation analysis.} Find approximate quasi-periodic solutions of \eqref{eq:1} 
up to a sufficiently small remainder (which, in our case, should be $ O( u^4 ) $).
In this step we also find the approximate ``frequency-to-amplitude'' modulation 
of the frequency with respect to the amplitude, see  \eqref{mappa freq amp}.  
This is the goal of sections \ref{sec:WBNF} and \ref{sec:4}. 

\item {\bf Nash-Moser implicit function theorem.} 
Prove that, close to the above approximate solutions,  
there exist exact quasi-periodic solutions of \eqref{eq:1}. 
By means of a Nash-Moser iteration, we construct a sequence of approximate solutions 
that converges to a quasi-periodic solution of \eqref{eq:1} 
(sections \ref{sec:functional}-\ref{sec:NM}). 

The key step consists in proving the invertibility of the linearized operator and 
tame estimates for its inverse. This is achieved in two main steps.
\begin{enumerate}
\item {\sc Symplectic decoupling  procedure.} 
The method in Berti-Bolle \cite{BB13} allows to approximately decouple the ``tangential'' and the ``normal'' dynamics around an approximate invariant torus (section \ref{costruzione dell'inverso approssimato}). 
It reduces the problem to the one of inverting a quasi-periodically forced PDE 
restricted to the normal subspace $ H_S^\bot $. 
Its precise form is found in section \ref{forma-linear-normal}. 

\item {\sc Analysis of the linearized operator in the normal directions.} 
In sections \ref{linearizzato siti normali}, \ref{operatore linearizzato sui siti normali}
we reduce the linearized equations to constant coefficients. 
This involves three steps: 
\begin{enumerate} 
\item {\it Reduction  in decreasing symbols}, sections \ref{step1}-\ref{step3} and  \ref{step5}, 
\item {\it Linear Birkhoff normal form},  section \ref{BNF:step1},
\item {\it KAM reducibility}, section \ref{subsec:mL0 mL5}.  
\end{enumerate}
\end{enumerate}
All the changes of variables used in the steps i)-iii)  are $ \vphi $-dependent families of 
symplectic maps $ \Phi (\vphi) $ 
which act on the phase space $ H^1_0 (\T_x) $. Therefore they preserve 
the Hamiltonian dynamical systems structure of the conjugated  linear operators.
\end{enumerate}

Let us discuss these issues in detail.

\medskip
\noindent 
\emph{Weak Birkhoff normal form.} 
According to the orthogonal splitting 
$$
H^1_0 (\T_x) := H_S \oplus H_S^\bot
$$
into the symplectic subspaces defined in \eqref{splitting S-S-bot}, we decompose 
\begin{equation}  \label{u = v + z}
u = v + z, \quad  
v = \Pi_S u := \sum_{j \in S} u_j \, e^{\ii jx}, \quad  
z = \Pi_S^\bot u := \sum_{j \in S^c} u_j \, e^{\ii jx},  
\end{equation}
where $\Pi_S $, $\Pi_S^\bot $ denote the orthogonal projectors on $ H_S $, $ H_S^\bot $. 

We perform   a ``weak'' Birkhoff normal form (weak BNF),
whose goal is to find an invariant manifold of solutions of the third order approximate mKdV equation \eqref{eq:1}, 
on which the dynamics is completely integrable,  see section \ref{sec:WBNF}.  
We construct in  Proposition \ref{prop:weak BNF} a symplectic map $ \Phi_B $ 
such that the transformed Hamiltonian $\mH := H \circ \Phi_B$ 
possesses the invariant subspace $ H_S $  (see \eqref{splitting S-S-bot}).  
To this purpose we have to eliminate the term $ \int v^3 z \, dx $ (which is linear in $ z $). 
Then we  check that its dynamics on $ H_S $ is integrable and non-isocronous. 
For that we perform the classical finite dimensional Birkhoff normalization of the Hamiltonian term $ \int v^4 \, dx $
which turns out to be  integrable and non-isocronous. 

Since the present weak  Birkhoff map has to remove only finitely many monomials, 
it is the time $ 1 $-flow  map  of an Hamiltonian system whose Hamiltonian is supported 
on only finitely many Fourier indices.  
Therefore it  is close to the  identity up to  {\it finite dimensional} operators, see Proposition \ref{prop:weak BNF}.
The key advantage is that it modifies  $ {\mathcal N}_4 $ very mildly, only up to finite dimensional operators 
(see for example Lemma \ref{lemma astratto potente}), and thus the spectral analysis of the linearized equations 
(that we shall perform in section \ref{operatore linearizzato sui siti normali}) is essentially the
same as if we were in the {\it original} coordinates.

The weak normal form \eqref{widetilde cal H} does not remove (nor normalize) the monomials $ O(z^2) $. 
We point out that a stronger normal form that removes/normalizes the monomials $O(z^2)$ 
is also  well-defined (it is called ``partial Birkhoff normal form'' in Kuksin-P\"oschel \cite{KP} and P\"oschel \cite{Po3}). 
However, we do not use it because, for such a stronger normal form, 
the corresponding Birkhoff map is close to the identity 
only up to an operator of order $ O(\partial_x^{-1}) $, 
and so it would produce terms of order $ \partial_{xx} $ and $ \partial_x $.
For the same reason, we do not use the global nonlinear Fourier transform in \cite{KaT} (Birkhoff coordinates), 
which is close to the Fourier transform up to smoothing operators of order $ O(\partial_x^{-1}) $
(this is explicitly proved for KdV). 

We remark that mKdV is simpler than 
KdV because the nonlinearity in \eqref{eq:1} is cubic and not only quadratic, 
and, as a consequence, less steps of Birkhoff normal form are required
to reach the sufficient smallness for the Nash-Moser scheme to converge
(see Remark \ref{rem:one power less}).

\medskip
\noindent 
{\it Action-angle and rescaling.}
At this point we introduce action-angle variables on the tangential sites (section \ref{sec:4})
and, after the rescaling \eqref{rescaling kdv quadratica}, 
we  look for quasi-periodic solutions of the 
Hamiltonian \eqref{formaHep}.
Note that the coefficients of the  normal form $ {\cal N } $ in \eqref{Hamiltoniana Heps KdV}
depend on the angles $ \theta $, unlike the usual KAM theorems \cite{Po3}, \cite{Ku},
where the whole normal form is reduced to constant coefficients.
This is because the weak BNF of section \ref{sec:WBNF} did not normalize the quadratic terms $ O(z^2) $. 
These terms are dealt with the ``linear Birkhoff normal form'' (linear BNF)
in section \ref{BNF:step1}.
In some sense  the ``partial'' Birkhoff normal form of \cite{Po3} is split into the weak BNF of section \ref{sec:WBNF} and the linear BNF of sections \ref{BNF:step1}.
 
The present functional formulation with the introduction of 
the action-angle variables allows to prove the stability of the solutions 
(unlike the Lyapunov-Schmdit reduction approach). 

\medskip
\noindent
{\it Nonlinear functional setting and approximate inverse.}
We look for a zero of the nonlinear operator \eqref{operatorF}, where the unknown is the
torus embeddeding $ \vphi \mapsto i(\vphi ) $, and where the frequency $ \om $ is seen as an ``external'' parameter.
This formulation is convenient in order to verify the Melnikov non-resonance conditions
required to invert the linearized operators at each step. 
The solution is obtained  by a Nash-Moser iterative scheme in Sobolev scales. 
The key step  is to  construct  (for $ \omega  $ restricted to a suitable Cantor-like set) 
 an approximate inverse ({\it \`a la} Zehnder \cite{Z1}) of the linearized operator
at any approximate solution. Roughly, this means to find a linear operator which 
is an inverse  at an exact solution. A major difficulty is that the tangential and the normal dynamics near an invariant torus are strongly coupled. 

\medskip
\noindent 
{\it Symplectic approximate decoupling.}
The above difficulty is overcome by implementing the abstract procedure in Berti-Bolle \cite{BB13}, 
which was developed in order to prove the existence of quasi-periodic solutions 
for autonomous NLW (and NLS) with a multiplicative potential. 
This approach reduces the search of an approximate inverse for \eqref{operatorF} 
to the invertibility of a quasi-periodically forced PDE 
restricted to the normal directions. 
This method approximately decouples the 
tangential and the normal dynamics around an approximate invariant torus, 
introducing a suitable set of symplectic variables
$$ 
(\psi, \eta, w) \in \T^\nu \times \R^\nu \times H_S^\bot 
$$ 
near the torus, see \eqref{trasformazione modificata simplettica}.
Note that, in the first line of \eqref{trasformazione modificata simplettica}, $ \psi $ is the ``natural'' angle variable which coordinates  the torus, and, in the third line, the normal variable $ z $ is 
only translated by the component $ z_0 (\psi )$ of the torus.
The second line 
completes this transformation
to a symplectic one. The canonicity of this map  is proved in  \cite{BB13} using the isotropy of
the approximate invariant torus $ i_\d $, see Lemma \ref{toro isotropico modificato}.
In these new variables 
the  torus $ \psi \mapsto i_\d  (\psi) $ reads $ \psi \mapsto (\psi, 0, 0 )$. 
The main advantage of these coordinates is that the second equation in \eqref{operatore inverso approssimato} 
(which corresponds to the action variables of the torus) can be immediately solved, see \eqref{soleta}. 
Then it remains to solve the third equation \eqref{cal L omega}, {i.e.} to invert
the linear operator  $ {\cal L}_\om $. 
This is a quasi-periodic Hamiltonian perturbed linear  Airy equation of the form  
\be\label{Lom KdVnew-simpler}
h \mapsto {\cal L}_\om h  :=  \Pi_S^\bot \big( \om \! \cdot \! \partial_\vphi h + \partial_{xx} (a_1 \partial_x h)  +
\partial_x ( a_0 h ) + \partial_x \mR h    \big)  \, , \quad \forall h \in  H_S^\bot \, ,
\ee
where $ \mR $ is a finite dimensional remainder.   The exact form of $ {\cal L}_\om $ is obtained in Proposition \ref{prop:lin}, 
see \eqref{Lom KdVnew}. 

\medskip
\noindent 
{\it Reduction to constant coefficients of the linearized operator in the normal directions.} 
In section \ref{operatore linearizzato sui siti normali}
we conjugate the variable coefficients operator $ {\cal L}_\om $
to a diagonal operator with constant coefficients
which describes infinitely many harmonic oscillators
\begin{equation}\label{linearosc}
{\dot v}_j + \mu_j^\infty v_j  = 0 \, , \quad 
\mu_j^\infty := \ii (-m_3 j^3 + m_1 j) + r_j^\infty \in \ii \R \, , \quad 
j \notin S \, , 
\end{equation}
where the constants $  m_3 -1 $, $ m_1 \in \R $ and  $ \sup_j |r_j^\infty | $ are small, 
see Theorem \ref{teoremadiriducibilita}.
The main perturbative effect to the spectrum (and the eigenfunctions) of $ {\cal L}_\om $
is due to  the term $ a_1 (\omega t, x ) \partial_{xxx} $ (see \eqref{Lom KdVnew-simpler}), 
and it is too strong for the usual 
 reducibility KAM techniques to work directly. 
The conjugacy of $ {\cal L}_\om $ with \eqref{linearosc} is obtained in several steps. 
The first task (obtained in sections \ref{step1}-\ref{step5}) is to conjugate 
$ {\cal L}_\om $ to another 
Hamiltonian operator of $ H_S^\bot $ with constant coefficients 
\be\label{L6 qualitativo}
{\cal L}_5 := \Pi_S^\bot \big(\om \cdot \partial_\vphi + m_3 \partial_{xxx}  + m_1 \partial_x + R_5 \big) \Pi_S^\bot  \, , \quad m_1, m_3 \in \R \, , 
\ee
up to a small bounded remainder $ R_5 = O(\partial_x^0 ) $, see \eqref{def L6}. 
This expansion of $ {\cal L}_\om $ in ``decreasing symbols'' with constant coefficients 
follows \cite{BBM}, and it is somehow in the spirit of the works of Iooss, 
Plotnikov and Toland \cite{Ioo-Plo-Tol}-\cite{IP09}
in water waves theory, 
and Baldi \cite{Baldi-Benj-Ono} for Benjamin-Ono.
It is obtained by transformations which are very different from the usual KAM changes of variables. 
We underline that the specific form of these transformations depend on the structure of mKdV. 
For other quasi-linear PDEs the analogous reduction requires different transformations,
see for example Alazard-Baldi \cite{Alazard-Baldi}, Berti-Montalto \cite{Berti-Montalto} for recent developments of these techniques 
for gravity-capillary  water waves, and Feola-Procesi \cite{FP} for quasi-linear forced perturbations of 
Schr\"odinger equations.

The transformation of \eqref{Lom KdVnew-simpler} into \eqref{L6 qualitativo} is made in several steps. 
\begin {enumerate}
\item 
{\it Reduction of the highest order.} The first step (section \ref{step1}) is to eliminate the $ x $-dependence from the coefficient $ a_1 (\omega t, x ) \partial_{xxx} $ 
of the Hamiltonian operator $ {\cal L}_\om $. 
In order to find a symplectic diffeomorphism of $ H_S^\bot $ near  $ {\cal A}_\bot $, 
the starting point is to observe that the diffeomorphism   (see  \eqref{primo cambio di variabile modi normali})
$$
u \mapsto ({\cal A} u)(\vphi,x) := (1 + \beta_x(\vphi,x)) u(\vphi,x + \beta(\vphi,x)) \, , 
$$
is, for each $ \vphi \in \T^\nu $,  the time-one flow map of the time dependent 
Hamiltonian transport linear PDE 
\be \label{transport-free}
\partial_\tau u = \partial_x (b(\vphi, \tau, x) u) \, , \quad b (\vphi, \tau, x) := \frac{\beta(\vphi, x)}{1 + \tau \beta_x(\vphi, x)}\, , 
\ee
Actually the flow of \eqref{transport-free} is the path of symplectic diffeomorphisms 
$$ 
u (\vphi, x) \mapsto (1+ \tau \b_x (\vphi, x) ) u (\vphi, x+ \tau \b(\vphi, x) )\, , \quad  \tau \in [0,1] \, .  
$$
Thus, like in \cite{bbm}, we conjugate $ {\cal L}_\om $ with the symplectic  time 1 flow
 map of the projected Hamiltonian equation 
\begin{equation}\label{problemi di cauchy} 
\partial_\tau u = \Pi_S^\bot \partial_x (b(\tau, x) u) 
= \partial_x (b(\tau, x) u) - \Pi_S \partial_x (b(\tau, x) u)  \, ,
\quad u \in H_S^\bot 
 \end{equation}
 generated by the 
the quadratic Hamiltonian $ \frac12 \int_{\T} b(\tau, x) u^2 dx  $ restricted to $ H_S^\bot $. 
By Lemma \ref{modifica simplettica cambio di variabile} (which was proved in \cite{bbm}) 
such symplectic map 
differs from 
$ {\cal A}_\bot := \Pi_S^\bot  {\cal A} \Pi_S^\bot $ only for finite dimensional operators.

This step may be seen as a quantitative application of the Egorov theorem, see \cite{Taylor},
which describes how the principal symbol of a pseudo-differential operator (here $ a_1 (\om t, x) \pa_{xxx} $)
transforms under the flow of a  linear hyperbolic PDE (here \eqref{problemi di cauchy}). 

\smallskip

Because of the Hamiltonian structure, the previous step also eliminates the term $ O( \pa_{xx} )$, see \eqref{cal L1 Kdv}.
In section \ref{step2} we eliminate the time-dependence of the 
coefficient at the order $ \partial_{xxx} $.  

\item 
{\it Linear Birkhoff normal form.} In section \ref{BNF:step1} we eliminate 
the variable coefficient terms at the order $ O(\e^2 )$, which are present in 
the operator $ {\cal L}_\om $,  see \eqref{Lom KdVnew}-\eqref{a1p1p2}.
This is a consequence of the fact that the weak BNF procedure of section \ref{sec:WBNF} 
did not touch the quadratic terms $ O(z^2 ) $. 
These terms cannot be reduced to constants by the 
perturbative scheme in  section \ref{subsec:mL0 mL5} 
(developed in \cite{BBM})
which applies to terms $ R $  such that 
$ R \g^{ -1} \ll 1 $ where $ \g $ is the diophantine constant of the frequency vector $ \om $ 
(the case in \cite{BBM} is simpler because the diophantine constant is $ \gamma = O(1) $).
Here, as well as in \cite{bbm}, since mKdV is completely resonant,  
such $ \gamma = o(\e^2 ) $, see \eqref{omdio}.
The terms of size $\e^2$ are reduced to constant coefficients in section \ref{BNF:step1} 
by means of purely algebraic arguments (linear BNF), which, ultimately, 
stem from the complete integrability of the fourth order BNF of the 
mKdV equation \eqref{KdVmKdV}. More general nonlinearities should be dealt with the
normal form arguments of Procesi-Procesi \cite{PP1} for generic choices of the tangential sites. 
\end{enumerate}

\medskip
\noindent 
{\it Complete diagonalization of \eqref{L6 qualitativo}.} 
In section \ref{subsec:mL0 mL5} we apply the abstract KAM reducibility Theorem 4.2 of \cite{BBM}, 
 which 
completely diagonalizes the linearized operator, obtaining  \eqref{linearosc}. 
The required smallness condition \eqref{R6resto} for $ R_5 $ holds, 
after that the linear BNF of section \ref{BNF:step1} has put into constant coefficients 
the unbounded terms of nonperturbative size $\e^2$, 
and the conjugation procedure of sections \ref{step1}-\ref{step3} and \ref{step5} 
has arrived to a bounded and small remainder $R_5$. 

\medskip
\noindent 
{\it The Nash-Moser iteration to an invariant torus embedding.}
In section \ref{sec:NM} we perform the nonlinear Nash-Moser iteration which 
finally proves Theorem \ref{main theorem} and, therefore, Theorem \ref{thm:mKdV}.
The smallness condition that is required for the convergence of the scheme 
is $ \e^2 \| {\cal F}(\vphi, 0, 0 ) \|_{s_0+ \mu} \g^{-2}$ sufficiently small,  
see \eqref{nash moser smallness condition}.
It is verified because $ \| X_P(\vphi, 0 , 0 ) \|_s \leq_s \e^{5 - 2b} $ 
(Lemma \ref{lemma quantitativo forma normale}) and $ \g = \e^{2+a}$ with $ a >  0 $ small. 
See also remark \ref{rem:one power less} for a comparison between the smallness condition 
required here with the one in \cite{bbm}.  

\paragraph{Notation.}
We shall use the notation 
$$
a \leq_s b \quad \ \Longleftrightarrow \quad a \leq C(s) b \quad
\text{for  some  constant  } C(s) > 0  \, . 
$$ 
We denote by $ \pi_0 $ the operator 
\be\label{def:pi0}
u \mapsto \pi_0(u) := u -  \frac{1}{2\pi} \int_\T u \, dx \, .
\ee

\section{Functional setting}

For a function $u : \Om_o \to E$, $\om \mapsto u(\om)$, where $(E, \| \ \|_E)$ is a Banach space and 
$ \Om_o $ is a subset of $\R^\nu $, we define the sup-norm and the Lipschitz semi-norm
\begin{equation} \label{def norma sup lip}
\begin{aligned}
\| u \|^{\sup}_E 
& := \| u \|^{\sup}_{E,\Om_o} 
:= \sup_{ \om \in \Om_o } \| u(\om ) \|_E, \\ 
\| u \|^{\lip}_E 
& := \| u \|^{\lip}_{E,\Om_o}  
:= \sup_{\om_1 \neq \om_2 } 
\frac{ \| u(\om_1) - u(\om_2) \|_E }{ | \om_1 - \om_2 | }\,,
\end{aligned}
\end{equation}
and, for $ \g > 0 $, the Lipschitz norm
\begin{equation} \label{def norma Lipg}
\| u \|^{\Lipg}_E  
:= \| u \|^{\Lipg}_{E,\Om_o}
:= \| u \|^{\sup}_E + \g \| u \|^{\lip}_E  \, . 
\end{equation}
If $ E = H^s $ we simply denote $ \| u \|^{\Lipg}_{H^s} := \| u \|^{\Lipg}_s $.  
 
\paragraph{Sobolev norms.}
We  denote by 
\begin{equation}\label{Sobolev coppia}
\| u \|_s := \| u \|_{H^s( \T^{\nu + 1})} := \| u \|_{H^s_{\vphi,x} }
\end{equation}
the Sobolev norm of functions $ u = u(\vphi,x) $ in the Sobolev space $ H^{s} (\T^{\nu + 1} ) $. 
We denote by $ \| \ \|_{H^s_x} $ the Sobolev norm in the phase space of functions $ u :=  u(x) \in H^{s} (\T ) $.
Moreover $ \| \ \|_{H^s_\vphi} $ denotes the Sobolev norm of scalar functions, like 
the Fourier components $ u_j (\vphi)  $.  

We fix $ s_0 := (\nu+2) \slash 2 $ so that   $ H^{s_0} (\T^{\nu + 1} ) \hookrightarrow L^{\infty} (\T^{\nu + 1} ) $ and any space $ H^s (\T^{\nu + 1} ) $, $ s \geq s_0  $, is an algebra and satisfy the interpolation inequalities:
for $s \geq s_0$, 
$$
\| uv \|_s \leq C(s_0) \|u\|_s \|v\|_{s_0} + C(s) \|u\|_{s_0} \| v \|_s \, , 
\quad \forall u,v \in H^s(\T^d) \, .
$$
The above inequalities also  hold for  the norms $\Vert \  \Vert_s^{{\rm Lip}(\gamma)}$.

We  also denote 
\begin{align} 
H^s_{S^\bot} (\T^{\nu+1}) 
& := \big\{  u \in H^s(\T^{\nu + 1} )  \, : \, u (\vphi, \cdot ) \in H_S^\bot \  
\forall \vphi \in \T^\nu \big\} \,, \nonumber 
\\
H^s_{S} (\T^{\nu+1}) 
& := \big\{  u \in H^s(\T^{\nu + 1} )  \, : \, u (\vphi, \cdot ) \in H_{S} \   
\forall \vphi \in \T^\nu \big\} \,.  \nonumber 
\end{align}

\paragraph{Matrices with off-diagonal decay.}
A linear operator can be identified, as usual, with its matrix representation. 
We recall the definition of the $ s $-decay norm (introduced in \cite{BB13JEMS}) 
of an infinite dimensional matrix. 

\begin{definition}\label{def:norms}
Let $ A := (A_{i_1}^{i_2} )_{i_1, i_2 \in \Z^b } $, $b \geq 1$, be an infinite dimensional matrix.  
Its $s$-decay norm $|A|_s$ is defined by 
\begin{equation} \label{matrix decay norm}
\left| A \right|_{s}^2 := 
\sum_{i \in \Z^b} \left\langle i \right\rangle^{2s} 
\big( \sup_{ \begin{subarray}{c} i_{1} - i_{2} = i 
\end{subarray}}
| A^{i_2}_{i_1}| \big)^{2}.
\end{equation}
For parameter dependent matrices $ A := A(\omega) $, $\omega  \in \Omega_o \subseteq \R^\nu $, 
the definitions \eqref{def norma sup lip} and \eqref{def norma Lipg} become 
\begin{equation} \label{matrix decay norm Lip}
| A |^{\sup}_s  := \sup_{ \omega \in \Omega_o } | A(\om ) |_s, \quad
| A |^{\lip}_s := \sup_{\om_1 \neq \om_2} 
\frac{ | A(\om_1) - A(\om_2) |_s }{ | \om_1 - \om_2 | },
\end{equation}
and $| A |^{\Lipg}_s := | A |^{\sup}_s + \g | A |^{\lip}_s$. 
\end{definition}

Such a norm is modeled on the behavior of matrices representing the multiplication
operator by a function.
Actually, given a function $ p \in H^s(\T^b) $, the multiplication operator $ h \mapsto p h $ is represented 
by the T\"oplitz matrix 
$ T_i^{i'} = p_{i - i'} $ and  $ |T|_s = \| p \|_s $. 
If $p = p(\om )$ is a Lipschitz family of functions, then 
$$
|T|_s^\Lipg = \| p \|_s^\Lipg\,.
$$
The $s$-norm  satisfies classical algebra and interpolation inequalities proved in \cite{BBM}. 

\begin{lemma} \label{prodest}
Let  $A = A(\om), B = B(\om)$ be matrices depending in a Lipschitz way on the parameter $\om \in 
\Omega_o \subset \R^\nu $. 
Then for all $s \geq s_0 > b/2 $ there are $ C(s) \geq C(s_0) \geq 1 $ such that
\begin{align}
|A B |_s^{\Lipg} 
& \leq  C(s) |A|_s^{\Lipg} |B|_s^{\Lipg} \, , \nonumber
\\
|A B|_{s}^{\Lipg} 
& \leq C(s) |A|_{s}^{\Lipg} |B|_{s_0}^{\Lipg} 
+ C(s_0) |A|_{s_0}^{\Lipg} |B|_{s}^{\Lipg} . \nonumber 
\end{align}
\end{lemma}
The $ s $-decay norm controls the Sobolev norm, namely 
$$
\| A h \|_s^\Lipg 
\leq C(s) \big(|A|_{s_0}^\Lipg \| h \|_s^\Lipg + |A|_{s}^\Lipg \| h \|_{s_0}^\Lipg \big).
$$
Let now $ b := \nu + 1 $. 
An important sub-algebra  is formed by the {\it T\"oplitz in time matrices} defined by
$$
 A^{(l_2, j_2)}_{(l_1, j_1)}  := A^{j_2}_{j_1}(l_1 - l_2 )\,  ,
$$
whose  decay norm \eqref{matrix decay norm} is
$$
|A|_s^2 =  \sum_{j \in \Z, l \in \Z^\nu} \big( \sup_{j_1 - j_2 = j} |A_{j_1}^{j_2}(l)| \big)^2  \langle l,j \rangle^{2 s} \, . 
$$
These matrices are identified with the $ \vphi $-dependent family
of operators
$$
A(\vphi) := \big( A_{j_1}^{j_2} (\vphi)\big)_{j_1, j_2 \in \Z} \, , \quad 
A_{j_1}^{j_2} (\vphi) := {\mathop\sum}_{l \in \Z^\nu} A_{j_1}^{j_2}(l) e^{\ii l \cdot \vphi}
$$
which act on functions of the $x$-variable as
$$
A(\vphi) : h(x) = \sum_{j \in \Z} h_j e^{\ii jx} \mapsto  
A(\vphi) h(x) = \sum_{j_1, j_2 \in \Z}  A_{j_1}^{j_2} (\vphi) h_{j_2} e^{\ii j_1 x} \, .  
$$
All the transformations that  we  construct in this paper are of this type (with $ j, j_1, j_2 \neq 0 $ because
they act on the phase space $ H^1_0 (\T_x) $).

\begin{definition}\label{operatore Hamiltoniano}
We say that   
\begin{enumerate}
\item
an operator $(A h)(\vphi, x) := A(\vphi) h(\vphi, x)$ 
is \emph{symplectic} if each $ A (\vphi ) $, $ \vphi \in \T^\nu $, is a symplectic map of the phase space (or of a symplectic subspace like $ H_S^\bot $)
\item 
the operator $\om \cdot \partial_{\vphi} - \partial_x G( \vphi)$ 
is \emph{Hamiltonian} if each $ G (\vphi) $, $  \vphi \in \T^\nu $, is symmetric;
\item 
an operator is \emph{real} if it maps real-valued functions into real-valued functions.
\end{enumerate}
\end{definition}

A Hamiltonian operator 
is transformed, under a symplectic map,  
into another Hamiltonian operator, see  \cite{BBM}-section 2.3.

We conclude this preliminary section recalling the following well known lemmata about composition of functions
(see, e.g., Appendix of \cite{BBM}).

\begin{lemma}[Composition] 
\label{lemma:composition of functions, Moser}
Assume $ f \in C^s (\T^d \times B_1)$, $B_1 := \{ y \in \R^m  : |y| \leq 1 \}$. Then 
$ \forall u \in H^{s}(\T^d, \R^m) $ such that $ \| u \|_{L^\infty} < 1  $, 
the composition operator $\tilde{f}(u)(x) := f(x, u(x))$ satisfies
$ \| \tilde f(u) \|_s \leq C \| f \|_{C^s} (\|u\|_{s} + 1)  $
where the constant $C $ depends on $ s ,d $.  
If $ f \in C^{s+2} $ and $ \| u + h \|_{L^\infty} < 1$, then for $k=0,1$
\[
\big\| \tilde f(u+h) - \sum_{i = 0}^k \frac{\tilde{f}^{(i)}(u)}{i !} [h^i] \big\|_s
\leq C \| f \|_{C^{s+ 2}} \, \| h \|_{L^\infty}^k ( \| h \|_{s} + \| h \|_{L^\infty} \| u \|_{s}).
\]
The statement also holds replacing $\| \ \|_s$ with the norms $| \ |_{s, \infty}$
of $W^{s,\infty}(\T^d)$.
\end{lemma}

\begin{lemma}[Change of variable]   \label{lemma:utile} 
Let $p \in W^{s,\infty} (\T^d,\R^d) $,  $ s \geq 1$, with 
$ \| p \|_{W^{1, \infty}}$ $ \leq 1/2 $.  
Then the function  $f(x) = x + p(x)$ is invertible, with inverse  $ f\inv(y)  = y + q(y)$ 
where $q \in W^{s,\infty}(\T^d,\R^d)$, 
and $ \| q \|_{W^{s, \infty}}  \leq C \| p \|_{ W^{s, \infty}} $.

If, moreover,  $p$ depends in a Lipschitz way on a parameter $\om \in \Omega \subset \R^\nu $, 
and $\| D_x p \|_ {L^\infty}  \leq 1/2 $ for all $\om$, 
then $ \| q \|_{W^{s, \infty}}^{{\rm Lip}(\gamma)} 
\leq  C  \| p \|_{W^{s+1, \infty}}^{{\rm Lip}(\gamma)} $.
The constant $C := C (d, s) $ is independent of $\g$.

If $u \in H^s (\T^d,\C)$, then $ (u\circ f)(x) := u(x+p(x))$ satisfies 
\begin{align*}
\| u \circ f \|_s 
& \leq  C (\|u\|_s + \| p \|_{W^{s, \infty}} \|u\|_1),
\\
\| u \circ f - u \|_s 
& \leq C ( \| p \|_{L^\infty} \| u \|_{s + 1}  + \| p \|_{W^{s, \infty}} \| u \|_{2} ) ,
\\
\| u \circ f \|_{s}^{{{\rm Lip}(\gamma)}}
& \leq C  \, 
\big( \| u \|_{s+1}^{{{\rm Lip}(\gamma)}} + \| p \|_{W^{s, \infty}}^{{\rm Lip}(\gamma)}\| u \|_2^{{\rm Lip}(\gamma)} \big). \nonumber
\end{align*}
The function  $u \circ f^{-1} $  satisfies the same bounds.
\end{lemma}

\section{Weak Birkhoff normal form}\label{sec:WBNF}

In this section it is convenient to analize the mKdV equation  in the Fourier representation 
\begin{equation}\label{Fourier}
u(x) = {\mathop \sum}_{j \in \Z \setminus \{0\} } u_j e^{\ii j x}, \quad 
u(x) \longleftrightarrow u := (u_j)_{j \in \Z \setminus \{0\} }, \quad 
u_{-j} = \overline{u}_j,
\end{equation}
where the Fourier indices are nonzero integers $j$,  by the definition \eqref{def phase space} of the phase space, and $u_{-j} = \overline{u}_j$ because $u(x)$ is real-valued. 
The symplectic structure \eqref{2form KdV} writes
\begin{equation}\label{2form0}
\Omega = \frac12 \sum_{j \neq 0}  \frac{1}{\ii j} du_j \wedge d u_{-j}, 
\quad 
\Omega ( u, v ) = \sum_{j \neq 0} \frac{1}{\ii j} u_j v_{-j},
\end{equation}
the Hamiltonian vector field $X_H$ in \eqref{KdV:HS} and the Poisson bracket $\{ F, G \}$ in \eqref{Poisson bracket} are
respectively 
\begin{equation}\label{PoissonBr}
[X_H (u)]_j = \ii j \partial_{u_{-j}} H(u), 
\ \ 
\{ F, G \}(u) = - \sum_{j \neq 0} \ii j (\partial_{u_{-j}} F) (u)  (\partial_{u_j} G) (u).  
\end{equation}
We shall sometimes identify $ v \equiv (v_j)_{j \in S } $ and  $ z \equiv (z_j)_{j \in S^c } $. 

\smallskip

The Hamiltonian of the perturbed cubic mKdV equation \eqref{eq:1} is 
$ H = H_2 + H_4 + H_{\geq 5} $ (see \eqref{Ham in intro}) where 
\begin{equation} \label{H iniziale KdV}
H_2(u) := \int_{\T} \frac{u_x^{2}}{2} dx, \quad  
H_4(u) :=  - \varsigma \int_\T \frac{u^4}{4} dx, \quad  
H_{\geq 5}(u) := \int_\T f(x, u,u_x) dx,
\end{equation}
$\varsigma = \pm1$ and $f$ satisfies \eqref{order5}. 
According to the splitting \eqref{u = v + z} $ u = v + z $, 
where $ v \in H_S $ and $ z \in H_S^\bot $, we have 
$H_2(u) = H_2(v) + H_2(z)$ and 
\[ 
H_4(u) = 
- \frac{\varsigma}{4} \int_{\T} v^4 \, dx 
- \varsigma \int_{\T} v^3 z \, dx 
- \frac{3\varsigma}{2} \int_{\T} v^2 z^2 \, dx 
- \varsigma \int_{\T} v z^3 \, dx 
- \frac{\varsigma}{4} \int_{\T} z^4 \, dx.
\] 
For a finite-dimensional space
\begin{equation} \label{def E finito}
E := E_{C} :=  \mathrm{span} \{ e^{\ii jx} :  0 < |j| \leq C \}, \quad C > 0,
\end{equation}
let $\Pi_E $ denote the corresponding $ L^2 $-projector on $E$.

In the next proposition we  construct a symplectic map $ \Phi_B $ 
such that the transformed Hamiltonian $\mH := H \circ \Phi_B$ 
possesses the invariant subspace $ H_S $  defined in \eqref{splitting S-S-bot}, 
and its dynamics on $ H_S $ is integrable and non-isocronous.  
To this purpose we have to eliminate the term $ \int v^3 z \, dx $ 
(which is linear in $ z $) and to normalize the term $ \int v^4 \, dx $ 
(which is independent of $ z $) in the quartic component of the Hamiltonian.

\begin{proposition}[Weak Birkhoff normal form]  
\label{prop:weak BNF}  
There exists an analytic invertible symplectic transformation 
of the phase space $ \Phi_B : H^1_0 (\T_x) \to H^1_0 (\T_x) $ 
of the form 
\begin{equation} \label{finito finito}
\Phi_B(u) = u + \Psi(u), 
\quad
\Psi(u) = \Pi_E \Psi(\Pi_E u),
\end{equation}
where $ E $ is a finite-dimensional space as in \eqref{def E finito}, such that the transformed Hamiltonian is
\begin{equation} \label{widetilde cal H}
{\cal H} := H \circ \Phi_B  
= H_2 + \mH_4 + {\cal H}_{\geq 5} \,,
\end{equation}
where $H_2$ is defined in \eqref{H iniziale KdV}, 
\begin{equation} \label{H3tilde} 
\begin{aligned}
\mH_4 & := \frac{3\varsigma}{4} \Big( \sum_{j \in S} |u_j|^4 - \sum_{j,j' \in S} |u_j|^2 |u_{j'}|^2 \Big) 
- \frac{3\varsigma}{2} \int_\T v^2 z^2 \, dx  \\
& \quad \ \ 
- \varsigma \int_\T v z^3 \, dx - \frac{\varsigma}{4} \int_\T z^4 \, dx,
\end{aligned}
\end{equation} 
and ${\cal H}_{\geq 5}$ collects all the terms of order at least five in $(v,z)$. 
\end{proposition}

\begin{proof} 
In Fourier coordinates \eqref{Fourier} we have (see \eqref{H iniziale KdV})
\begin{equation}\label{H3 Fourier}
H_2(u) = \frac{1}{2} \sum_{j \neq 0} j^2 |u_j|^2, \quad 
H_4(u) = - \frac{\varsigma}{4} \, \sum_{j_1 + j_2 + j_3 + j_4 = 0} u_{j_1} u_{j_2} u_{j_3} u_{j_4} \,.
\end{equation}
We look for a symplectic transformation $\Phi$ of the phase space which eliminates or normalizes
the monomials $ u_{j_1} u_{j_2} u_{j_3} u_{j_4} $ of $ H_4 $ with at most one index outside $ S $. 
By the relation $ j_1 + j_2 + j_3 + j_4 = 0 $, they are {\it finitely} many.
Thus, we look for a map $\Phi := (\Phi_{F}^t)_{|t=1}$ which is the time $ 1$-flow map 
of an auxiliary quartic Hamiltonian
$$
F(u) := \sum_{j_1 + j_2 + j_3 + j_4 = 0} 
F_{j_1 j_2 j_3 j_4} u_{j_1} u_{j_2} u_{j_3} u_{j_4} \,.
$$
The transformed Hamiltonian is
\begin{equation}\label{callH}
\mH :=  H \circ \Phi 
= H_2 + \mH_4 +  \mH_{\geq 5}, \quad 
\mH_4 = \{ H_2, F \} + H_4,
\end{equation}
where $ \mH_{\geq 5} $ collects all the terms in $\mH$ of order at least five. 
By \eqref{H3 Fourier} and \eqref{PoissonBr} we calculate 
$$
\mH_4 = \sum_{j_1 + j_2 + j_3 + j_4 = 0} 
\Big\{ - \frac{\varsigma}{4} \, - \ii (j_1^3 + j_2^3 + j_3^3 + j_4^3) F_{j_1 j_2 j_3 j_4} \Big\} \, 
u_{j_1} u_{j_2} u_{j_3} u_{j_4}\,.
$$
In order to eliminate or normalize only the monomials with at most one index outside $ S $, 
we choose
\begin{equation}\label{F4}
F_{j_1 j_2 j_3 j_4} := \begin{cases}
 \dfrac{\ii \varsigma}{4 (j_1^3 + j_2^3 + j_3^3 + j_4^3)} & \text{if} \,\,(j_1,j_2,j_3,j_4) \in {\cal A}\,, 
\\
0 & \text{otherwise},
\end{cases}
\end{equation}
where 
\begin{multline*}
{\cal A} := \big\{ (j_1 , j_2 , j_3, j_4) \in (\Z \setminus \{ 0 \})^4 : \ 
j_1 + j_2 + j_3 + j_4= 0, \quad  
j_1^{3} + j_2^{3} + j_3^{3} + j_4^3 \neq 0, \\ 
\text{and at least three among} \  j_1 , j_2 , j_3, j_4 \ \text{belong to } S \big\}.
\end{multline*}
We recall the following elementary identity (Lemma 13.4 in \cite{KaP}).

\begin{lemma}  \label{lemma:interi} 
Let $j_1, j_2, j_3, j_4 \in \Z $ such that $ j_1 + j_2 + j_3 + j_4 = 0 $. 
Then 
$$
j_1^3 + j_2^3 + j_3^3 + j_4^3 = -3 (j_1 + j_2) (j_1 + j_3) (j_2 + j_3).
$$
\end{lemma}

By definition \eqref{F4}, $\mH_4$ does not contain any monomial $u_{j_1} u_{j_2} u_{j_3} u_{j_4}$
with three indices in $S$ and one outside, 
because there exist no integers $ j_1, j_2 , j_3 \in S$, $ j_4 \in S^c $ satisfying 
$ j_1 + j_2 + j_3 + j_4 = 0 $ and $ j_1^3 + j_2^3 + j_3^3 + j_4^3 = 0 $, 
by Lemma \ref{lemma:interi} and the fact that $ S $ is symmetric. 

By construction, the quartic monomials with at least two indices outside $S$ are not changed by $\Phi$. 
Also, by construction, the monomials $u_{j_1} u_{j_2} u_{j_3} u_{j_4}$ in $\mH_4$ 
with all integers in $S$ are those for which $j_1 + j_2 + j_3 + j_4 = 0$ and 
$j_1^3 + j_2^3 + j_3^4 + j_4 ^3 = 0$.  
By Lemma \ref{lemma:interi}, we split 
\[
\sum_{\begin{subarray}{c}
j_1, j_2, j_3, j_4  \in S \\
j_1 + j_2 + j_3 + j_4 = 0 \\
j_1^3 + j_2^3 + j_3^3 + j_4^3 = 0 \end{subarray}} 
u_{j_1} u_{j_2} u_{j_3} u_{j_4} 
= A_1 + A_2 + A_3
\]
where $A_1$ is given by the sum over $j_1, j_2, j_3, j_4  \in S$, $j_1 + j_2 + j_3 + j_4 = 0$ 
with the restriction $j_1 + j_2 = 0$, 
$A_2$ with the restriction $j_1 + j_2 \neq  0$ and $j_1 + j_3 = 0$, 
and $A_3$ with the restriction $j_1 + j_2 \neq  0$, $j_1 + j_3 \neq 0$ and $j_2 + j_3 = 0$. 
We get 
\begin{align*}
A_2 & = \sum_{\begin{subarray}{c} j, j' \in S \\ j' \neq -j \end{subarray}} |u_j|^2 |u_{j'}|^2 
= \sum_{j, j' \in S} |u_j|^2 |u_{j'}|^2  - \sum_{j \in S} |u_j|^4 \,, 
\qquad A_1 = \sum_{j, j' \in S} |u_j|^2 |u_{j'}|^2 \,,  
\\
A_3 & = \sum_{\begin{subarray}{c} j, j' \in S \\ j' \neq \pm j \end{subarray}} |u_j|^2 |u_{j'}|^2 
= \sum_{j, j' \in S} |u_j|^2 |u_{j'}|^2  - 2 \sum_{j \in S} |u_j|^4 \,,
\end{align*}
whence \eqref{H3tilde} follows. 
\end{proof}

\begin{remark} \label{rem:lm M2-1}
In the Birkhoff normal form 
for the Hamiltonian $ K = H + \lm M^2 $ defined in \eqref{eq:Kv}, 
 three additional terms appear in  \eqref{H3tilde}, which are
\[
\lm \sum_{j, j' \in S} |u_j|^2 |u_{j'}|^2 + 2 \lm M(v) M(z) + \lm M^2(z).
\]
Then in \eqref{H3tilde} the sum $(\lm - \frac{3\varsigma}{4}) \sum_{j, j' \in S} |u_j|^2 |u_{j'}|^2$ vanishes if 
we choose $\lm := 3 \varsigma /4$.
\end{remark}

\section{Action-angle variables}\label{sec:4}

We introduce action-angle variables on the tangential directions by the change of coordinates
\begin{equation}\label{coordinate azione angolo}
u_j := \sqrt{\tilde\xi_j + |j| \tilde y_j} \, e^{\ii \tilde \theta_j} \quad \text{for} \  j \in S \,; \qquad 
u_j := \tilde z_j \quad \text{for} \  j \in S^c \,, 
\end{equation}
where (recall that $ u_{-j} = {\overline u}_j $)
\begin{equation}\label{simmeS}
\tilde \xi_{-j} = \tilde \xi_j  \, , \quad 
\tilde \xi_j > 0 \, , \quad 
\tilde y_{-j} = \tilde y_j \, , \quad 
\tilde \theta_{-j} = - \tilde \theta_j \, , \quad 
\tilde \teta_j, \, \tilde y_j \in \R \, , \quad  
\forall j \in S \,.
\end{equation}
To simplify notation, 
for the tangential sites $ S^+ := \{ {\bar \jmath_1}, \ldots, {\bar \jmath_\nu} \} $ we also denote 
$\tilde \teta_{\bar \jmath_i} := \tilde \teta_i $, $ \tilde y_{\bar \jmath_i} := \tilde y_i $, 
$ \tilde \xi_{\bar \jmath_i} := \tilde \xi_i $,  $ i =1, \ldots \, \nu $.

The symplectic 2-form $ \Omega $ in \eqref{2form0} (i.e.  \eqref{2form KdV}) becomes 
\begin{equation}\label{2form}
{\cal W} := \sum_{i=1}^\nu  d \tilde \theta_i \wedge d \tilde y_i  
+ \frac12 \sum_{j \in S^c \setminus \{ 0 \} } \frac{1}{\ii j} \, d \tilde z_j \wedge d \tilde z_{-j} = \big( \sum_{i=1}^\nu  d \tilde \theta_i \wedge d \tilde y_i \big)  \oplus \Om_{S^\bot} = d \Lambda 
\end{equation}
where $ \Om_{S^\bot} $ denotes the restriction of $ \Om $ to $ H_S^\bot $ (see \eqref{splitting S-S-bot}) and 
$ \Lambda $ is the Liouville $ 1 $-form on $ \T^\nu \times \R^\nu \times H_S^\bot $ defined by
$ \Lambda_{(\tilde \theta, \tilde y, \tilde z)} : \R^\nu \times \R^\nu \times H_S^\bot \to \R $, 
\begin{equation}\label{Lambda 1 form}
\Lambda_{(\tilde \theta, \tilde y, \tilde z)}[\widehat \theta, \widehat y, \widehat z] := 
- \tilde y \cdot \widehat \theta + \frac12 ( \partial_x^{-1} \tilde z, \widehat z )_{L^2 (\T)} \, .   
\end{equation}
We rescale the ``unperturbed actions'' $ \xi $ and the variables $\tilde \theta, \tilde y, \tilde z$ as
\begin{equation}\label{rescaling kdv quadratica}
\tilde \xi = \e^2 \xi \, , \quad  
\tilde y = \e^{2b} y \, , \quad 
\tilde z = \e^b z  \,, \quad b > 1.
\end{equation}
The symplectic $ 2 $-form in \eqref{2form} transforms into $ \e^{2b} {\cal W } $. 
Hence the Hamiltonian system generated by $ {\cal H} $ in \eqref{widetilde cal H} 
transforms into the new Hamiltonian system 
\begin{equation}  \label{def H eps}
\begin{cases}
\dot \theta = \partial_y H_{\e} (\theta, y, z), \\ 
\dot y = -  \partial_\teta H_{\e} (\theta, y, z), \\ 
\dot z =  \partial_x \nabla_z H_{\e} (\theta, y, z),
\end{cases} \qquad 
H_{\e} := \e^{-2b} \mH \circ A_\e,
\end{equation}
where 
\begin{equation}  \label{def A eps}
A_\e (\theta, y, z) := \e v_\e(\theta, y) + \e^b z, \ \ \    
v_\e(\theta,y) := \sum_{j \in S} \sqrt{\xi_j + \e^{2(b-1)} |j| y_j} \, e^{\ii \theta_j} e^{\ii j x}. \end{equation}
We still denote by 
$$ 
X_{H_\e} = (\partial_y H_\e, - \partial_\teta H_\e, \pa_x \nabla_z H_\e) 
$$ 
the Hamiltonian vector field in the variables 
$ (\teta, y, z ) \in \T^\nu \times \R^\nu \times H_S^\bot $.

We now write explicitly the Hamiltonian $ H_{\e} (\theta, y, z) $ defined in \eqref{def H eps}. Recall the
expression of $ {\cal H } $  given in \eqref{widetilde cal H}.
The quadratic Hamiltonian $ H_2 $ in \eqref{H iniziale KdV} transforms into
\be\label{shape H2}
\e^{-2b} H_2 \circ A_\e = const + 
{\mathop \sum}_{j \in S^+} j^3 y_j + \frac{1}{2} \int_{\T} z_x^2 \, dx \, , 
\ee
and, by \eqref{H3tilde},   \eqref{widetilde cal H} 
we get (writing, in short, $ v_\e := v_\e (\theta, y) $)  
\begin{align} 
H_{\e} (\theta, y, z) 
& =  e(\xi) + \a (\xi) \cdot y + \frac12 \int_\T z_x^2 \, dx 
- \frac{3 \varsigma}{2}\, \e^2 \int_\T v_\e^2 z^2 \, dx 
\notag \\ & \quad
+ 3 \varsigma \e^{2b} \Big( \frac12 \sum_{j \in S^+} j^2 y_j^2 - \sum_{j,j' \in S^+} j y_j j' y_{j'} \Big) 
- \varsigma \e^{1+b} \int_\T v_\e z^3 \, dx 
\notag \\ & \quad
- \frac{\varsigma}{4}\, \e^{2b} \int_\T z^4 \, dx 
+ \e^{-2b} {\cal H}_{\geq 5} (\e v_\e(\theta,y) + \e^b z )
\label{formaHep}
\end{align}
where $e(\xi)$ is a constant, and $\a(\xi) \in \R^\nu$ is the vector of components  
$$
\a_i(\xi) := \bar \jmath_i^3 + 3 \varsigma \e^2 [ \xi_i - 2 (\xi_1 + \ldots + \xi_\nu) ] \bar \jmath_i \, , \quad 
i = 1, \ldots, \nu \, . 
$$
This is  the ``frequency-to-amplitude'' map which describes, at the main order,  
how the tangential frequencies are shifted by  the amplitudes $ \xi := ( \xi_1, \ldots, \xi_\nu ) $. 
It can be written in compact form as 
\begin{equation} \label{mappa freq amp}
\a(\xi) := \bar \om + \e^2 {\mathbb A} \xi  \, , \quad 
{\mathbb A} := 3\varsigma D_S (I - 2 U),
\end{equation} 
where $ \bar\om := (\bar\jmath_1^3, \ldots, \bar\jmath_\nu^3) \in \N^\nu  $ (see \eqref{bar omega})
is the vector of the unperturbed linear frequencies of oscillations on the tangential  sites, 
$ D_S $ is the diagonal matrix 
$$
D_S := \mathrm{diag}(\bar \jmath_1, \ldots, \bar \jmath_\nu) \in {\rm Mat}(\nu \times \nu) \, , 
$$
$I$ is the $\nu \times \nu$ identity matrix, 
and $U$ is the $\nu \times \nu$ matrix with all entries equal to 1. 
The matrix $\mathbb A$ is often called the ``twist'' matrix . It turns out to be invertible.
Indeed, since $U^2 = \nu U$, one has 
$(I - 2 U)( I - \frac{2}{2\nu-1}\, U ) = I$, and therefore 
\begin{equation} \label{AA -1}
\mathbb A^{-1} = \frac{1}{3\varsigma}\, \Big( I - \frac{2}{2\nu-1}\, U \Big) D_S^{-1} \,.
\end{equation}
With this notation, one can also write 
\be\label{twist-pezzo}
\frac12 \sum_{j \in S^+} j^2 y_j^2 - \sum_{j,j' \in S^+} j y_j j' y_{j'} 
= \frac12 (I-2U) (D_S y) \cdot (D_S y). 
\ee
\begin{remark} \label{rem:lm M2-2}
By remark \ref{rem:lm M2-1}, 
for the Hamiltonian $ K = H + \lm M^2 $, $\lm := 3 \varsigma /4$, defined in \eqref{eq:Kv}
the twist matrix in the  frequency-amplitude relation \eqref{mappa freq amp} 
becomes $\mathbb A = 3 \varsigma D_S$, which is diagonal.
\end{remark}

We write the Hamiltonian in \eqref{formaHep} (eliminating the constant $e(\xi)$
which is  irrelevant for the dynamics) 
as $H_{\e} = {\cal N} +  P$, where 
\begin{equation}  \label{Hamiltoniana Heps KdV}
\begin{aligned}
{\cal N}(\theta, y, z) 
& = \a (\xi) \cdot y + \frac12 \big(N(\theta) z , z \big)_{L^2(\T)} \,,
\\
\big(N(\theta) z, z \big)_{L^2(\T)}  
& := \int_\T z_x^2 dx - 3\varsigma \e^2 \int_\T v_\e^2(\theta,0) z^2 \, dx \,,
\end{aligned}
\end{equation}
describes the linear dynamics, 
and $ P := H_{\e} - {\cal N} $, namely 
\begin{align} 
& P := \frac{3\varsigma}{2}\, \e^{2b} (I-2U) (D_S y) \cdot (D_S y)
- \frac{3 \varsigma}{2}\, \e^2 \int_\T [v_\e^2(\theta,y) - v_\e^2(\theta,0)] z^2 \, dx 
\notag \\ &  
- \varsigma \e^{1+b} \int_\T v_\e(\theta,y) z^3 \, dx 
- \frac{\varsigma}{4}\, \e^{2b} \int_\T z^4 \, dx 
+ \e^{-2b} {\cal H}_{\geq 5} (\e v_\e(\theta,y) + \e^b z ) \, , 
\label{def P1}
\end{align}
collects the nonlinear perturbative effects.

\section{The nonlinear functional setting}\label{sec:functional}

We look for an embedded invariant torus 
\begin{equation} \label{embedded torus i}
i : \T^\nu \to \T^\nu \times \R^\nu \times H_S^\bot, \quad  
\vphi \mapsto i (\vphi) := ( \theta (\vphi), y (\vphi), z (\vphi)) 
\end{equation}
of the Hamiltonian vector field $ X_{H_\e} $ 
filled by quasi-periodic solutions with diophantine frequency $ \om \in \R^\nu $, that we regard as  independent parameters. 
We require that $ \om $ belongs to the set 
\begin{equation}\label{Omega epsilon}
\Omega_\e := \a ( [1,2]^\nu ) 
= \{ \a(\xi) : \xi \in [1,2]^\nu \} 
\end{equation}
where $ \a $ is the affine diffeomorphism \eqref{mappa freq amp}. 
Since any $ \omega \in \Omega_\e $ is $ \e^2 $-close to the integer vector $ \bar \om  \in \N^\nu $ 
(see \eqref{mappa freq amp},  \eqref{bar omega}), we require that the constant $\g$ in the diophantine inequality 
\be\label{omdio}
 |\omega \cdot l | \geq \gamma \langle l \rangle^{-\tau} \, ,  \  
 \forall l \in \Z^\nu \setminus \{0\} \, , 
\quad \text{satisfies} \  \  \gamma = \e^{2+a} \quad \text{for some} \ a > 0 \, .  
\ee
Note that the definition of $\g$ in \eqref{omdio} is slightly stronger than the minimal condition, which is $ \g \leq c \e^2 $ with $ c $ small enough. 
In addition to \eqref{omdio} we shall also require that $ \om $ satisfies the first and second order Melnikov-non-resonance conditions \eqref{Omegainfty}. 

We {\it fix} the amplitude $\xi$ as a function of $\om $ and $ \e$, as 
\begin{equation} \label{linkxiomega}
\xi := \e^{-2} \mathbb{A}^{-1} [\om - \bar \om]\,,
\end{equation}
so that $\a(\xi) = \om$ (see \eqref{mappa freq amp}). 

\smallskip

Now  we look for an embedded invariant torus of the modified Hamiltonian vector field 
$ X_{H_{\e, \zeta}} = X_{H_\e} + (0, \zeta, 0) $, $ \zeta \in \R^\nu $, 
which is generated by the Hamiltonian 
\begin{equation}\label{hamiltoniana modificata}
H_{\e, \zeta} (\teta, y, z) :=  H_\e(\teta, y, z) + \zeta \cdot \theta\,,\quad \zeta \in \R^\nu \,.
\end{equation}
Note that the vector field $ X_{H_{\e, \zeta}} $ is  periodic in $\theta $
(unlike the Hamiltonian $ H_{\e, \zeta} $). 
We introduce $\zeta $ in order to adjust the average 
in the second equation of the linearized system \eqref{operatore inverso approssimato}, see 
\eqref{fisso valore di widehat zeta}.  The vector $ \zeta $ has however no dynamical consequences. Indeed
it turns out that an invariant torus for the Hamiltonian vector field 
$ X_{H_{\e, \zeta}} $ is actually 
invariant  for $ X_{H_\e} $ itself, see Lemma \ref{zeta = 0}. 
Hence we look for zeros of the nonlinear operator
\begin{align} \label{operatorF}
{\cal F} (i, \zeta) 
& :=   {\cal F} (i, \zeta, \omega, \e ) 
:= {\cal D}_\om i (\vphi) - X_{H_\e} (i(\vphi)) + (0, \zeta, 0 )  
\\ & 
= \begin{pmatrix} 
{\cal D}_\om \theta (\vphi) - \partial_y H_\e ( i(\vphi)  )   \\
{\cal D}_\om y (\vphi)  +  \partial_\teta H_\e ( i(\vphi)  ) + \zeta  \\
{\cal D}_\om z (\vphi) -  \partial_x \nabla_z H_\e ( i(\vphi)) 
\end{pmatrix}
\notag \\ & 
= \begin{pmatrix} 
{\cal D}_\om \Theta (\vphi) - \partial_y P (i(\vphi) )   \\
{\cal D}_\om  y (\vphi) 
+ \frac12 \partial_\teta ( N(\theta (\vphi)) z(\vphi), z(\vphi) )_{L^2(\T)} 	
+ \partial_\teta P ( i(\vphi) ) + \zeta  \\
{\cal D}_\om  z (\vphi) - \partial_x N ( \theta  (\vphi )) z (\vphi)  
- \partial_x \nabla_z P  ( i(\vphi) ) 
\end{pmatrix} 
\notag
\end{align}
where $ \Theta(\ph) := \teta (\vphi) - \vphi $ is $ (2 \pi)^\nu $-periodic 
and we use (here and everywhere in the paper) the short notation 
\begin{equation}\label{Domega}
{\cal D}_\om := \om \cdot \partial_\vphi \, . 
\end{equation}
The Sobolev norm of the periodic component of the embedded torus 
\begin{equation}\label{componente periodica}
{\mathfrak I}(\vphi)  := i (\vphi) - (\vphi,0,0) := ( {\Theta} (\ph), y(\ph), z(\ph))\,, \quad \Theta(\ph) := \teta (\vphi) - \vphi \, , 
\end{equation}
is $\| {\mathfrak I}  \|_s := \| \Theta \|_{H^s_\vphi} +  \| y  \|_{H^s_\vphi} +  \| z \|_s $
where $ \| z \|_s := \| z \|_{H^s_{\vphi,x}}  $ is defined in \eqref{Sobolev coppia}.
We link the rescaling \eqref{rescaling kdv quadratica}
with the diophantine constant $ \g = \e^{2+a} $ by choosing
\be\label{link gamma b}
\gamma = \e^{2+a} = \e^{2b}\,, \quad 
b = 1 + ( a \slash 2 ) \,, \quad 
a \in (0, 1/6).
\ee
Other choices are possible, see Remark \ref{comm3}.

\begin{theorem}\label{main theorem}
Let the tangential sites $ S $ in \eqref{tang sites} 
satisfy \eqref{scelta siti}. 
For all $ \e \in (0, \e_0 ) $, where $ \e_0 $ is small enough,   
there exist a constant $C>0$ and a Cantor-like set $ {\cal C}_\e \subset \Omega_\e $,
with  asympotically full measure as $ \e \to 0 $, namely
\begin{equation}\label{stima in misura main theorem}
\lim_{\e\to 0} \, \frac{|{\cal C}_\e|}{|\Omega_\e|} = 1 \, , 
\end{equation}
such that, for all $ \omega \in {\cal C}_\e $, 
there exists a solution $ i_\infty (\vphi) := i_\infty (\omega, \e)(\vphi) $ 
of the equation $\mF(i_\infty, 0, \om, \e) = 0$ 
(the nonlinear operator $\mF(i,\zeta,\om,\e)$ is defined in \eqref{operatorF}).
Hence the embedded torus $ \vphi \mapsto i_\infty (\vphi) $ is invariant for the Hamiltonian vector field $ X_{H_\e} $, and it is filled by quasi-periodic solutions with frequency $ \om $.
The torus $i_\infty$ satisfies 
\be\label{stima toro finale}
\|  i_\infty (\vphi) - (\vphi,0,0) \|_{s_0 + \mu}^\Lipg \leq C \e^{5-2b} \g^{-1}
= C \e^{1-2a}
\ee
for some $ \mu := \mu (\nu) > 0 $. 
Moreover, the torus $ i_\infty $ is linearly stable. 
\end{theorem}

Theorem \ref{main theorem} is  proved in sections \ref{costruzione dell'inverso approssimato}-\ref{sec:NM}. 
It implies Theorem \ref{thm:mKdV} where the $ \xi_j $ in \eqref{solution u} are the 
components of the vector $\mathbb{A}^{-1}[\om - \bar\om]$. 
By \eqref{stima toro finale}, going back to the variables before the rescaling \eqref{rescaling kdv quadratica}, we get 
$ \tilde \Theta_\infty = O( \e^{5-4b}) $, 
$ \tilde y_\infty = O( \e^{5-2b} ) $, 
$ \tilde z_\infty = O( \e^{5-3b} ) $. 

\begin{remark} \label{comm3}
The way to link the amplitude-rescaling \eqref{rescaling kdv quadratica} with the diophantine constant 
$ \g = \e^{2+a} $ in \eqref{omdio} is not unique. 
 
The choice $ \e^{2b} < \g   $ (i.e. ``$ b > 1 $ large'')  reduces to study 
the Hamiltonian $ H_\e $ in \eqref{formaHep} as a
perturbation of an isochronous system 
(as in \cite{Ku}, \cite{k1}, \cite{Po3}).  
We can take  $ b = 4 / 3 $ in order to minimize the size of the perturbation
$ P = O( \e^{7/3}) $,  
estimating uniformly all the terms in the last two lines of \eqref{formaHep}.
As a counterpart 
we have  to regard  in \eqref{formaHep} the constants 
$ \a := \a (\xi ) \in \R^\nu $ (or $ \xi $ in   \eqref{def A eps}) as independent variables. 
This is the perspective described for example in  \cite{BB13}.  
Then the Nash-Moser scheme produces iteratively a sequence of  
$ \xi_n = \xi_n (\om) $ and embeddings  
$ \vphi \mapsto i_n (\vphi) := (\theta_n (\vphi), y_n (\vphi), z_n (\vphi) )$
at the same time. 

The case  $ \e^{2b} > \g $ (i.e. ``$ b \geq 1 $ small''), in particular if $ b = 1 $, 
reduces to study  the Hamiltonian $ H_\e $ in \eqref{formaHep} as a
perturbation of a non-isochronous system {\`a la} Arnold-Kolmogorov (note that 
the quadratic Hamiltonian  in \eqref{twist-pezzo}
satisfies the usual Kolmorogov non-degeneracy condition).
In this case,  the constant $ \xi_j $ in  \eqref{def A eps}  and the average of $ |j| y_j (\vphi) $ have the same size and therefore the same role.
Then we may consider $ \xi_j $ as fixed, and tune the average of the action component $  y_j (\vphi) $
in order to solve the linear equation \eqref{equazione psi hat}, which corresponds to the angle component. 
We use the invertible (averaged) ``twist''-matrix \eqref{media-twist} to impose that 
the right hand side in \eqref{equazione psi hat} has zero average. 

The intermediate case $\e^{2b} = \g $, adopted in this paper (as well as in \cite{bbm}), 
has the advantage to avoid the introduction of the 
$ \xi(\om) $ as an independent variable,   
but it also enables to estimate uniformly the sizes of  the components of 
$  (\Theta (\vphi) , y (\vphi) , z (\vphi) ) $ with no distinctions.   
\end{remark}

Now we prove tame estimates for the composition operator induced by the Hamiltonian vector fields $ X_{\cal N} $ and  $ X_P $ in \eqref{operatorF}, which are used in the next sections. 
Since the functions $ y \mapsto \sqrt{\xi + \e^{2(b - 1)}|j| y} $, $\theta \mapsto e^{\ii \theta}$ 
are analytic  for $\e$ small enough and $|y| \leq C$,  the composition 
Lemma \ref{lemma:composition of functions, Moser} implies that, 
for all  $ \Theta, y \in H^s(\T^\nu, \R^\nu )$ with $\| \Theta \|_{s_0}$, $\| y \|_{s_0} \leq 1$, setting $\theta(\ph) := \ph + \Theta (\ph)$, one has the tame estimate
$$ 
\|  v_\e (\theta (\vphi) ,y(\vphi) ) \|_s \leq_s 1 + \| \Theta  \|_s + \| y \|_s  \, . 
$$ 
Hence the map $ A_\e $ in \eqref{def A eps} satisfies, for all 
$   \| {\mathfrak I} \|_{s_0}^\Lipg  \leq 1 $ (see \eqref{componente periodica})
\be \label{stima Aep}
 \| A_\e (\theta (\vphi),y(\vphi),z(\vphi)) \|_s^{\Lipg} \leq_s \e (1 + \| {\mathfrak I} \|_s^\Lipg) \, . 
\ee 
In the following lemma we collect tame estimates for the Hamiltonian vector fields 
$ X_{\cal N} $, $ X_P $,  $ X_{H_\e} $ (see \eqref{Hamiltoniana Heps KdV}, \eqref{def P1})
whose proof is a direct application of classical tame product and composition estimates.

\begin{lemma}\label{lemma quantitativo forma normale}
Let $ \fracchi(\ph) $ in \eqref{componente periodica} satisfy
$ \| {\mathfrak I} \|_{s_0 + 3}^\Lipg \leq C \e^{5-2b} \gamma^{-1} = C \e^{5-4b}$. 
Then, writing in short $\| \ \|_s$ to indicate $\| \ \|_s^\Lipg$, one has
\begin{alignat*}{2}
\| \partial_y P(i) \|_s 
& \leq_s \e^3 + \e^{2b} \| {\mathfrak I}\|_{s+3} 
&
\| \partial_\theta P(i) \|_s
& \leq_s \e^{5-2b} (1 + \| {\mathfrak I} \|_{s+3}) 
\\ 
\| \nabla_z P(i) \|_s
& \leq_s \e^{4-b} + \e^{6-3b} \| {\mathfrak I} \|_{s+3} 
&
\| X_P(i)\|_s 
& \leq_s \e^{5-2b} + \e^{2b} \| {\mathfrak I}\|_{s+3}  
\\
\| \partial_{\theta} \partial_y P(i)\|_s 
& \leq_s \e^3 + \e^{5-2b} \| {\mathfrak I}\|_{s+3} 
& \qquad 
\| \partial_y \nabla_z P(i)\|_s  
& \leq_s \e^{2+b} +  \e^{2b} \| {\mathfrak I} \|_{s+3}   
\\ 
& & 
\| \partial_{yy} P(i) - \e^{2b} \mathbb{A} D_S \|_s 
& \leq_s \e^{1+2b} + \e^3 \| \fracchi \|_{s+3} 
\end{alignat*}
($\mathbb{A}, D_S$ are defined in \eqref{mappa freq amp})
and, for all  $ \widehat \imath :=  (\widehat \Theta, \widehat y, \widehat z) $, 
\begin{align}
\label{D yii}
& \| \partial_y d_{i} X_P(i)[\widehat \imath \,] \|_s 
\leq_s \e^{2b} \big( \| \widehat \imath \,\|_{s + 3} 
+ \| {\mathfrak I}\|_{s + 3} \| \widehat \imath \,\|_{s_0 + 3}\big)  
\\
\label{tame commutatori} 
& \| d_i X_{H_\e}(i) [\widehat \imath \, ] + (0,0, \partial_{xxx} \hat z)\|_s  
\leq_s \e^2 \big( \| \widehat \imath \,\|_{s + 3} 
+ \|{\mathfrak I} \|_{s + 3} \| \widehat \imath \,\|_{s_0 + 3} \big) 
\\
\label{parte quadratica da P}
& \| d_i^2 X_{H_\e}(i) [\widehat \imath, \widehat \imath \,]\|_s 
\leq_s \e^2 \big( \| \widehat \imath \,\|_{s + 3} \| \widehat \imath \,\|_{s_0 + 3} 
+ \| {\mathfrak I}\|_{s + 3} \| \widehat \imath \,\|_{s_0 + 3}^2 \big) \, .
\end{align}
\end{lemma}

In the sequel we also use that, by the diophantine condition \eqref{omdio}, the operator $ {\cal D}_\om^{-1} $  (see \eqref{Domega})
is defined for all functions $ u  $ with zero $ \vphi $-average, and satisfies 
\be\label{Dom inverso}
 \| {\cal D}_\om^{-1} u \|_s \leq C \g^{-1} \| u \|_{s+ \tau} \, , \quad  \| {\cal D}_\om^{-1} u \|_s^{\Lipg} \leq C \g^{-1} \| u \|_{s+ 2 \tau+1}^{\Lipg} \, . 
\ee

\section{Approximate inverse}\label{costruzione dell'inverso approssimato}

In order to implement a convergent Nash-Moser scheme that leads to a solution of 
$ \mF(i, \zeta) = 0 $, 
we now construct an \emph{approximate right inverse} (which satisfies tame estimates) of the linearized operator 
\begin{equation}\label{operatore linearizzato}
d_{i, \zeta}{\cal F}(i_0, \zeta_0)[\widehat \imath\,, \widehat{\zeta} ] 
= {\cal D}_\om \widehat \imath - d_i X_{H_\e} ( i_0 (\vphi) ) [\widehat \imath ] + (0, \widehat \zeta, 0 ) \,,  
\end{equation}
see Theorem \ref{thm:stima inverso approssimato}. Note that 
$ d_{i, \zeta} {\cal F}(i_0, \zeta_0 ) $ is independent of $ \zeta_0 $ (see \eqref{operatorF}).

The notion of approximate right inverse is introduced in \cite{Z1}. It denotes a linear operator 
which is an \emph{exact} right inverse  at a solution $ (i_0, \zeta_0) $  of $ {\cal F}(i_0, \zeta_0) = 0 $.
We  implement the general strategy in \cite{BB13} 
which reduces the search of an approximate right inverse of \eqref{operatore linearizzato} 
to the search of an approximate inverse on the normal directions only. 

It is well known that an invariant torus $ i_0 $ with diophantine flow 
is isotropic (see e.g. \cite{BB13}), namely the pull-back $ 1$-form $ i_0^* \Lambda $ is closed, 
where $ \Lambda $ is the Liouville 1-form in \eqref{Lambda 1 form}. 
This is tantamount to say that the 2-form $ \cal W $ (see \eqref{2form}) vanishes on the torus $ i_0 (\T^\nu )$,  
because 
$ i_0^* {\cal W} =  i_0^* d \Lambda  = d i_0^* \Lambda $. 
For an ``approximately invariant'' torus $ i_0 $ the 1-form $ i_0^* \Lambda$ is only  ``approximately closed''.
In order to make this statement quantitative we consider
\begin{equation}\label{coefficienti pull back di Lambda}
\begin{aligned}
i_0^* \Lambda = {} & {\mathop \sum}_{k = 1}^\nu a_k (\vphi) d \vphi_k \,,   \\
a_k(\vphi) := {} & - ( [\pa_\ph \teta_0 (\vphi)]^T y_0 (\vphi) )_k 
+ \frac12 \big( \partial_{\vphi_k} z_0(\ph),  \partial_{x}^{-1} z_0(\ph) \big)_{L^2(\T)}
\end{aligned}
\end{equation}
and we quantify how small is 
\begin{equation} \label{def Akj} 
i_0^* {\cal W} 
= d \, i_0^* \Lambda 
= \mathop{\sum}_{1 \leq k < j \leq \nu} A_{k j}(\vphi) d \vphi_k \wedge d \vphi_j\,,
\quad 
A_{k j} := \partial_{\vphi_k} a_j - \partial_{\vphi_j} a_k.
\end{equation}
Along this section we will always assume the following hypothesis 
(which will be verified at each step of the Nash-Moser iteration):

\medskip

\noindent
$\bullet$ {\sc Assumption.} 
The map $\omega\mapsto i_0(\omega)$ is a Lipschitz function defined on some subset $\Omega_o \subset \Omega_\e$, where $\Omega_\e$ is defined in \eqref{Omega epsilon}, and, for some $ \mu := \mu (\t, \nu) >  0 $,   
\begin{gather} \label{ansatz 0}
\| {\mathfrak I}_0  \|_{s_0+\mu}^{\Lipg} 
\leq C \e^{5-2b} \g^{-1} 
= C \e^{5-4b}, \qquad 
\| Z \|_{s_0 + \mu}^{\Lipg} \leq C \e^{5-2b}, 
\\ \notag 
\gamma = \e^{2 + a} = \e^{2b}\,, \qquad 
b := 1 + (a/2) \,, \qquad 
a \in (0, 1/6), 
\end{gather}
where $\fracchi_0(\ph) := i_0(\ph) - (\ph,0,0)$, and 
\begin{equation} \label{def Zetone}
Z(\vphi) :=  (Z_1, Z_2, Z_3) (\vphi) := {\cal F}(i_0, \zeta_0) (\vphi) =
\om \cdot \pa_\vphi i_0(\vphi) - X_{H_{\e, \zeta_0}}(i_0(\vphi)) 
\end{equation}
is the ``error'' function. 

\begin{lemma}[Lemma 6.1 in \cite{bbm}] \label{zeta = 0}
$ |\zeta_0|^{\Lipg} \leq C \| Z \|_{s_0}^{\Lipg}$ \!\!\!. 
If $ {\cal F}(i_0, \zeta_0) = 0 $, then $ \zeta_0 = 0 $, 
and the torus $i_0(\ph)$ is invariant for $X_{H_\e}$.
\end{lemma}

Now we estimate the size of $ i_0^* {\cal W} $ in terms of $ Z $.
From \eqref{coefficienti pull back di Lambda}, \eqref{def Akj} one has  
$\| A_{kj} \|_s^\Lipg \leq_s \| \fracchi_0 \|_{s+2}^\Lipg$. 
Moreover, $A_{kj}$ also satisfies the following bound.

\begin{lemma}[Lemma 6.2 in \cite{bbm}]
The coefficients $A_{kj} (\vphi) $ in \eqref{def Akj} satisfy 
\begin{equation}\label{stima A ij}
\| A_{k j} \|_s^{\Lipg} \leq_s \gamma^{-1} \big(\| Z \|_{s+2\t+2}^{\Lipg}
+ \| Z \|_{s_0+1}^{\Lipg} \|  {\mathfrak I}_0 \|_{s+ 2 \tau + 2}^{\Lipg} \big)\,.
\end{equation}
\end{lemma}

As in \cite{BB13}, 
we first modify the approximate torus $ i_0 $ to obtain an isotropic torus $ i_\d $ which is 
still approximately invariant. We denote the Laplacian $ \Delta_\vphi := \sum_{k=1}^\nu \partial_{\vphi_k}^2 $.  

\begin{lemma}[Isotropic torus] \label{toro isotropico modificato}  
The torus $ i_\delta(\vphi) := (\theta_0(\vphi), y_\delta(\vphi), z_0(\vphi) ) $ defined by 
\begin{equation}\label{y 0 - y delta}
y_\d := y_0 +  [\pa_\ph \theta_0(\vphi)]^{- T}  \rho(\vphi) \, , \qquad 
\rho_j(\vphi) := \Delta_\vphi^{-1} {\mathop\sum}_{ k = 1}^\nu \partial_{\vphi_j} A_{k j}(\vphi) 
\end{equation}
is isotropic. 
If \eqref{ansatz 0} holds, then, for some $ \s := \s(\nu,\t) $,   
\begin{align} 
\label{2015-2}
\| y_\delta - y_0 \|_s^{\Lipg} 
& \leq_s \| \fracchi_0 \|_{s+\s}^{\Lipg} \,,  
\\
\| y_\delta - y_0 \|_s^{\Lipg} 
&  \leq_s  \gamma^{-1}  \big\{ \| Z \|_{s+\s}^\Lipg  + \| Z \|_{s_0+\s}^\Lipg \| \fracchi_0 \|_{s+\s}^\Lipg \big\} \,,
\label{stima y - y delta}
\\
\label{stima toro modificato}
\| {\cal F}(i_\delta, \zeta_0) \|_s^{\Lipg} 
& \leq_s  \| Z \|_{s + \s}^{\Lipg} + \| \fracchi_0 \|_{s+\s}^\Lipg   \| Z \|_{s_0 + \s}^{\Lipg} \,,  
\\
\label{derivata i delta}
\| \pa_i [ i_\d][ \widehat \imath ] \|_s & \leq_s \| \widehat \imath \|_s +  \| {\mathfrak I}_0\|_{s + \s} \| \widehat \imath  \|_s \, .
\end{align}
\end{lemma}
In the paper we denote equivalently the differential by $ \partial_i $ or  $ d_i $. Moreover we denote 
by $ \s := \s(\nu, \tau ) $ possibly different (larger) ``loss of derivatives''  constants. 

\begin{proof} 
It is sufficient to closely follow the proof of Lemma 6.3 of \cite{bbm}. 
We mention the only difference: equation (6.11) of \cite{bbm} is 
$\| {\cal F}(i_\delta, \zeta_0) \|_s^{\Lipg} 
\leq_s  \| Z \|_{s + \s}^{\Lipg} + \e^{2b-1} \g^{-1} \| \fracchi_0 \|_{s+\s}^\Lipg 
\| Z \|_{s_0 + \s}^{\Lipg}$, with a big factor $\e^{2b-1} \g^{-1} = \e^{-1}$ more with respect 
to the present bound \eqref{stima toro modificato}. 
In \eqref{stima toro modificato} there is no such a factor, because, 
by the estimates for $\pa_\theta \pa_y P, \pa_{yy}P, \pa_y \gr_z P$ 
in Lemma \ref{lemma quantitativo forma normale}, here we have
$\| \partial_y X_P (i) \|_s \leq_s \e^{2b} (1 + \| {\mathfrak I}\|_{s+3})$.
Hence \eqref{2015-2}, \eqref{stima y - y delta}, \eqref{ansatz 0} imply that
\begin{equation}\label{XP delta - XP}
\|X_ P(i_\delta ) - X_P(i_0 )\|_s 
\leq_s \|Z \|_{s + \s} + \|  {\mathfrak I}_0 \|_{s + \s} \|Z \|_{s_0 + \s}\,. 
\end{equation}
Then the proof goes on as in \cite{bbm}, without the large factor $\e^{2b-1} \g^{-1}$. 
\end{proof} 

In order to find an approximate inverse of the linearized operator $d_{i, \zeta} {\cal F}(i_\delta )$ 
we introduce a suitable set of symplectic coordinates nearby the isotropic torus $ i_\d $. 
We consider the map
$ G_\delta : (\psi, \eta, w) \to (\theta, y, z)$ of the phase space $\T^\nu \times \R^\nu \times H_S^\bot$ defined by
\begin{equation}\label{trasformazione modificata simplettica}
\begin{pmatrix} \theta \\ y \\ z \end{pmatrix} 
:= G_\delta \begin{pmatrix} \psi \\ \eta \\ w \end{pmatrix}  
:= \begin{pmatrix}
\theta_0(\psi) \\
y_\delta (\psi) + [\pa_\psi \theta_0(\psi)]^{-T} \eta + \big[ (\pa_\teta \tilde{z}_0) (\theta_0(\psi)) \big]^T \partial_x^{-1} w \\
z_0(\psi) + w
\end{pmatrix} 
\end{equation}
where $\tilde{z}_0 (\theta) := z_0 (\theta_0^{-1} (\theta))$. 
It is proved in \cite{BB13} that $ G_\delta $ is symplectic, using that the torus $ i_\d $ is isotropic 
(Lemma \ref{toro isotropico modificato}).
In the new coordinates,  $ i_\delta $ is the trivial embedded torus
$ (\psi , \eta , w ) = (\psi , 0, 0 ) $.  
The transformed Hamiltonian $ K := K(\psi, \eta, w, \zeta_0) $  
is (recall \eqref{hamiltoniana modificata})
\begin{align} \label{KHG}
K & := H_{\e, \zeta_0} \circ G_\d  \\ 
& = \, \theta_0(\psi) \cdot \zeta_0 
+ K_{00}(\psi) + K_{10}(\psi) \cdot \eta 
+ (K_{0 1}(\psi), w)_{L^2(\T)} 
+ \tfrac12 K_{2 0}(\psi)\eta \cdot \eta \nonumber \\ 
& \quad + \big( K_{11}(\psi) \eta , w \big)_{L^2(\T)} 
+ \tfrac12 \big(K_{02}(\psi) w , w \big)_{L^2(\T)} + K_{\geq 3}(\psi, \eta, w)  \notag 
\end{align}
where $ K_{\geq 3} $ collects the terms at least cubic in the variables $ (\eta, w )$.
At any fixed $\psi $, 
the Taylor coefficient $K_{00}(\psi) \in \R $,  
$K_{10}(\psi) \in \R^\nu $,  
$K_{01}(\psi) \in H_S^\bot$ (it is a function of $ x  \in \T $),
$K_{20}(\psi) $ is a $\nu \times \nu$ real matrix, 
$K_{02}(\psi)$ is a linear self-adjoint operator of $ H_S^\bot $ and 
$K_{11}(\psi) : \R^\nu \to H_S^\bot$. Note that the above Taylor coefficients do not
depend on the parameter $ \zeta_0 $.

The Hamilton equations associated to \eqref{KHG}  are 
\begin{equation}\label{sistema dopo trasformazione inverso approssimato}
\begin{cases}
\dot \psi \hspace{-30pt} & = K_{10}(\psi) +  K_{20}(\psi) \eta + 
K_{11}^T (\psi) w + \partial_{\eta} K_{\geq 3}(\psi, \eta, w)
\\
\dot \eta \hspace{-30pt} & = - [\partial_\psi \theta_0(\psi)]^T \zeta_0 - 
\partial_\psi K_{00}(\psi) - [\partial_{\psi}K_{10}(\psi)]^T  \eta - 
[\partial_{\psi} K_{01}(\psi)]^T  w  
\\
& \quad - \partial_\psi 
\{ \frac12 K_{2 0}(\psi)\eta \cdot \eta + ( K_{11}(\psi) \eta , w )_{L^2(\T)} + 
\frac12 ( K_{02}(\psi) w , w )_{L^2(\T)} 
\\
& \quad + K_{\geq 3}(\psi, \eta, w) \}
\\
\dot w \hspace{-30pt} & = \partial_x \big( K_{01}(\psi) + 
K_{11}(\psi) \eta +  K_{0 2}(\psi) w + \nabla_w K_{\geq 3}(\psi, \eta, w) \big) 
\end{cases} 
\end{equation}
where $ [\partial_{\psi}K_{10}(\psi)]^T $ is the $ \nu \times \nu $ transposed matrix 
and the operators $ [\partial_{\psi}K_{01}(\psi)]^T $
and $ K_{11}^T(\psi) : {H_S^\bot \to \R^\nu} $ are defined by the duality relation 
$( \partial_{\psi} K_{01}(\psi) [\hat \psi ],  w)_{L^2}$ 
$= \hat \psi \cdot [\partial_{\psi}K_{01}(\psi)]^T w $,
for all $\hat \psi \in \R^\nu$, $w \in H_S^\bot $, 
and similarly for $ K_{11} $. 
Explicitly, for all  $ w \in H_S^\bot $, 
and denoting $\underline{e}_k$ the $k$-th versor of $\R^\nu$, 
\[ 
K_{11}^T(\psi) w 
= \sum_{k=1}^\nu \big(K_{11}^T(\psi) w \cdot \underline{e}_k\big) \underline{e}_k   
= \sum_{k=1}^\nu  
\big( w, K_{11}(\psi) \underline{e}_k  \big)_{L^2(\T)}  \underline{e}_k  \, \in \R^\nu \, .  
\] 
In the next lemma we estimate the coefficients $ K_{00}, K_{10}, K_{01} $ 
of the Taylor expansion \eqref{KHG}.
Note that on an exact solution we have $ Z  = 0 $ and therefore
$ K_{00} (\psi) = {\rm const} $, $ K_{10}  = \om $ and $ K_{01}  = 0 $. 

\begin{lemma} \label{coefficienti nuovi}
Assume \eqref{ansatz 0}. Then there is $ \s := \s(\tau, \nu)$  such that 
\[ 
\|  \partial_\psi K_{00} \|_s^{\Lipg}, 
\| K_{10} - \om  \|_s^{\Lipg}, 
\| K_{0 1} \|_s^{\Lipg} 
\leq_s \| Z \|_{s + \s}^{\Lipg} + \| Z \|_{s_0 + \s}^{\Lipg} \| {\mathfrak I}_0 \|_{s+\s}^{\Lipg}\,.
\]
\end{lemma}

\begin{proof} Follow the proof of Lemma 6.4 in \cite{bbm}. 
The fact that here there is no factor $\e^{2b-1} \g^{-1}$ 
is a consequence of the better estimate \eqref{stima toro modificato} for $\mF(i_\d,\zeta_0)$ compared to the analogous estimate in \cite{bbm}. 
\end{proof}

\begin{remark} \label{rem:KAM normal form}  
If $ {\cal F} (i_0, \zeta_0) = 0 $ then $\zeta_0 = 0$  by 
Lemma \ref{zeta = 0}, and Lemma \ref{coefficienti nuovi} implies that
\eqref{KHG} simplifies to the normal form 
$$ 
K  = const + \om \cdot \eta + \frac12 K_{2 0}(\psi)\eta \cdot \eta 
+ ( K_{11}(\psi) \eta , w )_{L^2(\T)} 
+ \frac12 ( K_{02}(\psi) w , w )_{L^2(\T)} + K_{\geq 3} \, . 
$$ 
\end{remark}

We now estimate  $ K_{20}, K_{11}$ in \eqref{KHG}. 
The norm of $K_{20}$ is the sum of the norms of its matrix entries.  

\begin{lemma} \label{lemma:Kapponi vari}
Assume \eqref{ansatz 0}. Then 
\begin{align}\label{stime coefficienti K 20 11 bassa}
\|K_{20} - \e^{2b} \mathbb{A} D_S \|_s^{\Lipg} 
& \leq_s \e^{2b+1} + \e^{2b} \| {\mathfrak I}_0\|_{s + \s}^{\Lipg} \,,
\\ 
\label{stime coefficienti K 11 alta}  
\| K_{11} \eta \|_s^{\Lipg} 
& \leq_s \e^{5-2b} \| \eta \|_s^{\Lipg} 
+ \e^{2b} \| {\mathfrak I}_0\|_{s + \s}^{\Lipg} \| \eta \|_{s_0}^{\Lipg} \,, 
\\
\label{stime coefficienti K 11 alta trasposto}  
\| K_{11}^T w \|_s^{\Lipg} 
& \leq_s \e^{5-2b} \| w \|_{s + 2}^{\Lipg} 
+ \e^{2b}  \| {\mathfrak I}_0\|_{s + \s}^{\Lipg} \| w \|_{s_0 + 2}^{\Lipg} \, . 
\end{align}
In particular 
$ \| K_{20} - \e^{2b} \mathbb{A} D_S \|_{s_0 }^{\Lipg} \leq C \e^{5-2b}$, and 
$$
\| K_{11} \eta \|_{s_0}^{\Lipg} \leq C \e^{5-2b} \| \eta \|_{s_0}^{\Lipg} , \quad  
\| K_{11}^T w \|_{s_0}^{\Lipg} \leq C \e^{5-2b} \| w \|_{s_0+2}^{\Lipg} \,.
$$ 
\end{lemma}

\begin{proof}
See the proof of Lemma 6.6 in \cite{bbm}. 
\end{proof}

Consider the linear change of variables 
$(\widehat \theta, \widehat y, \widehat z) 
= D G_\delta(\vphi, 0, 0) [\widehat \psi, \widehat \eta, \widehat w]$, 
where $D G_\d(\ph,0,0)$ is obtained by linearizing $G_\d$ 
in \eqref{trasformazione modificata simplettica} at $(\ph,0,0)$, 
and it is represented by the matrix 
\begin{equation}\label{DGdelta}
D G_\delta(\vphi, 0, 0) 
= \begin{pmatrix}
\pa_\psi \theta_0(\vphi) & 0 & 0  \\
\pa_\psi y_\delta(\vphi) & \quad [\pa_\psi \theta_0(\vphi)]^{-T} & \quad  
- [(\pa_\theta \tilde{z}_0)(\theta_0(\vphi))]^T \partial_x^{-1} \\
\pa_\psi z_0(\vphi) & 0 & I
\end{pmatrix}.
\end{equation}

The linearized operator $d_{i, \zeta}{\cal F}(i_\delta, \zeta_0)$ 
transforms (approximately, see \eqref{verona 2}) into the operator obtained linearizing 
\eqref{sistema dopo trasformazione inverso approssimato} 
at $(\psi, \eta , w, \zeta ) = (\vphi, 0, 0, \zeta_0 )$
(with $ \partial_t \rightsquigarrow {\cal D}_\om $), 
which is the linear operator 
\[
B[\widehat \psi, \widehat \eta, \widehat w, \widehat \zeta] 
= \begin{pmatrix} 
B_1 [\widehat \psi, \widehat \eta, \widehat w, \widehat \zeta] \\
B_2 [\widehat \psi, \widehat \eta, \widehat w, \widehat \zeta] \\
B_3 [\widehat \psi, \widehat \eta, \widehat w, \widehat \zeta] 
\end{pmatrix},
\]
where
\begin{align} \label{lin idelta}
B_1 & := {\cal D}_\om \widehat \psi - \partial_\psi K_{10}(\vphi)[\widehat \psi \, ] - 
K_{2 0}(\vphi)\widehat \eta - K_{11}^T (\vphi) \widehat w, 
\\ \notag
B_2 & := {\cal D}_\om  \widehat \eta + [\partial_\psi \theta_0(\vphi)]^T \widehat \zeta 
+ \pa_\psi [\partial_\psi \theta_0(\vphi)]^T [ \widehat \psi, \zeta_0] 
+ \partial_{\psi\psi} K_{00}(\vphi)[\widehat \psi]  
\\ \notag
& \qquad + [\partial_\psi K_{10}(\vphi)]^T \widehat \eta 
+ [\partial_\psi  K_{01}(\vphi)]^T \widehat w,   
\\ \notag
B_3 & := {\cal D}_\om  \widehat w 
- \pa_x \{ \pa_\psi K_{01}(\vphi)[\widehat \psi] + K_{11}(\vphi) \widehat \eta + K_{02}(\vphi) \widehat w \}.
\end{align}

\begin{lemma}[Lemma 6.7 in \cite{bbm}] \label{lemma:DG}
Assume \eqref{ansatz 0} and let 
$ \widehat \imath := (\widehat \psi, \widehat \eta, \widehat w)$. 
Then 
\begin{gather} \label{DG delta}
\|DG_\delta(\vphi,0,0) [\widehat \imath] \|_s + \|DG_\delta(\vphi,0,0)^{-1} [\widehat \imath] \|_s 
\leq_s \| \widehat \imath \|_{s} +  \| {\mathfrak I}_0 \|_{s + \s}  \| \widehat \imath \|_{s_0}\,,
\\ 
\| D^2 G_\delta(\vphi,0,0)[\widehat \imath_1, \widehat \imath_2] \|_s 
\leq_s  \| \widehat \imath_1\|_s \| \widehat \imath_2 \|_{s_0} 
+ \| \widehat \imath_1\|_{s_0} \| \widehat \imath_2 \|_{s} 
+  \| {\mathfrak I}_0  \|_{s + \s}   \|\widehat \imath_1 \|_{s_0} \| \widehat \imath_2\|_{s_0}
\notag 
\end{gather}
for some $\s := \s(\nu,\t)$. 
The same estimates hold 
for the  $\| \ \|_s^{\Lipg}$ norm.
\end{lemma}

In order to construct an approximate inverse of \eqref{lin idelta} it is sufficient to solve the equation
\begin{equation}\label{operatore inverso approssimato} 
{\mathbb D} [\widehat \psi, \widehat \eta, \widehat w, \widehat \zeta ] := 
  \begin{pmatrix}
{\cal D}_\om \widehat \psi - 
K_{20}(\vphi) \widehat \eta - K_{11}^T(\vphi) \widehat w\\
{\cal D}_\om  \widehat \eta + [\partial_\psi \theta_0(\vphi)]^T \widehat \zeta  \\
{\cal D}_\om \widehat w  - \partial_x K_{11}(\vphi)\widehat \eta -  \partial_x K_{0 2}(\vphi) \widehat w 
\end{pmatrix} =
\begin{pmatrix}
g_1 \\ g_2 \\ g_3 
\end{pmatrix}
\end{equation}
which is obtained by neglecting in $B_1, B_2, B_3$ in \eqref{lin idelta} 
the terms $ \partial_\psi K_{10} $, $ \partial_{\psi \psi} K_{00} $, $ \partial_\psi K_{00} $, $ \partial_\psi K_{01} $ and
$ \pa_\psi [\partial_\psi \theta_0(\vphi)]^T [ \cdot , \zeta_0] $
(these terms are naught at a solution by Lemmata \ref{coefficienti nuovi} and \ref{zeta = 0}). 

First we solve the second equation in \eqref{operatore inverso approssimato}, namely 
$ {\cal D}_\om  \widehat \eta = g_2  -  [\partial_\psi \theta_0(\vphi)]^T \widehat \zeta $. 
We choose $ \widehat \zeta $ so that the $\vphi$-average of the right hand side 
is zero, namely 
\begin{equation}\label{fisso valore di widehat zeta}
\widehat \zeta = \langle g_2 \rangle 
\end{equation}
(we denote $ \langle g \rangle := (2 \pi)^{- \nu} \int_{\T^\nu} g (\vphi) d \vphi $). 
Note that the $\ph$-averaged matrix $ \langle [\partial_\psi \theta_0 ]^T  \rangle$ 
$ = \langle I + [\pa_\psi \Theta_0]^T \rangle = I $ because $\theta_0(\ph) = \ph + \Theta_0(\ph)$ and 
$\Theta_0(\ph)$ is a periodic function.
Therefore 
\begin{equation}\label{soleta}
\widehat \eta := {\cal D}_\om^{-1} \big(
g_2 - [\partial_\psi \theta_0(\ph) ]^T \langle g_2 \rangle \big) + 
\langle \widehat \eta \rangle \, , \quad \langle \widehat \eta \rangle \in \R^\nu \, , 
\end{equation}
where the average $\la \widehat \eta \ra$ will be fixed below. 
Then we consider the third equation
\begin{equation}\label{cal L omega}
{\cal L}_\om \widehat w = g_3 + \partial_x K_{11}(\vphi) \widehat \eta\,,  \
\quad  {\cal L}_\om := \om \cdot \partial_\vphi -  \partial_x K_{0 2}(\vphi) \, .
\end{equation}

\noindent
$\bullet$ {\sc Inversion assumption.} 
{\it There exists a set $ \Omega_\infty \subset \Omega_o$ such that  
for all $ \omega \in \Omega_\infty $, 
for every function $ g \in H^{s+\mu}_{S^\bot} (\T^{\nu+1}) $ 
there exists a solution $ h :=  {\cal L}_\om^{- 1} g  \in H^{s}_{S^\bot} (\T^{\nu+1}) $ 
of the linear equation $ {\cal L}_\om h = g $, which satisfies}
\begin{equation}\label{tame inverse}
\| {\cal L}_\om^{- 1} g \|_s^{\Lipg} \leq C(s) \g^{-1} 
\big( \| g \|_{s + \mu}^{\Lipg} + \e^2 \g^{-1} \| {\mathfrak I}_0 \|_{s + \mu}^{\Lipg}  
\|g \|_{s_0}^{\Lipg}  \big) 
\end{equation}
\emph{for some $ \mu := \mu (\tau, \nu) >  0 $}.

\medskip

By the above assumption  there exists a solution 
\begin{equation}\label{normalw}
\widehat w := {\cal L}_\om^{-1} [ g_3 + \partial_x K_{11}(\vphi) \widehat \eta \, ] 
\end{equation}
of \eqref{cal L omega}. 
Finally, we solve the first equation in \eqref{operatore inverso approssimato}, 
which, substituting \eqref{soleta}, \eqref{normalw}, becomes
\begin{equation}\label{equazione psi hat}
{\cal D}_\om \widehat \psi  = 
g_1 +  M_1(\vphi) \langle \widehat \eta \rangle + M_2(\vphi) g_2 + M_3(\vphi) g_3 - 
M_2(\vphi)[\pa_\psi \theta_0]^T \langle g_2 \rangle \,,
\end{equation}
where
\be \label{cal M2}
\begin{aligned}
M_1(\vphi) & := K_{2 0}(\vphi) + K_{11}^T(\vphi) \mL_\omega^{-1} \pa_x K_{11}(\vphi)\,, \quad
M_2(\vphi) :=  M_1 (\vphi)  {\cal D}_\om^{-1} \, , \\
M_3(\vphi) & :=  K_{11}^T (\vphi) {\cal L}_\om^{-1} \, .  
\end{aligned}
\ee
To solve equation \eqref{equazione psi hat} we have to choose 
$\langle \widehat \eta \rangle$ such that the right hand side in \eqref{equazione psi hat} has zero average.  
By Lemma \ref{lemma:Kapponi vari} and \eqref{ansatz 0}, the $\ph$-averaged matrix 
\be\label{media-twist} 
\langle M_1 \rangle = \e^{2 b} \mathbb{A} D_S + O(\e^{5-2b}) \, . 
\ee
Therefore, for $ \e $ small,  $\langle M_1 \rangle$ is invertible and $\langle M_1 \rangle^{-1} = O(\e^{-2 b}) = O(\gamma^{- 1})$ 
(recall \eqref{link gamma b}). Thus we define 
\begin{equation}\label{sol alpha}
\langle \widehat \eta \rangle  := - \langle M_1 \rangle^{-1} 
[ \langle g_1 \rangle + \langle M_2 g_2 \rangle + \langle M_3 g_3 \rangle - 
\langle M_2 [\pa_\psi \theta_0]^T   \rangle  \langle g_2 \rangle ].
\end{equation}
With this choice of $\langle \widehat \eta \rangle$, equation \eqref{equazione psi hat} has the solution
\begin{equation}\label{sol psi}
\widehat \psi :=
{\cal D}_\om^{-1} [ g_1 + M_1(\vphi) \langle \widehat \eta \rangle + M_2(\vphi) g_2 + M_3(\vphi) g_3 -
M_2(\vphi)[\pa_\psi \theta_0]^T \langle g_2 \rangle ].
\end{equation}
In conclusion, we have constructed 
a solution  $(\widehat \psi, \widehat \eta, \widehat w, \widehat \zeta)$ of the linear system \eqref{operatore inverso approssimato}.

\begin{proposition}\label{prop: ai}
Assume \eqref{ansatz 0} and \eqref{tame inverse}. 
Then, $\forall \om \in \Omega_\infty $, $ \forall g := (g_1, g_2, g_3) $,
 the system \eqref{operatore inverso approssimato} has a solution 
$ {\mathbb D}^{-1} g := (\widehat \psi, \widehat \eta, \widehat w, \widehat \zeta ) $
where $(\widehat \psi, \widehat \eta, \widehat w, \widehat \zeta)$ are defined in 
\eqref{sol psi}, \eqref{soleta}, \eqref{sol alpha}, \eqref{normalw}, \eqref{fisso valore di widehat zeta},  and satisfy
\begin{equation} \label{stima T 0 b}
\| {\mathbb D}^{-1} g \|_s^{{\rm Lip}(\gamma)}  
\leq_s \gamma^{-1} \big( \| g \|_{s + \mu}^{{\rm Lip}(\gamma)}  
+ \e^2 \gamma^{-1}  \| {\mathfrak I}_0  \|_{s + \mu}^{{\rm Lip}(\gamma)}  
\| g \|_{s_0 + \mu}^{{\rm Lip}(\gamma)} \big).
\end{equation}
\end{proposition}

\begin{proof}
Recalling \eqref{cal M2}, by  Lemma \ref{lemma:Kapponi vari}, \eqref{tame inverse}, \eqref{ansatz 0}  we get $ \| M_2 h \|_{s_0} + \| M_3 h \|_{s_0} $ 
$\leq C \| h \|_{s_0 + \s} $. 
Then, by \eqref{sol alpha} and $\langle M_1 \rangle^{-1} = O(\e^{-2 b}) = O(\gamma^{-1}) $, 
we deduce 
$ |\langle \widehat \eta\rangle|^{\Lipg} \leq C\gamma^{-1} \| g \|_{s_0+ \s}^{\Lipg} $
and  \eqref{soleta}, \eqref{Dom inverso} imply 
$ \| \widehat \eta \|_s^{\Lipg} \leq_s \gamma^{-1} \big( \| g \|_{s + \s}^\Lipg $ 
$ + \| \fracchi_0 \|_{s + \s } \| g \|_{s_0}^\Lipg  \big)$. 
The bound \eqref{stima T 0 b} is sharp for $ \widehat w $ because $ {\cal L}_\om^{-1} g_3 $ in  \eqref{normalw}
is estimated using \eqref{tame inverse}. Finally $  \widehat \psi $ 
satisfies \eqref{stima T 0 b} using
\eqref{sol psi},  \eqref{cal M2}, \eqref{tame inverse}, \eqref{Dom inverso} and Lemma \ref{lemma:Kapponi vari}. 
\end{proof}

Let $\widetilde{G}_\delta (\psi, \eta, w, \zeta) := ( G_\delta (\psi, \eta, w), \zeta)$.  
Let $\| (\psi, \eta, w, \zeta) \|_s^\Lipg$ denote the maximum between $\| (\psi, \eta, w) \|_s^\Lipg$
and $| \zeta |^\Lipg$. 
We prove that the operator 
\begin{equation}\label{definizione T} 
{\bf T}_0 := (D { \widetilde G}_\delta)(\vphi,0,0) \circ {\mathbb D}^{-1} \circ (D G_\delta) (\vphi,0,0)^{-1}
\end{equation}
is an approximate right inverse for $d_{i,\zeta} {\cal F}(i_0 )$. 

\begin{theorem} {\bf (Approximate inverse)} \label{thm:stima inverso approssimato}
Assume \eqref{ansatz 0} and the inversion assumption \eqref{tame inverse}. 
Then there exists $ \mu := \mu (\tau, \nu) >  0 $ such that, for all $ \om \in \Om_\infty $, 
for all $ g := (g_1, g_2, g_3) $,  
the operator $ {\bf T}_0 $ defined in \eqref{definizione T} satisfies  
\begin{equation}\label{stima inverso approssimato 1}
\| {\bf T}_0 g \|_{s}^{{\rm Lip}(\gamma)}  
\leq_s  \gamma^{-1} \big(\| g \|_{s + \mu}^{{\rm Lip}(\gamma)}  
+ \e^2 \gamma^{-1}  \| {\mathfrak I}_0 \|_{s + \mu}^{{\rm Lip}(\gamma)}  
\| g \|_{s_0 + \mu}^{{\rm Lip}(\gamma)}  \big). 
\end{equation}
The operator $\mathbf{T}_0$ 
is an approximate inverse of $d_{i, \zeta} {\cal F}(i_0 )$, namely 
\begin{align} \label{stima inverso approssimato 2} 
& \| ( d_{i, \zeta} {\cal F}(i_0) \circ {\bf T}_0 - I ) g \|_s^{{\rm Lip}(\gamma)}  
\\ & \leq_s \gamma^{- 1}\| {\cal F}(i_0, \zeta_0) \|_{s_0 + \mu}^\Lipg \| g \|_{s + \mu}^\Lipg  
\notag 
\\ & \quad 
+ \gamma^{- 1} \big\{ \| {\cal F}(i_0, \zeta_0) \|_{s + \mu}^\Lipg 
+ \e^2 \g^{-1} \| {\cal F}(i_0, \zeta_0) \|_{s_0 + \mu}^\Lipg \| {\mathfrak I}_0 \|_{s + \mu}^\Lipg \big\} \| g \|_{s_0 + \mu}^\Lipg \,. \nonumber
\end{align}
\end{theorem}

\begin{proof} 
In this proof we denote $\| \ \|_s$ instead of $\| \ \|_s^{\Lipg}$. 
The bound \eqref{stima inverso approssimato 1} 
follows from \eqref{definizione T}, \eqref{stima T 0 b},  \eqref{DG delta}.
By \eqref{operatorF}, since $ X_\mN $ does not depend on $ y $,  
and $ i_\d $ differs from $ i_0 $ only for the $ y$ component,  
we have 
\begin{align} \label{verona 0}
& d_{i, \zeta} {\cal F}(i_0 )[\, \widehat \imath, \widehat \zeta \, ]  
- d_{i, \zeta} {\cal F}(i_\delta ) [\, \widehat \imath, \widehat \zeta \, ]
= d_i X_P (i_\delta)  [\, \widehat \imath \, ]  - d_i X_P (i_0) [\, \widehat \imath \, ]
\\ 
& = \int_0^1 \pa_y d_i X_P (\theta_0, y_0 + s(y_\delta -y_0), z_0) [y_\d -y_0, \widehat \imath\,] ds 
=: {\cal E}_0 [\, \widehat \imath, \widehat \zeta \, ]. \nonumber
\end{align}
By \eqref{D yii},  \eqref{2015-2}, \eqref{stima y - y delta}, \eqref{ansatz 0}, 
we estimate
\begin{equation}\label{stima parte trascurata 1}
\| {\cal E}_0 [\, \widehat \imath, \widehat \zeta \, ] \|_s \leq_s 
\| Z \|_{s_0 + \s} \| \widehat \imath \|_{s + \s} 
+ (\| Z \|_{s + \s} + \| Z \|_{s_0 + \s} \| \fracchi_0 \|_{s+\s}) \| \widehat \imath \|_{s_0 + \s} 
\end{equation}
where $Z := \mF(i_0, \zeta_0)$ (recall \eqref{def Zetone}). 
 Note that $\mE_0[\widehat \imath, \widehat \zeta]$ is, in fact, independent of $\widehat \zeta$.
Denote the set of variables $  (\psi, \eta, w) =: {\mathtt u} $. 
Under the transformation $G_\delta $, the nonlinear operator ${\cal F}$ in \eqref{operatorF} transforms into 
\be \label{trasfo imp}
{\cal F}(G_\delta(  {\mathtt u} (\vphi) ), \zeta ) 
= D G_\delta( {\mathtt u}  (\vphi) ) \big(  {\cal D}_\om {\mathtt u} (\vphi) - X_K ( {\mathtt u} (\vphi), \zeta)  \big), 
\ee
where $K = H_{\e, \zeta} \circ G_\delta$, 
see \eqref{KHG}-\eqref{sistema dopo trasformazione inverso approssimato}. 
Differentiating  \eqref{trasfo imp} at the trivial torus 
$ {\mathtt u}_\delta (\vphi) = G_\delta^{-1}(i_\delta) (\vphi) = (\ph, 0 , 0 ) $, 
at $ \zeta = \zeta_0 $, 
in the direction $(\widehat {\mathtt u}, \widehat \zeta\,)$
$= (D G_\d ({\mathtt u}_\d)^{-1} [\, \widehat \imath \, ], \widehat \zeta) 
= D {\widetilde G}_\d ({\mathtt u}_\d)^{-1} [\, \widehat \imath , \widehat \zeta \, ] $, 
we get
\begin{align} \label{verona 2}
d_{i , \zeta} {\cal F}(i_\delta ) [\, \widehat \imath, \widehat \zeta \, ]
=  & D G_\delta( {\mathtt u}_\delta) 
\big( {\cal D}_\om \widehat {\mathtt u} 
- d_{\mathtt u, \zeta} X_K( {\mathtt u}_\delta, \zeta_0) [\widehat {\mathtt u}, \widehat \zeta \, ] 
\big) 
+ {\cal E}_1 [ \, \widehat \imath , \widehat \zeta \, ]\,,
\\
\label{E1}
{\cal E}_1 [\, \widehat \imath , \widehat \zeta \, ] 
:=  & 
D^2 G_\delta( {\mathtt u}_\delta) \big[ D G_\delta( {\mathtt u}_\delta)^{-1} {\cal F}(i_\delta, \zeta_0), \,  D G_\d({\mathtt u}_\d)^{-1} 
[ \, \widehat \imath \,  ] \big] \,, 
\end{align}
where  $ d_{\mathtt u, \zeta} X_K( {\mathtt u}_\delta, \zeta_0) $ is expanded in \eqref{lin idelta}.
In fact, ${\cal E}_1$ is independent of $\widehat \zeta$. 
We split  
\[ 
{\cal D}_\om \widehat {\mathtt u} 
- d_{\mathtt u, \zeta} X_K( {\mathtt u}_\delta, \zeta_0) [\widehat {\mathtt u}, \widehat \zeta] 
= \mathbb{D} [\widehat {\mathtt u}, \widehat \zeta \, ] + R_Z [ \widehat {\mathtt u}, \widehat \zeta \, ], 
\]
where $ {\mathbb D} [\widehat {\mathtt u}, \widehat \zeta] $ is defined in \eqref{operatore inverso approssimato} and 
$R_Z [  \widehat \psi, \widehat \eta, \widehat w, \widehat \zeta]$ is defined by difference, 
so that its first component is $ - \pa_\psi K_{10}(\vphi) [\widehat \psi ]$, 
its second component is 
\[
\pa_\psi [\pa_\psi \theta_0(\vphi)]^T [\widehat \psi, \zeta_0] 
+ \pa_{\psi \psi} K_{00} (\vphi) [\widehat \psi] 
+ [\pa_\psi K_{10}(\vphi)]^T \widehat \eta 
+ [\pa_\psi K_{01}(\vphi)]^T \widehat w,
\]
and its third component is 
$- \pa_x \{ \pa_{\psi} K_{01}(\vphi)[\widehat \psi] \} $
(in fact, $R_Z$ is independent of $\widehat \zeta$). 
By \eqref{verona 0} and \eqref{verona 2}, 
\begin{equation} \label{E2}
\begin{aligned}
d_{i, \zeta} {\cal F}(i_0 ) 
& = D G_\delta({\mathtt u}_\delta) \circ {\mathbb D} \circ D {\widetilde G}_\delta 
({\mathtt u}_\delta)^{-1} 
+ {\cal E}_0 + {\cal E}_1 + \mE_2, \\
\mE_2 & := D G_\delta( {\mathtt u}_\delta) \circ R_Z \circ D {\widetilde G}_\delta 
({\mathtt u}_\delta)^{-1}.
\end{aligned}
\end{equation}
By Lemmata \ref{coefficienti nuovi}, \ref{lemma:DG}, \ref{zeta = 0}, 
and \eqref{stima toro modificato}, \eqref{ansatz 0},
the terms $\mE_1, \mE_2$ satisfy the same bound \eqref{stima parte trascurata 1} as $\mE_0$.
Thus the sum $\mE := \mE_0 + \mE_1 + \mE_2$ satisfies \eqref{stima parte trascurata 1}.
Applying $ {\bf T}_0 $ defined in \eqref{definizione T} to the right in \eqref{E2}, 
since $ {\mathbb D} \circ  {\mathbb D}^{-1} = I $ (see Proposition \ref{prop: ai}), 
we get $d_{i, \zeta} {\cal F}(i_0 ) \circ {\bf T}_0  - I 
= \mE \circ {\bf T}_0$. 
Then \eqref{stima inverso approssimato 2} follows from 
\eqref{stima inverso approssimato 1} and the bound \eqref{stima parte trascurata 1} for $\mE$. 
\end{proof}

\section{The linearized operator in the normal directions}\label{linearizzato siti normali}

The goal of this section is to write an explicit expression of the linearized operator 
$\mL_\om$ defined in \eqref{cal L omega}, see Proposition \ref{prop:lin}. 
To this aim, we compute $ \frac12 ( K_{02}(\psi) w, w )_{L^2(\T)} $, $ w \in H_S^\bot$, 
which collects all the terms of 
$(H_\e \circ G_\d)(\psi, 0, w)$ that are quadratic in $w$, see \eqref{KHG}.
We first recall some preliminary lemmata. 

\begin{lemma}[Lemma 7.1-\cite{bbm}] \label{lemma astratto potente}
Let $ H $ be a Hamiltonian function of class $ C^2 (  H^1_0(\T_x), \R )$ 
and consider a map
$ \Phi(u) := u + \Psi(u) $ satisfying $\Psi (u) = \Pi_E \Psi(\Pi_E u)$, for all $ u $, 
where $E$ is a finite dimensional subspace as in \eqref{def E finito}. Then 
\begin{equation}\label{lint2}
\pa_u [\nabla ( H \circ \Phi)] (u)  [h] = (\partial_u  \nabla H )(\Phi(u)) [h] + {\cal R}(u)[h]\,,
\end{equation}
where  $ {\cal R}(u) $ 
has the ``finite dimensional'' form 
\begin{equation}\label{forma buona resto}
{\cal R}(u)[h] =  {\mathop\sum}_{|j| \leq C} \big( h , g_j(u) \big)_{L^2(\T)} \chi_j(u) 
\end{equation}
with $ \chi_j (u) = e^{\ii j x} $ or  $ g_j(u) = e^{\ii j x} $. 
The remainder in \eqref{forma buona resto} is 
$ {\cal R} (u) =  {\cal R}_0 (u) + {\cal R}_1 (u) + {\cal R}_2 (u) $ with
\begin{align}\label{resti012} 
{\cal R}_0 (u) & :=  (\partial_u \nabla H)(\Phi(u)) \pa_u \Psi (u), \qquad  
{\cal R}_1 (u) :=  [\partial_{u }\{ \Psi'(u)^T\}] [ \cdot , \nabla H(\Phi(u)) ], \nonumber \\
\,  {\cal R}_2 (u)  & :=  [\pa_u \Psi (u)]^T (\partial_u \nabla H)(\Phi(u)) \pa_u \Phi(u). 
\end{align}
\end{lemma}

\begin{lemma}[Lemma 7.3 in \cite{bbm}] \label{remark : decay forma buona resto}
Let $ {\cal R} $ be an operator of the form
\begin{equation}\label{forma buona con gli integrali}
{\cal R} h = \sum_{|j| \leq C } \int_0^1 \big(h\,,\,g_j(\tau) \big)_{L^2(\T)} \chi_j (\tau)\,d \tau\,,
\end{equation}
where the functions $g_j(\tau),\,\chi_j(\tau) \in H^s$, $\tau \in [0, 1]$ depend in a Lipschitz way on the parameter $\omega$. 
Then its matrix $s$-decay norm (see \eqref{matrix decay norm}-\eqref{matrix decay norm Lip}) satisfies 
$$ 
| {\cal R} |_s^\Lipg 
\leq_s \sum_{|j| \leq C} \sup_{\tau \in [0,1]} 
\big( \| \chi_j(\tau) \|_s^\Lipg \| g_j(\tau) \|_{s_0}^\Lipg 
+ \| \chi_j(\tau) \|_{s_0}^\Lipg \| g_j(\tau) \|_s^\Lipg \big). 
$$ 
\end{lemma}

\subsection{Composition with the map $G_\d$} \label{section:appr}

In the sequel we use the fact that $\fracchi_\d := \fracchi_\d (\ph ; \om) := i_\d (\ph; \om ) - (\ph,0,0) $ satisfies, by \eqref{2015-2} and \eqref{ansatz 0},  
\begin{equation}\label{ansatz delta}
\| {\mathfrak I}_\d \|_{s_0+\mu}^{\Lipg} 
\leq C\e^{5 - 2b} \gamma^{-1} = C \e^{5-4b}.
\end{equation}
In this section we study the Hamiltonian 
$ K := H_\e \circ G_\d = \e^{-2b} \mH \circ A_\e \circ G_\d $ 
defined in \eqref{KHG}, \eqref{def H eps}. 
Recalling \eqref{def A eps}, \eqref{trasformazione modificata simplettica}, 
$A_\e \circ G_\d$ has the form  
\begin{equation}  \label{A eps G delta}
A_\e (G_\d(\psi, \eta, w))
= \e v_\e \big( \theta_0(\psi), \, y_\d(\psi) + L_1(\psi) \eta + L_2(\psi) w \big)
+ \e^b (z_0(\psi) + w)
\end{equation}
where $v_\e$ is defined in \eqref{def A eps}, and 
\be\label{L1 L2}
L_1(\psi) := [\pa_\psi \theta_0(\psi)]^{-T} \, , \quad
L_2(\psi) := \big[ (\pa_\teta \tilde{z}_0) (\theta_0(\psi)) \big]^T \partial_x^{-1} \, . 
\ee
By Taylor's formula, we develop \eqref{A eps G delta} in $w$ at 
$(\eta,w)=(0,0)$, and we get
$$
(A_\e \circ G_\d)(\psi, 0, w) 
= T_\delta(\psi) + T_1(\psi) w + T_2(\psi)[w,w] + T_{\geq 3}(\psi, w) \, , 
$$ 
where 
\begin{equation}\label{T0}
T_\delta(\psi) := A_\e(G_\d(\psi, 0, 0))
= \e v_\d(\psi) + \e^b z_0(\psi), \ \ 
v_\d (\psi) := v_\e(\theta_0(\psi), y_\d(\psi))
\end{equation}
is the approximate isotropic torus in the phase space $  H^1_0 (\T) $ (it corresponds to $ i_\d $ in Lemma \ref{toro isotropico modificato}),  
\[
T_1(\psi) w := \e^{2b-1} U_1 (\psi) w + \e^b w, \quad 
T_2(\psi)[w,w] := \e^{4b - 3} U_2(\psi)[w,w]
\]
\begin{align}
U_1(\psi) w 
& = \sum_{j \in S} \frac{|j| [ L_2(\psi) w ]_j \, e^{\ii [\theta_0(\psi)]_j}} 
{2 \sqrt{ \xi_j + \e^{2(b-1)} |j| [ y_\d(\psi) ]_j }} \,  e^{\ii jx},
\label{T1} \\
U_2(\psi)[w,w] 
& = - \sum_{j \in S} \frac{j^2 [ L_2(\psi) w ]_j^2 \, e^{\ii [\theta_0(\psi)]_j}} 
{8 \{ \xi_j + \e^{2(b-1)} |j| [ y_\d(\psi) ]_j \}^{3/2} } \,  e^{\ii jx},
\label{T2}
\end{align}
and $T_{\geq 3}(\psi, w)$ collects all the terms of order at least cubic in $w$. 
The terms  $U_1, U_2 = O(1)$ in $\e$.
Moreover, using that $ L_2 (\psi) $ in \eqref{L1 L2} vanishes as $ z_0 = 0 $,  
they satisfy
\begin{equation}\label{extra piccolezza}
\begin{aligned}
\| U_1 w \|_s 
& \leq_s \| \fracchi_\d \|_s \| w \|_{s_0}   +  \| \fracchi_\d \|_{s_0}  \| w \|_s  \,, \\
\| U_2 [w,w] \|_s 
& \leq_s \| \fracchi_\d \|_s \| \fracchi_\d \|_{s_0} \| w \|_{s_0}^2 
+ \| \fracchi_\d \|_{s_0}^2 \| w \|_{s_0} \| w \|_s  
\end{aligned}
\end{equation}
and also in the  $ \|  \ \|_s^\Lipg $-norm.
We expand $ {\cal H} $ by Taylor's formula 
\[
\mH(u+h) = \mH(u) + ( (\gr \mH)(u), h )_{L^2(\T)} 
+ \tfrac12 ( (\pa_u \gr \mH)(u) [h], h )_{L^2(\T)} 
+ O(h^3).
\] 
Specifying at $u = T_\delta(\psi)$ and $ h = T_1(\psi) w + T_2(\psi)[w,w] + T_{\geq 3}(\psi,w)$, 
we obtain that the sum of all the components of $ K = \e^{-2b} (\mH \circ A_\e \circ G_\d)(\psi, 0, w) $ 
that are quadratic in $w$ is  
$$
\begin{aligned}
\tfrac12 ( K_{02}w, w )_{L^2(\T)} 
& = \e^{-2b} ( (\gr \mH)(T_\delta ), T_2 [w,w] )_{L^2(\T)} \\ 
& \quad + \e^{-2b} \tfrac12 ( (\pa_u \gr \mH)(T_\delta ) [T_1 w], T_1 w )_{L^2(\T)} \,. 
\end{aligned}
$$
Inserting the expressions \eqref{T1}, \eqref{T2} in the last equality we get
\begin{align} \label{K02}
K_{02}(\psi) w 
& = (\pa_u \gr \mH)(T_\delta) [w]  
+ 2 \e^{b-1} (\pa_u \gr \mH)(T_\delta) [U_1 w] 
\\ & \quad \,
+ \e^{2(b-1)} U_1^T (\pa_u \gr \mH)(T_\delta) [U_1 w]  
+ 2 \e^{2b- 3} U_2[w, \cdot]^T (\gr \mH)(T_\delta). \notag
\end{align}

\begin{lemma}\label{dopo l'approximate inverse}
The operator $K_{02}$ reads 
\begin{equation}\label{piccolezza resti}
 ( K_{02}(\psi) w, w )_{L^2(\T)} 
=  ( (\pa_u \gr \mH)(T_\delta) [w], w )_{L^2(\T)} 
+  ( R(\psi) w, w )_{L^2(\T)}
\end{equation}
where $R(\psi)w $ has the ``finite dimensional'' form 
\begin{equation}\label{forma buona resto con psi}
R(\psi) w  =  {\mathop\sum}_{|j| \leq C} \big( w , g_j(\psi) \big)_{L^2(\T)} \chi_j(\psi).
\end{equation}
The functions $g_j, \chi_j$ satisfy, for some $\sigma := \sigma (\nu, \tau) > 0$, 
\begin{align} \label{piccolo FBR}
& \| g_j \|_s^\Lipg \| \chi_j \|_{s_0}^\Lipg + \| g_j \|_{s_0}^\Lipg \| \chi_j \|_s^\Lipg 
\leq_s \e^{b+1} \| {\mathfrak I}_\delta \|_{s + \sigma}^\Lipg \,,
\\
& \| \partial_i g_j [\widehat \imath ]\|_s \| \chi_j \|_{s_0} 
+ \| \partial_i g_j [\widehat \imath ]\|_{s_0} \| \chi_j \|_{s} 
+ \| g_j \|_{s_0} \| \partial_i \chi_j [\widehat \imath ] \|_s 
+ \| g_j \|_{s} \| \partial_i \chi_j [\widehat \imath ]\|_{s_0}  
\notag \\ & 
\leq_s \e^{b + 1} ( \| \widehat \imath \|_{s + \sigma}
+ \| {\mathfrak I}_\delta\|_{s + \sigma}  \|\widehat \imath \|_{s_0 + \sigma}) \,,  
\label{derivata piccolo FBR}
\end{align}
where $i = (\theta, y, z)$ (see \eqref{embedded torus i}) and 
$\widehat \imath = (\widehat \theta, \widehat y, \widehat z)$.
\end{lemma}

\begin{proof}
Since $ U_1 = \Pi_S U_1 $ and  $ U_2 = \Pi_S U_2 $, 
the last three terms in \eqref{K02} have all the form \eqref{forma buona resto con psi}.
We have to prove that they are also small in size. 
 
By \eqref{shape H2}, \eqref{trasformazione modificata simplettica}, \eqref{L1 L2}, 
the only term in $\e^{-2b} H_2(A_\e(G_\delta(\psi, \eta, w)))$ 
that is quadra\-tic in $w$ is $\frac12 \int_\T w_x^2 \, dx$, 
so this is the only contribution to \eqref{K02} coming from $H_2$. 

It remains to consider all the terms coming from $\mH_{\geq 4} := \mH_4 + {\cal H}_{\geq 5} 
= O(u^4)$.
The term $\e^{b - 1} \partial_u \nabla {\cal H}_{\geq 4}(T_\delta) U_1$, 
the term $\e^{2(b - 1)} U_1^T (\pa_u \nabla {\cal H}_{\geq 4})(T_\delta) U_1$ 
and the term $\e^{2 b - 3} U_2^T \nabla {\cal H}_{\geq 4}(T_\delta) $
have all the form \eqref{forma buona resto con psi} and, 
using the inequality $ \| T_\delta \|_s^\Lipg \leq \e (1 + \| \fracchi_\d \|_s^\Lipg) $,
\eqref{extra piccolezza} and \eqref{ansatz 0}, 
the bound \eqref{piccolo FBR} holds. 
By \eqref{derivata i delta} and using explicit formulae \eqref{L1 L2}-\eqref{T2} we get \eqref{derivata piccolo FBR}.
\end{proof}

The conclusion of this section is that,  after the composition with 
the action-angle variables, the rescaling \eqref{rescaling kdv quadratica}, 
and the transformation $ G_\delta $, the linearized operator 
to analyze is $w \mapsto (\pa_u \gr \mH)(T_\delta) [w] $, $w \in H_S^\bot$, 
up to finite dimensional operators which have the form \eqref{forma buona resto con psi} 
and size \eqref{piccolo FBR}.

\subsection{The linearized operator in the normal directions}
\label{forma-linear-normal}

In view of \eqref{piccolezza resti} we now compute  
$  ( (\pa_u \gr \mH)(T_\delta) [w], w )_{L^2(\T)} $, $ w \in H_S^\bot $, where
$ \mH  = H \circ \Phi_B $ and $\Phi_B $ is the Birkhoff map of Proposition \ref{prop:weak BNF}.
We recall that $\Phi_B(u) = u + \Psi(u)$ where $\Psi$ satisfies \eqref{finito finito} and $\Psi(u) = O(u^3)$.  
It is convenient to estimate separately the terms in 
\be\label{mH H2H3H5}
\mH  = H \circ \Phi_B = H_2 \circ \Phi_B + H_4 \circ \Phi_B + H_{\geq 5} \circ \Phi_B 
\ee
where $ H_2, H_4, H_{\geq 5}$ are defined in \eqref{H iniziale KdV}.

We first consider $ H_{\geq 5} \circ \Phi_B $. 
By \eqref{H iniziale KdV} we get
$ \nabla H_{\geq 5}(u) = \pi_0[ (\partial_u f)(x, u, u_x) ]$ 
$- \pa_x \{ (\pa_{u_x} f)(x, u,u_x) \} $ where
$ \pi_0  $ is the operator defined in \eqref{def:pi0}. 
Since $ \Phi_B $ has the form \eqref{finito finito}, 
Lemma \ref{lemma astratto potente} (at $ u = T_\delta $, see \eqref{T0}) implies that
\begin{align} 
\pa_u \nabla ( H_{\geq 5} \circ \Phi_B ) (T_\delta)  [h] 
& = 
(\pa_u \gr H_{\geq 5})(\Phi_B(T_\delta)) [h] + {\cal R}_{H_{\geq 5}}(T_\delta)[h] 
\notag \\ 
& = \partial_x (r_1(T_\delta) \partial_x h ) + r_0(T_\delta) h + {\cal R}_{H_{\geq 5}}(T_\delta)[h] 
\label{der grad struttura separata5}
\end{align}
where the multiplicative functions $r_0(T_\d)$, $r_1(T_\d)$ are
\begin{align} 
\label{r0r1 def}
r_0 (T_\delta) & := \s_0(\Phi_B(T_\delta)), \quad 
r_1 (T_\delta) := \s_1(\Phi_B(T_\delta)), \\
\s_0(u) & :=  (\pa_{uu} f)(x, u, u_x) - \pa_x \{ (\partial_{u u_x} f)(x, u, u_x) \}, \notag \\
\s_1(u) & := -  (\pa_{u_x u_x} f)(x, u, u_x), \notag
\end{align}
the remainder $ {\cal R}_{H_{\geq 5}}(u) $ has the form \eqref{forma buona resto}
with $\chi_j = e^{\ii jx}$ or $g_j = e^{\ii jx}$ and, using \eqref{resti012}, 
it satisfies, for some $ \sigma := \sigma (\nu, \tau) > 0$, 
\begin{align} \label{2015-1}
& \| g_j \|_s^\Lipg \| \chi_j \|_{s_0}^\Lipg + \| g_j \|_{s_0}^\Lipg \| \chi_j \|_s^\Lipg 
\leq_s \e^5 (1 + \| \fracchi_\d \|_{s+2}^\Lipg ),  \\
& \| \partial_i g_j [\widehat \imath ]\|_s \| \chi_j \|_{s_0} 
+ \| \partial_i g_j [\widehat \imath ]\|_{s_0}  \| \chi_j \|_{s} 
+ \| g_j \|_{s_0} \| \partial_i \chi_j [\widehat \imath ] \|_s 
+ \| g_j \|_{s} \| \partial_i \chi_j [\widehat \imath ]\|_{s_0}  \notag \\ 
& \leq_s  \e^5 ( \| \widehat \imath \|_{s+\s} 
+ \| \fracchi_\d \|_{s+2} \| \widehat \imath \|_{s_0 + 2} ).  \notag
\end{align}
Now we consider the contributions from $ H_2 \circ \Phi_B$ and $H_4 \circ \Phi_B $.
By Lemma \ref{lemma astratto potente} and the expressions of $ H_2, H_4 $ in \eqref{H iniziale KdV} we deduce that
\begin{align} \label{2015-2bis}
\pa_u \nabla ( H_2 \circ \Phi_B) (T_\delta) [h] 
& = - \partial_{xx} h + {\cal R}_{H_2}(T_\delta)[h] \,, \\
\pa_u \nabla ( H_4 \circ \Phi_B) (T_\delta) [h] 
& = -3\varsigma (\Phi_B (T_\delta))^2 h + {\cal R}_{H_4}(T_\delta)[h] \,,
\label{2015-3}
\end{align}
where 
$ {\cal R}_{H_2}(u) $, $ {\cal R}_{H_4}(u) $ have the form \eqref{forma buona resto}. 
By \eqref{resti012}, they have size 
${\cal R}_{H_2}(T_\delta) = O(\e^2)$, ${\cal R}_{H_4}(T_\delta) = O(\e^4)$. 
More precisely, the functions $g_j, \chi_j$ in $\mR_{H_4}(T_\d)$ satisfy the bounds in \eqref{2015-1} 
with $\e^5$ replaced by $\e^4$. 
Regarding $\mR_{H_2}(T_\d)$, we need to find an exact formula for the terms of order $\e^2$. 

The sum of \eqref{der grad struttura separata5}, \eqref{2015-2bis} and \eqref{2015-3} gives a formula for 
$\pa_u \nabla \mH(T_\d)[h]$, where the terms of form \eqref{forma buona resto} and order $\e^2$ 
are confined in $\mR_{H_2}(T_\d)$.
On the other hand, recalling \eqref{widetilde cal H}, $\mH = H_2 + \mH_4 + \mH_{\geq 5}$, 
and $\pa_u \nabla H_2(T_\d) = -\pa_{xx}$, while $\pa_u \nabla \mH_{\geq 5}(T_\d) = O(\e^3)$. 
Therefore all the terms of order $\e^2$ in $\pa_u \nabla \mH(T_\d)$ 
can only come from $\pa_u \nabla \mH_4(T_\d)$. 
Using formula \eqref{H3tilde} for $\mH_4$, we calculate
\[
\Pi_S^\bot \big( \pa_u \nabla \mH_4(T_\d)[h] \big) 
= - 3\varsigma \Pi_S^\bot (T_\d^2 h) \quad \forall h \in H_{S^\bot}^s \,.
\]
Hence all the terms of order $\e^2$ in $\Pi_S^\bot (\pa_u \nabla \mH(T_\d)[h] )$
are contained in the term $- 3\varsigma \Pi_S^\bot (T_\d^2 h)$ 
(and the term $- 3\varsigma \Pi_S^\bot (T_\d^2 h)$ is included in 
$-3\varsigma \Pi_S^\bot[ (\Phi_B (T_\delta))^2 h]$  
because $\Phi_B(T_\d) = T_\d + \Psi(T_\d)$). 
As a consequence, $\Pi_S^\bot \mR_{H_2}(T_\d)$ is of size $O(\e^3)$, 
and its functions $g_j, \chi_j$ (see \eqref{forma buona resto}) satisfy \eqref{2015-1} 
with $\e^5$ replaced by $\e^3$.    

By Lemma \ref{dopo l'approximate inverse} and the results of this section we deduce: 

\begin{proposition}\label{prop:lin}
Assume \eqref{ansatz delta}. 
Then the Hamiltonian operator $ {\cal L}_\om $ has the form, 
$ \forall h \in H_{S^\bot}^s ( \T^{\nu+1}) $,
\be\label{Lom KdVnew}
{\cal L}_\om h  :=   \mD_\om h - \partial_x K_{02} h   
= \Pi_S^\bot \big( \mD_\om h + \partial_{xx} (a_1 \partial_x h) 
+ \pa_x ( a_0 h ) - \partial_x \mR_* h \big)  
\ee
where $ {\mR}_* := \mR_{H_2}(T_\d) + \mR_{H_4}(T_\d) + \mR_{H_{\geq 5}}(T_\d) + R(\psi) $ 
(with $R(\psi)$  defined in Lemma \ref{dopo l'approximate inverse}, 
and $\mR_{H_2}(T_\d)$, $\mR_{H_4}(T_\d)$, $\mR_{H_{\geq 5}}(T_\d)$ defined in 
\eqref{der grad struttura separata5}, \eqref{2015-2bis}, \eqref{2015-3}), 
the functions 
\begin{equation}\label{a1p1p2}
a_1  := 1 - r_1 ( T_\delta ) \, , \qquad 
a_0 := 3\varsigma (\Phi_B(T_\d))^2 - r_0(T_\d) \,, \qquad 
\end{equation}
$ r_0, r_1 $ are defined in \eqref{r0r1 def}, 
and $ T_\delta $ in \eqref{T0}.
They satisfy 
\begin{align} 
\| a_1 -1 \|_s^\Lipg + \| a_0 - 3\varsigma T_\d^2 \|_s^\Lipg 
& \leq_s \e^3 ( 1 + \| {\mathfrak I}_\d \|_{s+\s}^\Lipg ) \,,  
\label{n 10}
\\
\| \pa_i a_1[\widehat \imath ] \|_s + \| \pa_i (a_0 - 3\varsigma T_\d^2) [\widehat \imath ] \|_s  
& \leq_s \e^3 ( \| \widehat \imath \|_{s+\s} + \| {\mathfrak I}_\delta \|_{s+\s} 
\| \widehat \imath \|_{s_0+\s} ) 
\label{n 11}  
\end{align}
where $ \fracchi_\d (\ph) := (\theta_0(\ph) - \ph, y_\d(\ph), z_0(\ph)) $ corresponds to $T_\delta$. 
The remainder 
$ {\cal R}_* $ has the form \eqref{forma buona resto}, 
and its coefficients $g_j, \chi_j$ satisfy bounds \eqref{piccolo FBR}-\eqref{derivata piccolo FBR}.
\end{proposition}

\begin{remark} \label{rem:lm M2-3}
For $ K = H + \lm M^2$, $ \lambda = 3 \varsigma / 4 $,  
the coefficient $a_0$ in \eqref{a1p1p2} 
becomes 
\[
a_0 = 3 \varsigma \pi_0 \big[ (\Phi_B (T_\d))^2 \big] - r_0(T_\d),
\]
where $ \pi_0$ is defined in \eqref{def:pi0}. 
Thus the space average of $ a_0$ has size $O(\e^3)$. 
\end{remark}

Bound \eqref{piccolo FBR} imply, by Lemma \ref{remark : decay forma buona resto}, estimates 
for the $ s $-decay norms of ${\cal R}_*$.   
The linearized operator $ {\cal L}_\om := {\cal L}_\om (\om, i_\d (\om))$ 
depends on the parameter $ \om $ both directly  and also through the dependence on the 
torus $i_\d (\om )$.
We have estimated also the partial derivative $ \pa_i $ with respect to  the variables $ i $ (see \eqref{embedded torus i}) in order to control, along the nonlinear Nash-Moser iteration, 
the Lipschitz variation of the eigenvalues of $ {\cal L}_\om $ with respect to $ \om $
and the approximate solution $ i_\d $.

\section{Reduction of the linearized operator in the normal directions}
\label{operatore linearizzato sui siti normali}

The goal of this section is to conjugate the Hamiltonian linear operator $ {\cal L}_\om $ in \eqref{Lom KdVnew} 
to the constant coefficients linear operator $ {\cal L}_\infty $ defined in \eqref{Lfinale}. 
The proof is obtained applying different kind of symplectic transformations. 
We shall always assume \eqref{ansatz delta}.

\subsection{Space reduction at the order $ \pa_{xxx} $ }\label{step1}

As a first step, we symplectically conjugate the operator 
$ {\cal L}_\om $ in \eqref{Lom KdVnew} to  $ {\cal L}_1 $ in 
\eqref{cal L1 Kdv},  which has the coefficient of $\partial_{xxx}$ independent on the space variable. 
Because of the Hamiltonian structure, this step also eliminates the terms $ O( \pa_{xx} )$.
 
We look for a  $ \vphi $-dependent family of symplectic diffeomorphisms $\Phi (\vphi) $ of $ H_S^\bot $ 
which differ from 
\begin{equation}\label{primo cambio di variabile modi normali}
{\cal A}_{\bot} := \Pi_S^\bot {\cal A} \Pi_S^\bot  \, ,  \quad  
({\cal A} h)(\vphi,x) := (1 + \beta_x(\vphi,x)) h(\vphi,x + \beta(\vphi,x)) \, , 
\end{equation}
up to a small ``finite dimensional'' remainder, see  \eqref{forma buona resto cambio di variabile hamiltoniano}.
For each $ \varphi \in \T^\nu $, the map $ {\cal A}(\vphi) $ is 
a symplectic map of the phase space, see Remark 3.3 in \cite{BBM}. 
If  $ \| \b \|_{W^{1,\infty}} < 1/2$, then $ {\cal A} $ is invertible (see Lemma \ref{lemma:utile}), 
and its inverse and  adjoint maps are 
\begin{equation}\label{cambio di variabile inverso}
\begin{aligned}
({\cal A}^{-1} h)(\vphi,y) 
& = (1 + \tilde{\beta}_y(\vphi,y)) h(\vphi, y + \tilde{\beta}(\vphi,y)) \,, 
\\
({\cal A}^T h) (\vphi,y) 
& = h(\vphi, y + \tilde{\beta}(\vphi,y)) 
\end{aligned}
\end{equation}
where
$ x = y + \tilde{\b} (\vphi, y) $ is the inverse diffeomorphism (of $\T$) of $ y = x + \b (\vphi, x) $. 

The restricted map $ {\cal A}_\bot (\vphi): H_S^\bot \to H_S^\bot $ is not symplectic.  
We have already observed in the introduction that
$ {\cal A }(\vphi ) $  
is the time-$1$ flow map of the linear Hamiltonian PDE \eqref{transport-free}.
The equation \eqref{transport-free} is a linear transport equation, 
whose charactheristic curves are the solutions of the ODE 
$$
\frac{d}{d\tau} x  = - b(\vphi, \tau, x) \, .
$$
To obtain a symplectic transformation close to $\mA_\bot$, 
we define a symplectic map $\Phi $ of $ H_S^\bot $  as the time 1 flow of the Hamiltonian PDE \eqref{problemi di cauchy}.
The linear operator $ \Pi_S^\bot \partial_x (b(\tau, x) u) $ is the Hamiltonian vector field generated by 
the quadratic Hamiltonian $ \frac12 \int_{\T} b(\tau, x) u^2 dx  $ restricted to $ H_S^\bot $. 
The flow of \eqref{problemi di cauchy}  
is well defined in the Sobolev spaces $ H^s_{S^\bot} (\T_x) $ for $ b(\vphi, \tau, x) $ smooth enough, 
by standard theory of linear hyperbolic PDEs (see e.g.\ section 0.8 in \cite{Taylor}). 
The difference between the time 1 flow map $ \Phi $ and $ {\cal A}_\bot $ is a  ``finite-dimensional'' remainder
of size $O(\b)$. 

\begin{lemma}[Lemma 8.1 of \cite{bbm}]
\label{modifica simplettica cambio di variabile}
For $ \| \beta \|_{W^{s_0 + 1,\infty}} $ small, 
there exists an invertible symplectic transformation 
$\Phi = {\cal A}_\bot + {\cal R}_\Phi$ of $H_{S^\bot}^s$, 
where $ {\cal A}_\bot $ is defined in \eqref{primo cambio di variabile modi normali} 
and $ {\cal R}_\Phi $ is a ``finite-dimensional'' remainder 
\begin{equation}\label{forma buona resto cambio di variabile hamiltoniano}
{\cal R}_\Phi h= \sum_{j \in S} \int_0^1 (h, g_j (\tau)  )_{L^2(\T)} \, \chi_j (\tau) \, d \tau 
+ \sum_{j \in S} \big(h, \psi_j \big)_{L^2(\T)} e^{\ii j x}
\end{equation}
for some functions $ \chi_j (\tau), g_j (\tau) , \psi_j \in H^s $ satisfying 
for all $\t \in [0,1]$
\begin{equation}\label{stime forma buona resto cambio di variabile hamiltoniano}
\| \psi_j\|_s + \| g_j(\tau)\|_s \leq_s \| \beta\|_{W^{s + 2, \infty}}\,,
\quad \| \chi_j(\tau)\|_s \leq_s  1 + \| \beta \|_{W^{s + 1, \infty}} \,.
\end{equation}
Moreover 
\begin{equation}\label{stime Phi Phi -1}
\| \Phi h \|_s + \| \Phi^{-1} h \|_s 
\leq_s \| h \|_s + \| \beta \|_{W^{s + 2, \infty}} \| h \|_{s_0}  
\quad \forall h \in H^s_{S^\bot} \,. 
\end{equation}
\end{lemma}

We conjugate  $ {\cal L}_\om $  in  \eqref{Lom KdVnew} 
via the symplectic map $ \Phi = {\cal A}_\bot + {\cal R}_\Phi $
of Lemma \ref{modifica simplettica cambio di variabile}. 
Using the splitting $ \Pi_S^{\bot} = I - \Pi_S $, we compute 
\begin{equation}\label{LAbot}
{\cal L}_\om \Phi =  \Phi {\cal D}_\omega + 
\Pi_S^\bot {\cal A} \big( 
b_3 \partial_{yyy}  + b_2 \partial_{yy}   + b_1 \partial_{y}  + b_0   \big) \Pi_S^\bot +  {\cal R}_I \,,
\end{equation}
where the coefficients $b_i(\ph,y)$, $i=0,1,2,3$, are
\begin{align} \label{step1: b3KdV} 
b_3
& := {\cal A}^T [ a_1 ( 1 + \b_x)^3 ], 
\quad 
b_2 := {\cal A}^T \big[ 2 (a_1)_x (1 + \beta_x )^2 + 6 a_1  \beta_{xx} (1 + \beta_x )\big],  
\\ 
b_1
& := {\cal A}^T \big[ ({\cal D}_\om \beta) + \frac{3 a_1 \beta_{xx}^2 }{1 + \beta_x} 
+ 4 a_1 \beta_{xxx} + 6 (a_1)_x \beta_{xx} 
+ ((a_1)_{xx} + a_0) (1 + \beta_x) \big],  
\notag \\ 
b_0
& := {\cal A}^T \Big[ \frac{1}{1 + \b_x} 
\Big( \mD_\om \beta_x + a_1 \beta_{xxxx} + 2 (a_1)_{x} \beta_{xxx} 
+ ((a_1)_{xx} + a_0) \b_{xx} \Big)
+ (a_0)_x \Big], 
\notag
\end{align}
and the remainder
\begin{align}
{\cal R}_I := {} & 
- \Pi_S^\bot  \big( a_1 \partial_{xxx} 
+ 2 (a_1)_x \partial_{xx}  + ( (a_{1})_{xx} + a_0)\partial_x  
+ (a_0)_x \big)  \Pi_{S} {\cal A}  \Pi_S^\bot \,  \nonumber\\
& - \Pi_S^\bot \partial_x \mR_{*} {\cal A}_\bot  
+ [{\cal D}_\omega, {\cal R}_\Phi] + ({\cal L}_\omega - {\cal D}_\omega) {\cal R}_\Phi\,. 
\label{calR1KdV}
\end{align}
The commutator $[{\cal D}_\omega, {\cal R}_\Phi] $ has the form \eqref{forma buona resto cambio di variabile hamiltoniano} with ${\cal D}_\om g_j$ or ${\cal D}_\om \chi_j$, ${\cal D}_\om \psi_j$ instead of $\chi_j$, $g_j$, $\psi_j$ respectively. Also the last term $({\cal L}_\omega - {\cal D}_\omega) {\cal R}_\Phi$ in \eqref{calR1KdV} has the form \eqref{forma buona resto cambio di variabile hamiltoniano} (note that ${\cal L}_\omega - {\cal D}_\omega$ does not contain derivatives with respect to $\vphi$). By \eqref{LAbot}, and decomposing $ I = \Pi_S + \Pi_S^\bot $, we get 
\begin{align} \label{L A bot finaleKdV}
{\cal L}_\om \Phi 
= {} & \Phi  ( {\cal D}_\omega  +  b_3 \partial_{yyy}  + b_2 \partial_{yy}   + b_1 \partial_{y}  + b_0 ) \Pi_S^\bot
+ {\cal R}_ {II} \,,
\\
\label{calR2KdV}
{\cal R}_{II}  
:= {} &  \{\Pi_S^\bot ({\cal A} - I) \Pi_{S} - {\cal R}_\Phi \} 
( b_3 \partial_{yyy}  + b_2 \partial_{yy}   + b_1 \partial_{y}  + b_0 ) \Pi_S^\bot   + {\cal R}_I \,. 
\end{align}
Now we choose the function $ \beta = \b (\vphi, x)  $  such that 
\begin{equation}\label{choice beta}
a_1(\vphi, x) (1 + \b_x (\vphi, x))^3 = b_3 (\vphi) 
\end{equation}
so that the coefficient $ b_3 $ in \eqref{step1: b3KdV} depends only on $ \vphi $
(note that $ {\cal A}^T [b_3 (\vphi)]$ $= b_3 (\vphi)$). 
The only solution of \eqref{choice beta} with zero space average is (see e.g.\ \cite{BBM}-section 3.1)
$\b := \partial_x^{-1} \rho_0$, 
where $\rho_0 := b_3 (\vphi)^{1/3} (a_1 (\vphi, x))^{-1/3} - 1$, and 
\begin{equation}\label{sol beta}
b_3 (\vphi) = \Big( \frac{1}{2 \pi} \int_{\T} (a_1 (\vphi, x))^{-1/3} dx \Big)^{-3}. 
\end{equation}
Applying  the symplectic map  $ \Phi^{-1} $ in \eqref{L A bot finaleKdV} 
we  obtain the Hamiltonian operator (see Definition \ref{operatore Hamiltoniano})
\begin{equation}\label{cal L1 Kdv}
{\cal L}_1 := \Phi^{-1} {\cal L}_\om \Phi
= \Pi_S^\bot \big( \om \cdot \partial_\vphi + b_3(\vphi) \partial_{yyy} 
+ b_1 \partial_y + b_0 \big) \Pi_S^\bot +  {\mathfrak R}_1 
\end{equation}
where $ {\mathfrak R}_1 := \Phi^{-1} {\cal R}_{II} $. 
Note that the term $b_2 \pa_{yy}$ has disappeared from \eqref{cal L1 Kdv} because, 
by the Hamiltonian nature of $ {\cal L}_1 $, the coefficient $ b_2 = 2 (b_3)_y $ (see \cite{BBM}-Remark 3.5) 
and therefore, by \eqref{sol beta}, $ b_2  =  2 (b_3)_y = 0 $. 

\begin{lemma}[Lemma 8.2 of \cite{bbm}]  \label{cal R3}
The operator $ {\mathfrak R}_1 $ in \eqref{cal L1 Kdv} has the form \eqref{forma buona con gli integrali}.
\end{lemma}

Since $a_1 = 1 + O(\e^3)$ and $a_0 = 3\varsigma T_\d^2 + O(\e^3)$ (see \eqref{n 10}-\eqref{n 11} for the precise estimates), by the usual composition estimates we deduce the following lemma. 

\begin{lemma} \label{lemma:stime coeff mL1}
There is $\s = \s(\t,\nu) > 0$ such that 
\begin{align} 
& \| \b \|_s^\Lipg + \| b_3 -1 \|_s^\Lipg 
+ \| b_1 - 3 \varsigma T_\d^2 \|_s^\Lipg + \| b_0 - 3 \varsigma (T_\d^2)_x \|_s^\Lipg 
\notag \\ 
& \leq_s \e^3 (1 + \| \fracchi_\d \|_{s+\s}^\Lipg), 
\label{n 1}
\\
& \| \pa_i \b [\widehat \imath] \|_s 
+ \| \pa_i b_3 [\widehat \imath] \|_s 
+ \| \pa_i(b_1 - 3 \varsigma T_\d^2) [\widehat \imath] \|_s
+ \| \pa_i(b_0 - 3 \varsigma (T_\d^2)_x) [\widehat \imath] \|_s 
\notag \\ 
& \leq_s \e^3 \big( \| \widehat \imath \|_{s+\s} + \| {\mathfrak I}_\d \|_{s+\s} \| \widehat \imath \|_{s_0+\s} \big), 
\label{n 2}
\end{align}
where $T_\d$ is defined in \eqref{T0}. 
The transformations $\Phi$, $\Phi^{-1}$ satisfy
\begin{align}
\label{stima cal A bot}
\|\Phi h \|_s^{\Lipg} + \|\Phi^{-1} h \|_s^{\Lipg}
& \leq_s \| h \|_{s + 1}^{\Lipg} + \| {\mathfrak I}_\delta \|_{s + \sigma}^{\Lipg} \| h \|_{s_0 + 1}^{\Lipg} 
\\
\label{stima derivata cal A bot}
\| \pa_i (\Phi h) [\widehat \imath] \|_s + \| \pa_i (\Phi^{-1} h) [\widehat \imath] \|_s 
& \leq_s 
\| h \|_{s + \sigma} \| \widehat \imath \|_{s_0 + \sigma} + 
\| h\|_{s_0 + \sigma} \| \widehat \imath \|_{s + \sigma} 
\\ & \qquad 
+ \| {\mathfrak I}_\delta\|_{s + \sigma} \| h\|_{s_0 + \sigma} \| \widehat \imath \|_{s_0 + \sigma}\,.
\notag 
\end{align}
Moreover the remainder ${\mathfrak R}_1$ has the form \eqref{forma buona con gli integrali},
where the functions $\chi_j(\tau)$, $g_j(\tau)$ satisfy the estimates 
\eqref{piccolo FBR}-\eqref{derivata piccolo FBR}  
uniformly in $\tau \in [0, 1]$.
\end{lemma}

\subsection{Time reduction at the order $ \pa_{xxx}$ }\label{step2}

The goal of this section is to get a constant coefficient in front of 
$ \partial_{yyy} $, using a quasi-periodic reparametrization of time. 
We consider the change of variable 
\begin{equation} \label{def B}
(B w)(\vphi, y) := w(\vphi + \omega \alpha(\vphi), y), 
\qquad  
( B^{-1} h)(\vartheta  , y ) := h(\vartheta + \omega \tilde{\alpha}(\vartheta), y)\,,
\end{equation}
where $\T^\nu \to \T^\nu$, $\vartheta \mapsto \vphi = \vartheta + \omega \tilde{\alpha}(\vartheta )$ 
is the inverse diffeomorphism of $ \vartheta =  \vphi + \omega \alpha(\vphi) $ in $\T^\nu$.
By conjugation, the differential operators become 
\begin{equation}  \label{anche def rho}
B^{-1} \om \cdot \partial_\ph B = \rho(\vartheta)\,   \omega \cdot \partial_{\vartheta} ,
\quad 
B^{-1}  \partial_y B = \partial_y, 
\quad 
\rho := B^{-1} (1 + \omega \cdot \partial_{\varphi} \a).
\end{equation}
By \eqref{cal L1 Kdv}, using also that $ B $ and $ B^{-1} $ commute with $ \Pi_S^\bot $, 
the conjugate operator $B^{-1} {\cal L}_1 B$ is equal to
\begin{equation} \label{secondo coniugio siti normali Kdv}
\Pi_S^\bot [ \rho  \omega \cdot \partial_{\vartheta} 
 + (B^{-1} b_3)  \partial_{yyy}  
 + ( B^{-1}  b_1 ) \partial_y + ( B^{-1} b_0 ) ] \Pi_S^\bot 
 + B^{-1} {\mathfrak R}_1 B.
\end{equation}
We choose $ \alpha $ such that 
$(B^{-1}b_3 )(\vartheta ) = m_3 \rho (\vartheta)$ for some constant $m_3 \in \R$, 
namely 
\begin{equation}\label{B3solu}
b_3 (\vphi) = m_3 ( 1 + \om \cdot \partial_\vphi \a (\vphi) ) 
\end{equation}
(recall  \eqref{anche def rho}). 
The unique solution with zero average of \eqref{B3solu} is
\begin{equation}  \label{def alpha m3}
\a (\vphi) := \frac{1}{m_3} ( \om \cdot \partial_\vphi )^{-1} ( b_3 - m_3 ) (\vphi) , 
\qquad 
m_3 := \frac{1}{(2 \pi)^\nu} \int_{\T^\nu} b_3 (\vphi) d \vphi  \,. 
\end{equation}
Hence, by \eqref{secondo coniugio siti normali Kdv}, 
\begin{alignat}{2} \label{L2 Kdv}
& B^{-1} {\cal L}_1 B = \rho {\cal L}_2, &
& {\cal L}_2 := \Pi_S^\bot ( \omega \cdot \partial_{\vartheta} + m_3 \partial_{yyy} 
+ c_1 \partial_y + c_0 ) \Pi_S^\bot +  {\mathfrak R}_2
\\
\label{tilde c1 Kdv}
& c_1 := \rho^{-1} (B^{-1} b_1 ), \quad &
& c_0 :=  \rho^{-1} (B^{-1}  b_0 ), \quad 
{\mathfrak R}_2 := \rho^{-1} B^{-1}{\mathfrak R}_1 B \, .
\end{alignat}
The transformed operator ${\cal L}_2$ in \eqref{L2 Kdv} is still Hamiltonian, 
because the re\-para\-metrization of time preserves the Hamiltonian structure 
(see Section 2.2 and Remark 3.7 in \cite{BBM}).

\begin{lemma} \label{lemma:stime coeff mL2}
There is $ \s = \s(\nu,\t) > 0 $ (possibly larger than $ \s $ in Lemma \ref{lemma:stime coeff mL1}) such that 
\begin{align} \label{stima m3}
| m_3 - 1 |^\Lipg 
& \leq  C \e^3, \qquad  
| \pa_i m_3 [\widehat \imath ]| 
\leq C \e^3 \| \widehat \imath \|_{s_0 + \s}  
\\
\notag  
\| \a \|_s^\Lipg  
& \leq_s \e^3 \g^{-1} (1 + \| \fracchi_\d \|_{s + \s}^\Lipg ) 
\\
\notag 
\| \pa_i \a [\widehat \imath] \|_s
& \leq_s \e^3 \g^{-1} ( \| \widehat \imath \|_{s+\s} 
+ \| {\mathfrak I}_\d \|_{s+\s} \| \widehat \imath \|_{s_0+\s} )
\\
\notag
\| \rho -1 \|_s^\Lipg 
& \leq_s \e^3 (1 + \| \fracchi_\d \|_{s+\s}^\Lipg )
\\
\notag 
\| \pa_i \rho [\widehat \imath] \|_s 
& \leq_s \e^3 ( \| \widehat \imath \|_{s+\s} + \| \fracchi_\d \|_{s+\s} \| \widehat \imath \|_{s_0 +\s} ) 
\end{align}
\vspace{-23pt}
\begin{align}
\label{n 5}
& \| c_1 - 3\varsigma T_\d^2 \|_s^\Lipg + \| c_0 - 3\varsigma (T_\d^2)_x \|_s^\Lipg 
\leq_s \e^5 \g^{-1} (1 + \| \fracchi_\d \|_{s+\s}^\Lipg ), 
\\
\notag 
& \| \pa_i (c_1 - 3 \varsigma T_\d^2) [\widehat \imath] \|_s 
+ \| \pa_i (c_0 - 3 \varsigma (T_\d^2)_x) [\widehat \imath] \|_s 
\\ 
\notag & \hspace{57mm} 
\leq_s \e^5 \g^{-1} ( \| \widehat \imath \|_{s+\s} 
+ \| \fracchi_\d \|_{s+\s} \| \widehat \imath \|_{s_0 +\s} ). 
\end{align}
The transformations $B$, $B^{-1}$ satisfy the estimates \eqref{stima cal A bot}, \eqref{stima derivata cal A bot}.
The remainder $ \mathfrak{R}_2 $ has the form \eqref{forma buona con gli integrali}, and the functions $g_j(\tau)$, $\chi_j(\tau)$ satisfy the estimates \eqref{piccolo FBR}-\eqref{derivata piccolo FBR} for all $\tau \in [0,1]$. 
\end{lemma}

\begin{proof} To estimate $\| \a \|_s^\Lipg$ we also differentiate 
\eqref{def alpha m3} with respect to the parameter $ \om $.  
Note that $c_1 - 3 \varsigma B^{-1}(T_\d^2) = O(\e^3)$, 
and similarly $c_0 - 3 \varsigma B^{-1}((T_\d^2)_x)$ $= O(\e^3)$. 
The factor $\e^5 \g^{-1}$ in the last two inequalities 
comes from the estimate of the difference 
$B^{-1}(T_\d^2) - T_\d^2 \simeq (T_\d^2)_\ph \a = O(\e^2 \e^3 \g^{-1})$.
\end{proof}

\subsection{Translation of the space variable}\label{step3}

In this section we remove the space average from the coefficient in front of $ \partial_y $. 
Consider the change of the space variable $ z = y + p(\vartheta ) $ which induces on 
$ H^s_{S^\bot} (\T^{\nu+1}) $ the operators
\begin{equation}\label{gran tau}
({\cal T} w)(\vartheta, y ) := w(\vartheta , y + p(\vartheta)) \, , \quad 
({\cal T}^{-1} h) (\vartheta ,z ) = h(\vartheta, z - p(\vartheta))
\end{equation}
(which are a particular case of those used in section \ref{step1}). 
The differential operators become
$ {\cal T}^{-1} \omega \cdot \partial_{\vartheta} {\cal T} $ 
$ =  \omega \cdot \partial_{\vartheta} 
+ \{ \omega \cdot \partial_{\vartheta}p (\vartheta) \} \partial_z $, 
$ {\cal T}^{-1} \partial_{y} {\cal T}  =  \partial_{z} $. 
Since $\mT, \mT^{-1}$ commute with $ \Pi_S^\bot $, we get
\begin{align} \label{L3 KdV}
\mL_3 & := {\cal T}^{-1}{\cal L}_2 {\cal T} 
= \Pi_S^\bot (\omega \cdot \partial_{\vartheta} + m_3 \partial_{zzz} + d_1 \partial_z + d_0 ) \Pi_S^\bot 
+ {\mathfrak R}_3 \,,
\\
\label{d1d0R3}
d_1 & := ({\cal T}^{-1} c_1)   + \omega \cdot \partial_{\vartheta} p  \, , 
\qquad 
d_0 :=  {\cal T}^{-1}  c_0   \, , 
\qquad 
{\mathfrak R}_3 :=   {\cal T}^{-1} {\mathfrak R}_2 {\cal T}.
\end{align}
We choose 
\begin{equation} \label{def m1 p Kdv}
m_1 :=  \frac{1}{(2\pi)^{\nu + 1}} \int_{\T^{\nu + 1}}   c_1 d\vartheta dy \, , \quad 
p := (\omega \cdot \partial_\vartheta)^{-1} 
\Big( m_1 -  \frac{1}{2 \pi} \int_{\T}  c_1  d y \Big) \, ,
\end{equation}
so that
\begin{equation} \label{n 19}
\frac{1}{2 \pi} \int_{\T} d_1 (\vartheta, z) \, dz = m_1 \quad \forall \vartheta \in \T^\nu. 
\end{equation}
Recalling \eqref{n 5}, we analyze the space average of $c_1$ in more detail. 
To avoid ambiguity between the space variable $y \in \T$ and the action $y_\d : \T^\nu \to \R^\nu$ of \eqref{T0}, 
we rename $x \in \T$ the space variable, and $\ph \in \T^\nu$ the variable on the torus (time variable). 
Let 
\begin{equation}\label{def bar pi}
{\bar v} (\vphi, x) := {\mathop \sum}_{j \in S} \sqrt{\xi_j} e^{\ii \ell (j) \cdot \vphi } e^{\ii j x}, 
\end{equation}
where $\ell : S \to \Z^\nu$ is the odd injective map (see \eqref{tang sites})
\be\label{def:function-ell}
\ell(\bar \jmath_i) := e_i \, , \quad 
\ell(-\bar \jmath_i) := - e_i \,  , \quad i = 1,\ldots,\nu
\ee
and $e_i = (0,\ldots,1, \ldots,0)$ denotes the $i$-th vector of the canonical basis of $\R^\nu$.
In view of the next linear Birkhoff normal form step (whose goal is to normalize the term of size $\e^2$), 
we observe that the component of order $\e^2$ in $T_\d^2$ (see \eqref{T0}) is $\e^2 \bar v^2$, 
with
\begin{equation} \label{n 18}
\begin{aligned}
\| T_\d^2 - \e^2 \bar v^2 \|_s^\Lipg 
& \leq_s \e^2 \| \fracchi_\d \|_{s+\s}^\Lipg \,, \\
\| \pa_i (T_\d^2 - \e^2 \bar v^2)[\widehat \imath\,] \|_s 
& \leq_s \e^2 ( \| \widehat \imath \|_{s+\s} + \| \fracchi_\d \|_{s+\s} \| \widehat \imath \|_{s_0+\s} ) \,.
\end{aligned}
\end{equation}
Moreover, from \eqref{T0}, since $(v_\d, z_0)_{L^2(\T)} = 0$, 
and $(\theta_0)_{-j} = - (\theta_0)_j$ for all $j \in S$, 
we have 
\[ 
\int_\T T_\d^2 \, dx 
= \e^2 \int_\T v_\d^2 \, dx + \e^{2b} \int_\T z_0^2 \, dx
= \e^2 \sum_{j \in S} \xi_j + \e^{2b} \sum_{j \in S} |j| (y_\d)_j + \e^{2b} \int_\T z_0^2 \, dx.
\] 
We define 
\begin{equation} \label{n 20}
\widetilde d_1 := d_1 - 3\varsigma \e^2 \bar v^2, \quad 
\widetilde d_0 := d_0 - 3\varsigma \e^2 (\bar v^2)_x, 
\end{equation}
and note that, by \eqref{n 19} and \eqref{def bar pi},  
\begin{equation} \label{c(xi)}
\frac{1}{2\pi} \int_\T \tilde d_1 \, dx 
= m_1 - \frac{3\varsigma \e^2}{2\p} \int_\T \bar v^2 \, dx 
= m_1 - \e^2 c(\xi), \quad  
c(\xi) := 3\varsigma \sum_{j \in S} \xi_j \,.
\end{equation}
Using the explicit formulae above, and Lemma \ref{remark : decay forma buona resto} for the estimate of $\mathfrak R_3$, we get the following bounds.

\begin{lemma} \label{lemma:stime coeff mL3}
There is $ \s := \s(\nu,\t) > 0 $ (possibly larger than in Lemma \ref{lemma:stime coeff mL2}) 
such that 
\begin{align} \label{stima m1}
| m_1 - \e^2 c(\xi) |^\Lipg 
& \leq C \e^5 \g^{-1}, 
\quad 
| \partial_i m_1 [\widehat \imath ]| 
\leq C \e^{2b} \| \widehat{\imath} \|_{s_0 + \sigma} 
\\ 
\notag  
\| p \|_s^\Lipg
& \leq_s \e^5 \g^{-2} + \| \fracchi_\d \|_{s + \s}^\Lipg\,, 
\\
\notag 
\| \partial_i p [\widehat \imath] \|_s
& \leq_s \| \widehat \imath \|_{s+\s} 
+ \e^5 \g^{-2} \| {\mathfrak I}_\d \|_{s+\s} \| \widehat \imath \|_{s_0+\s} \,, 
\\ 
\label{tilde d1 d0 KdV}
\| \widetilde d_k  \|_s^\Lipg
& \leq_s \e^7 \g^{-2} + \e^2 \| \fracchi_\d \|_{s + \s}^\Lipg \,, 
\quad k = 0,1,
\\
\notag 
\| \partial_i \widetilde d_k [\widehat \imath] \|_s 
& \leq_s \e^5 \g^{-1} (\| \widehat \imath \|_{s + \s} 
+ \| {\mathfrak I}_\d \|_{s + \sigma} \| \widehat \imath \|_{s_0 + \s} ) \,, 
\quad k=0,1.
\end{align}
The matrix $ s $-decay norm (see \eqref{matrix decay norm}) of the operator ${\mathfrak R}_3$ satisfies 
\begin{equation} \label{nuove R3}
\begin{aligned}
|{\mathfrak R}_3|_s^{\Lipg} 
& \leq_s \e^{1+b} \| {\mathfrak I}_\d \|_{s+\sigma}^\Lipg \,, \\
| \pa_i {\mathfrak R}_3 [\widehat \imath] |_s 
& \leq_s \e^{1+b} ( \| \widehat \imath \|_{s+\s} + \| {\mathfrak I}_\d \|_{s+\s} \| \widehat \imath \|_{s_0+\s} )\,.
\end{aligned}
\end{equation}
The transformations ${\cal T}$, ${\cal T}^{-1}$ satisfy \eqref{stima cal A bot}, \eqref{stima derivata cal A bot}.
\end{lemma}

\begin{remark} \label{rem:lm M2-4}
When $ K = H + \lambda M^2 $, $ \lambda = 3 / 4 $, 
the constant coefficient $m_1$ in \eqref{def m1 p Kdv} becomes of size 
\begin{equation} \label{m1 IF}
| m_1 |^\Lipg \leq C \e^5 \g^{-1}.
\end{equation}
The inequality \eqref{m1 IF} is the key difference between the cases $H + (3\varsigma/4) M^2$ and $H$
(compare \eqref{m1 IF} with \eqref{stima m1}, where $m_1$ contains 
the non-perturbative term $\e^2 c(\xi)$). 
\end{remark}

It is sufficient to estimate $ \mathfrak R_3 $ (which has the form \eqref{forma buona con gli integrali})
only in the $ s $-decay norm (see \eqref{nuove R3}) 
because the next transformations will preserve it. 
Such norms will be used in the reducibility scheme of section \ref{subsec:mL0 mL5}.

\subsection{Linear Birkhoff normal form} \label{BNF:step1}

Now we normalize the terms of order $ \e^2 $ of $ {\cal L}_3 $.
This step is different from the reducibility steps 
that we shall perform in section \ref{subsec:mL0 mL5}:  
the diophantine constant $\g$ in \eqref{omdio} is $ \g = o(\e^2 ) $, 
and therefore the terms of order $ \e^2 $ are not perturbative, 
because $\e^2 \g^{-1}$ is not small (in fact, it is big). 
The reduction of this section is possible thanks to the special form of the term 
$ \e^2 {\cal B} $ defined in \eqref{def cal B1 B2}: 
the harmonics of $ \e^2 \mB $  
corresponding to a possible small divisor
are naught, except $\mB_j^j(0)$, see  Lemma \ref{lem:bnf 1}.
Note that, since the previous linear transformations $ \Phi $, $ B $, $ {\cal T} $ are 
$ O(\e^5 \g^{-2} ) $-close to the identity, 
the terms of order $ \e^2 $ in $ {\cal L}_3 $ are the same as in the original linearized operator.

First, we collect all the terms of order $ \e^2 $ in the operator $ {\cal L}_3 $ in \eqref{L3 KdV}. 
We have 
$$
\mL_3 = \Pi_S^\bot ( \om  \cdot \pa_\vphi
 + m_3 \partial_{xxx} + \e^2 {\cal B} 
 + {\tilde d}_1 \partial_x + {\tilde d}_0 ) \Pi_S^\bot 
 + {\mathfrak R}_3 
$$
where $ \widetilde d_1, \widetilde d_0, {\mathfrak R}_3 $ are defined in 
\eqref{n 20}, \eqref{d1d0R3} and (recall \eqref{def bar pi})
\begin{equation} \label{def cal B1 B2}
{\cal B} h := 3 \varsigma \bar v^2 \pa_x h + 3 \varsigma (\bar v^2)_x h 
= \pa_x (3 \varsigma \bar v^2 h ).
\end{equation}
Note that ${\cal B}$ is the linear Hamiltonian vector field of $ H_{S}^\bot $
generated by the Hamiltonian $ z \mapsto \frac{3\varsigma}{2} \int_\T \bar v^2 z^2 \, dx $.

We transform $ {\cal L}_3 $ by a symplectic operator 
$ \Phi_2 : H_{S^\bot}^s(\T^{\nu + 1}) \to H_{S^\bot}^s(\T^{\nu + 1}) $ of the form 
\begin{equation}\label{Phi_1}
\Phi_2 := {\rm exp}(\e^2 A) = I_{H_S^\bot} + \e^2 A + \e^4 \widehat A, \quad 
\widehat A := \sum_{k \geq 2} \frac{\e^{2(k-2)}}{k!} A^k \,,
\end{equation}
where $ A(\vphi) h = {\mathop \sum}_{j,j' \in S^c} A_j^{j'}(\vphi) h_{j'} e^{\ii j x} $
is a Hamiltonian vector field.
The map $ \Phi_2 $ is symplectic, because it is the time 1 flow of a Hamiltonian vector field. 
We calculate 
\begin{multline} \label{L3 new KdV}
{\cal L}_3 \Phi_2 - \Phi_2 \Pi_S^\bot ( {\cal D}_\om + m_3 \partial_{xxx} ) \Pi_S^\bot \\
= \e^2 \Pi_S^\bot \{ \mB + ({\cal D}_\om A) + m_3 [\partial_{xxx}, A] \} \Pi_S^\bot  
+ \Pi_S^\bot \tilde d_1 \pa_x \Pi_S^\bot + R_3
\end{multline}
where 
\begin{align} \label{def R3}
R_3 := \, & \e^4 \Pi_S^\bot \{ (\mD_\om \widehat A) + m_3 [\pa_{xxx}, \widehat A] 
+ \mB (A + \e^2 \widehat A) \} \Pi_S^\bot \\
\notag & + \Pi_S^\bot \tilde d_1 \pa_x \Pi_S^\bot (\Phi_2 - I)
+ (\Pi_S^\bot \tilde d_0 \Pi_S^\bot + \mathfrak R_3 ) \Phi_2 \,.
\end{align}

\begin{remark}
$ R_3 $ has no longer the form \eqref{forma buona con gli integrali}. 
However $ R_3 = O( \partial_x^0 ) $ 
because $A = O(\pa_x^{-1})$ (see Lemma \ref{lemma:Dx A bounded}), 
and therefore $\Phi_2 - I_{H_S^\bot} = O(\partial_x^{-1})$. 
Moreover the matrix decay norm of $ R_3 $ is $ o(\e^2) $. 
\end{remark}

In order to normalize the term of order $\e^2$ 
of \eqref{L3 new KdV}, we develop 
$A_j^{j'}(\ph) = \sum_{l \in \Z^\nu} A_j^{j'}(l) e^{\ii l \cdot \ph}$, 
and for each $j, j' \in S^c$, $l \in \Z^\nu$, 
we choose 
\begin{equation}\label{cal A 1}
A_j^{j'}(l) := \begin{cases}
- \dfrac{\mB_j^{j'}(l)}{\ii (\omega \cdot l + m_3( j'^3 - j^3))} & \text{if} \  \bar{\omega} \cdot l + j'^3 - j^3 \neq 0 \, ,  \\
0 & \text{otherwise}.
\end{cases}
\end{equation}
This definition is well posed. 
Indeed, by \eqref{def cal B1 B2} and \eqref{def bar pi}, 
\begin{equation}\label{def cal B1 b}
{\cal B}_{j}^{j'}(l) := 3 \varsigma \ii j 
\sum_{\begin{subarray}{c} j_1, j_2 \in S \\ j_1 + j_2 = j - j' \\ \ell(j_1) + \ell(j_2) = l \end{subarray}} \sqrt{\xi_{j_1} \xi_{j_2}} \,. 
\end{equation}
In particular $ {\cal B}_{j}^{j'}(l) = 0 $ unless $ |l| \leq 2 $. 
For $|l| \leq 2$ and $ \bar{\omega} \cdot l + j'^3 - j^3 \neq 0 $, 
the denominators in \eqref{cal A 1} satisfy 
\begin{align}
|\omega \cdot l +m_3( j'^3 - j^3)| 
& = | m_3 (\bar \omega \cdot l + j'^3 - j^3) + ( \om - m_3 \bar \om ) \cdot l | 
\nonumber \\ & 
\geq  |m_3| | \bar \omega \cdot l + j'^3 - j^3 |
- | \om - m_3 \bar \om | |l| \geq  1/2   
\label{BNFdeno}
\end{align}
for $ \e $ small, because $ |\bar{\omega} \cdot l + j'^3 - j^3| \geq 1 $
($\bar{\omega} \cdot l + j'^3 - j^3$ is a nonzero integer), 
$ \om = \bar \om + O(\e^2) $ and by \eqref{stima m3}. 

\begin{remark} 
\label{rem:Ham solving homolog}
The operator $A$ defined in \eqref{cal A 1} is  Hamiltonian, 
because $\mB$ is Hamiltonian. 
The reason is a general fact: the denominators $ \d_{l,j,k} := \ii (\omega \cdot l + m_3( k^3 - j^3)) $ 
satisfy $ \overline{ \d_{l,j,k} } = \d_{-l,k,j} $ and an operator $G(\ph)$ is self-adjoint if and only if its matrix elements satisfy $ \overline{ G_j^k(l) } = G_k^j(-l) $, see \cite{BBM}-Remark 4.5. 
Alternatively, we could solve the homological equation of this Birkhoff step directly for the Hamiltonian function 
whose flow generates $ \Phi_2 $.
\end{remark}

By the definition \eqref{cal A 1}, 
the term of order $\e^2$ in \eqref{L3 new KdV} is zero 
on the Fourier indices $(l,j,j')$ such that $\bar \om \cdot l + j'^3 - j^3 \neq 0$, 
while it is equal to $\e^2 \mB_j^{j'}(l)$ for $(l,j,j')$ such that $\bar \om \cdot l + j'^3 - j^3 = 0$. 
Now we prove that the only nonzero components of $\mB$ that remain in \eqref{L3 new KdV} are $\mB_j^j(0)$.

\begin{lemma} \label{lem:bnf 1}
If $\bar\om \cdot l + j'^3 - j^3 = 0$ and $\mB_j^{j'}(l) \neq 0$, then $l=0$ and $j=j'$. 
\end{lemma}

\begin{proof}
If $\mB_j^{j'}(l) \neq 0$, then, by \eqref{def cal B1 b}, there exist $j_1, j_2 \in S$ such that 
$j_1 + j_2 = j - j'$ and $\ell(j_1) + \ell(j_2) = l$. 
Hence, recalling \eqref{bar omega} and \eqref{def:function-ell}, 
\[
0 = \bar \om \cdot l + j'^3 - j^3 
= \bar\om \cdot \ell(j_1) + \bar\om \cdot \ell(j_2) + j'^3 - j^3 
= j_1^3 + j_2^3 + j'^3 - j^3.
\]
This equality, together with $j_1 + j_2 + j' - j = 0$, 
implies that $(j_1 + j_2) (j_1 + j') (j_2 + j') = 0$ by Lemma \ref{lemma:interi}. 
Since $j_1, j_2 \in S$, $j' \in S^c$, the set $S$ is symmetric, and $0 \notin S$,
we deduce that the factors $j_1 + j'$ and $j_2+j'$ are nonzero. 
Hence $j_1 + j_2 = 0$, and therefore $l=\ell(j_1) + \ell(-j_1) = 0$. 
\end{proof}

Thus, the only nonzero term of order $\e^2$ in \eqref{L3 new KdV} is $\mB_j^j(0)$. 
By \eqref{def cal B1 b}, we calculate 
$\mB_j^j(0) = \ii j c(\xi)$, where $c(\xi)$ is defined in \eqref{c(xi)}. 
Hence, by \eqref{cal A 1}, Lemma \ref{lem:bnf 1} and \eqref{c(xi)}, 
the term of order $\e^2$ in \eqref{L3 new KdV} is 
\be\label{primo termine BNF1}
\e^2 \Pi_S^\bot \{ \mB + ({\cal D}_\om A) + m_3 [\partial_{xxx}, A] \} \Pi_S^\bot 
= \e^2 c(\xi) \pa_x \Pi_S^\bot \,. 
\ee

\begin{remark} \label{rem:lm M2-5}
When $ K = H + \lambda M^2 $, $ \lambda = 3 \varsigma / 4 $,  the operator in \eqref{def cal B1 B2} becomes 
$\mB h = \pa_x (3 \varsigma \pi_0(\bar v^2) h)$. Hence $\mB_j^j(0) = 0$, 
and the right-hand side term in \eqref{primo termine BNF1} is zero, 
namely the first step of linear Birkhoff normal form  
completely eliminates all the terms of order $\e^2$. 
\end{remark}

We now estimate the transformation $ A $.

\begin{lemma}\label{lem: A decay} 
$(i)$ For all $l \in \Z^\nu$, $j,j' \in S^c$, 
\begin{equation}\label{Acoef}
| A_j^{j'}(l)| \leq C (| j | +  | j'  |)^{-1}\, , \quad 
| A_j^{j'}(l)|^{\rm lip} \leq \e^{-2} (|j| + |j'|)^{-1} \,.
\end{equation}
$(ii)$ $ (A_1)_j^{j'}(l) = 0$ for all $l \in \Z^\nu$, $j,j' \in S^c$ such that $|j - j'| > 2 C_S $, 
where $C_S := \max\{ |j| : j \in S \}$. 
\end{lemma}
 
\begin{proof}
$(i)$ As already observed, for all $|l| > 2$ one has $ \mB_j^{j'}(l) = 0$, and therefore $ A_j^{j'}(l) = 0$. 
For $|l| \leq 2$, $ j \neq j' $, one has (since $ | \om | \leq |\bar \om | + 1 $) 
\[
|\om \cdot l + m_3 (j'^3 - j^3)| 
\geq |m_3||j'^3 - j^3| - |\om \cdot l| 
\geq \tfrac14 (j'^2 + j^2) - 2 |\om| 
\geq \tfrac18  (j'^2 + j^2) 
\]
for $(j'^2 + j^2) \geq C$, for some constant $C$. 
Since also \eqref{BNFdeno} holds, we deduce that, for all $ j \neq j' $, 
\begin{equation} \label{lower bound}
A_j^{j'}(l) \neq 0 
\quad \Rightarrow \quad
|\om \cdot l + m_3 (j'^3 - j^3)| \geq c ( | j | +  | j' | )^2 \, . 
\end{equation}
On the other hand, if $ j = j' \in S^c$, and $l \neq 0$, then $\mB_j^{j'}(l) = 0$, and therefore 
$A_j^{j'}(l) = 0$. For $j=j'$ and $l=0$ we also have $A_j^{j'}(l) = 0$ because $\bar \om \cdot l + j'^3 - j^3 = 0$.
Hence \eqref{lower bound} holds for all $ j, j'  $.
By \eqref{cal A 1}, \eqref{lower bound},  \eqref{def cal B1 b}  we deduce the first 
bound in \eqref{Acoef}. 
The Lipschitz bound follows similarly (use also $ |j - j'| \leq 2 C_S $).  
$(ii)$ follows by \eqref{cal A 1}-\eqref{def cal B1 b}. 
\end{proof}

The previous lemma means that $ A = O(| \partial_x|^{-1})$. 
More precisely, we deduce the following bound. 

\begin{lemma}[Lemma 8.19 of \cite{bbm}] \label{lemma:Dx A bounded}
$ | A \pa_x  |_s^\Lipg + | \pa_x A |_s^\Lipg \leq C(s) $.
\end{lemma}

It follows that the symplectic map $ \Phi_2 $ in \eqref{Phi_1} is invertible for $ \e $ small, 
with inverse
\begin{equation}\label{A1 check}
\begin{aligned}
& \Phi_2^{-1} = {\rm exp}(-\e^2 A) =  I_{H_S^\bot} + \e^2 {\check A} \,, \quad 
{\check A} := {\sum}_{n \geq 1} \frac{\e^{2n-2}}{n!} \, (-A)^n \,, \\
& | {\check A} \pa_x |_s^\Lipg + | \pa_x {\check A} |_s^\Lipg \leq C(s) \,. 
\end{aligned}
\end{equation}
By \eqref{L3 new KdV} and \eqref{primo termine BNF1} we get the Hamiltonian operator
\begin{align} \label{bernardino1}
{\cal L}_4 & := \Phi_2^{-1} {\cal L}_3 \Phi_2
= \Pi_S^\bot \big( {\cal D}_\om + m_3 \partial_{xxx} + (\e^2 c(\xi) + {\tilde d}_1) \pa_x \big) \Pi_S^\bot
+ R_4 \,,
\\
\label{bernardino 2}
R_4 & := (\Phi_2^{-1} - I) \Pi_S^\bot (\e^2 c(\xi) + {\tilde d}_1) \pa_x \Pi_S^\bot 
+ \Phi_2^{-1} R_3 \,.
\end{align}

\begin{lemma} \label{lemma:R4}
There is $ \s = \s(\nu,\t) > 0 $ (possibly larger than in Lemma \ref{lemma:stime coeff mL3}) 
such that 
\begin{equation} \label{stima Lip R4} 
\begin{aligned}
| R_4 |_s^\Lipg
& \leq_s \e^7 \g^{-2} + \e^2 \| \fracchi_\d \|_{s+\s}^\Lipg \,, 
\\
| \partial_i R_4 [\widehat \imath] |_s 
& \leq_s \e^{1+b} \| \widehat \imath \|_{s + \s} 
+ \e^2 \| \fracchi_\d \|_{s + \s} \| \widehat \imath \|_{ s_0 + \s} \,. 
\end{aligned}
\end{equation}
\end{lemma}

\begin{proof} 
Use \eqref{def R3}, \eqref{Phi_1}, \eqref{tilde d1 d0 KdV}, \eqref{nuove R3}, \eqref{stima m3}
and Lemma \ref{lemma:Dx A bounded}. 
\end{proof}

\subsection{Space reduction at the order $ \pa_x $}
\label{step5}

The goal of this section is to transform $ {\cal L}_5 $ in \eqref{bernardino1} 
so that the coefficient of $ \pa_x $ becomes constant.
We conjugate $ {\cal L}_4 $ via  a symplectic map of the form 
\begin{equation} \label{def descent}
{\cal S} := \exp(\Pi_S^\bot (w \partial_x^{-1})) \Pi_S^\bot 
= \Pi_S^\bot \big( I  + w   \partial_x^{-1} \big) \Pi_S^\bot + \widehat {\cal S}\,,
\end{equation}
where $\widehat {\cal S} := \sum_{k \geq 2} 
\frac{1}{k!} [\Pi_S^\bot (w \partial_x^{-1})]^k \Pi_S^\bot$
and $ w : \T^{\nu+1} \to \R $ is a function. 
Note that the linear operator $\Pi_S^\bot (w \partial_x^{-1}) \Pi_S^\bot$ 
is the Hamiltonian vector field  generated by the Hamiltonian 
 $ - \frac12 \int_\T w (\partial_x^{-1} h)^2\,dx$, $h \in H_S^\bot$.
We calculate
\begin{multline*} 
{\cal L}_4 {\cal S} - {\cal S} \Pi_S^\bot
( {\cal D}_\om + m_3 \partial_{xxx}  + m_1 \partial_x ) \Pi_S^\bot \\ 
= \Pi_S^\bot ( 3 m_3 w_x + \e^2 c(\xi) + \tilde{d}_1 - m_1 ) \partial_x \Pi_S^\bot + \tilde R_5 \,,
\end{multline*}
\vspace{-20pt}
\begin{align*}
\tilde R_5
& := \Pi_S^\bot \{ ( 3 m_3 w_{xx} + (\e^2 c(\xi) + \tilde d_1 - m_1) \Pi_S^\bot w ) \pi_0 
\\ & \qquad 
+ ( ({\cal D}_\om w) + m_3 w_{xxx} + (\e^2 c(\xi) + \tilde d_1) \Pi_S^\bot w_x ) \pa_x^{-1}
\\ & \qquad 
+ ({\cal D}_\omega \widehat{\cal S}) 
+ m_3 [\partial_{xxx}, \widehat{\cal S}] 
+ (\e^2 c(\xi) + \tilde d_1) \partial_x \widehat{\cal S} 
- m_1 \widehat{\cal S} \pa_x 
+ R_4 {\cal S} \} \Pi_S^\bot \,,
\end{align*}
where $\tilde R_5$ collects all the terms of order at most $\pa_x^0$. 
By \eqref{c(xi)}, we solve $ 3 m_3 w_x$ $+ \e^2 c(\xi) + \tilde{d}_1 - m_1 = 0 $  
by choosing $w := - (3 m_3)^{-1} \partial_x^{-1} ( \e^2 c(\xi) + \tilde{d}_1 - m_1 )$.
For $ \e $ small the operator $ {\cal S} $ is invertible, and we get 
\begin{equation}\label{def L6}
\mL_5 := \mS^{-1} \mL_4 \mS 
= \Pi_S^\bot ( {\cal D}_\om + m_3 \partial_{xxx} + m_1 \partial_x ) \Pi_S^\bot + R_5 \,, 
\quad R_5 :=  {\cal S}^{-1} \tilde R_5 \, .
\end{equation}
Since $ {\cal S} $ is symplectic, ${\cal L}_5$ is Hamiltonian (recall Definition \ref{operatore Hamiltoniano}).
By \eqref{tilde d1 d0 KdV}, \eqref{stima m1}, \eqref{stima m3}, one has
$\| w \|_s^\Lipg \leq_s  \e^7 \g^{-2} + \e^2 \| {\mathfrak I}_\delta\|_{s + \sigma}^\Lipg $.

\begin{lemma}\label{lemma L6}
 There is $ \s = \s(\nu,\t) > 0 $ (possibly larger than in Lemma \ref{lemma:R4}) such that 
\begin{align*}
|{\cal S}^{\pm 1} - I|_s^\Lipg 
& \leq_s \e^7 \g^{-2} + \e^2 \| {\mathfrak I}_\d \|_{s+\s}^\Lipg\,, \\
|\partial_i {\cal S}^{\pm 1} [\widehat \imath ]|_s  
& \leq_s \e^{2b} \| \widehat \imath\|_{s+\s} 
+ \e^5 \g^{-1} \|{\mathfrak I}_\d \|_{s + \s} \| \widehat \imath \|_{s_0 + \sigma} \,.
\end{align*}
The remainder $R_5$ satisfies the same estimates \eqref{stima Lip R4} as $R_4$. 
\end{lemma}

\subsection{KAM reducibility and inversion of $ {\cal L}_{\om} $} \label{subsec:mL0 mL5}

The coefficients $ m_3, m_1 $  of the operator $ {\cal L}_5 $ in \eqref{def L6} are constants, 
and the remainder $ R_5 $ is a bounded operator of order $ \partial_x^0 $ with  small matrix decay 
norm, see \eqref{decay R6}.
Then we  can diagonalize $ {\cal L}_5 $ by applying the iterative KAM reducibility 
Theorem 4.2 in \cite{BBM} along the sequence of scales  
\begin{equation}\label{defN}
N_n := N_{0}^{\chi^n}, \quad n = 0,1,2,\ldots,  
\quad  \chi := 3/2, \quad N_0  > 0 \, .
\end{equation}
In section \ref{sec:NM}, the initial $ N_0 $ will (slightly) increase to infinity 
as $ \e \to 0 $, see \eqref{nash moser smallness condition}. 
The required smallness condition (see (4.14) in \cite{BBM}) is (written in the present notations)
\be\label{R6resto}
N_0^{C_0} | R_5 |_{s_0 + \beta}^{\Lipg} \g^{-1} \leq 1 
\ee
where $ \b := 7 \tau + 6  $ (see (4.1) in \cite{BBM}), 
$ \tau $ is the diophantine exponent in \eqref{omdio} and \eqref{Omegainfty}, 
and the constant $ C_0 := C_0 (\t, \nu ) > 0 $  is fixed in
Theorem 4.2 in \cite{BBM}.
By Lemma \ref{lemma L6}, the remainder $ R_5 $ satisfies the bound \eqref{stima Lip R4}, 
and using \eqref{ansatz delta} we get (recall \eqref{link gamma b})
\be \label{decay R6}
| R_5|_{s_0 + \beta}^{\Lipg} \leq C \e^7 \g^{-2} = C \e^{3-2a},  
\quad 
| R_5 |_{s_0 + \beta}^{\Lipg} \g^{-1} \leq C \e^7 \g^{-3} = C \e^{1 - 3 a}. 
\ee
We use that  $ \mu $ in \eqref{ansatz delta}  is assumed to satisfy $ \mu  \geq \s + \b $ where $ \s := \s (\tau, \nu ) $ 
is given in Lemma \ref{lemma L6}.  

\begin{theorem} \label{teoremadiriducibilita} 
{\bf (Reducibility)}
Assume that  $\omega \mapsto i_\d (\omega)
$ is a Lipschitz function defined on some subset $\Omega_o \subset \Omega_\e $ (recall  \eqref{Omega epsilon}), satisfying 
\eqref{ansatz delta} with  $ \mu \geq \s + \b $, 
where $ \s := \s (\tau, \nu) $ is given in Lemma \ref{lemma L6} and $ \b  := 7 \tau + 6 $. 
Then there exists $ \delta_{0} \in (0,1) $  such that, if
\begin{equation}\label{condizione-kam}
N_0^{C_0}  \e^7 \g^{-3} = N_0^{C_0}  \e^{1 - 3 a}
\leq \delta_{0} \, , \quad \g := \e^{2b}:= \e^{2 + a} \, , \quad a \in (0,1/6) \, , 
\end{equation}
then:

$(i)$ {\bf (Eigenvalues)}.
For all $ \omega \in \Omega_\e $ there exists a sequence 
\begin{equation} \label{espressione autovalori}
\mu_j^\infty(\omega) := \mu_j^\infty(\omega, i_\d (\om)) 
:=  \ii \big( - {\tilde m}_3 (\omega) j^3 +  {\tilde m}_1(\omega)  j \big)  
+ r_j^\infty(\omega), \quad j \in  S^c \, ,  
\end{equation}
where $ {\tilde m}_3, {\tilde m}_1$  coincide with the coefficients $m_3, m_1$ of $ {\cal L}_5 $ in \eqref{def L6} for all $ \omega \in \Omega_o $, and 
\begin{equation} \label{autofinali}
\begin{aligned}
| {\tilde m}_3 - 1 |^\Lipg 
& \leq C \e^3, \qquad  
| {\tilde m}_1 - \e^2 c(\xi) |^\Lipg \leq C \e^5 \g^{-1}, \\
| r^{\infty}_j |^\Lipg 
& \leq C \e^{3 - 2 a} 
\quad \forall j \in  S^c
\end{aligned}
\end{equation}
for some  $ C > 0 $
(and $c(\xi)$ is defined in \eqref{c(xi)}).
All the eigenvalues $\mu_j^{\infty}$ are purely imaginary. 
We define, for convenience, $\mu_0^\infty (\om) := 0$. 

\smallskip

$(ii)$ {\bf (Conjugacy)}.
For all $\omega$ in the set
\begin{align}  
\Omega_\infty^{2\g} := \Omega_\infty^{2\g} (i_\d) 
& := \Big\{ \omega \in \Omega_o : \, 
| \ii\omega \cdot l + \mu^{\infty}_j (\omega) - \mu^{\infty}_{k} (\omega) |
\geq \frac{2 \gamma | j^{3} - k^{3} |}{ \langle l \rangle^{\tau}}  
\notag \\ 
& \hspace{120pt}
\forall l \in \Z^{\nu}, \ \forall j ,k \in S^c \cup \{0\} \Big\} 
\label{Omegainfty}
\end{align}
there is a real, bounded, invertible linear operator $\Phi_\infty(\omega) : H^s_{S^\bot} (\T^{\nu+1}) \to 
H^s_{S^\bot} (\T^{\nu+1}) $, with bounded inverse 
$\Phi_\infty^{-1}(\omega)$, that conjugates $\mL_6$ in \eqref{def L6} to constant coefficients, namely
\begin{equation}\label{Lfinale}
\begin{aligned}
{\cal L}_{\infty}(\omega)
& := \Phi_{\infty}\inv(\omega) \circ \mL_5(\omega) \circ  \Phi_{\infty}(\omega)
=  \om \cdot \partial_{\vphi} + {\cal D}_{\infty}(\omega),  \\
{\cal D}_{\infty}(\omega)
& := {\rm diag}_{j \in S^c} \{ \mu^{\infty}_{j}(\omega) \} \, .
\end{aligned}
\end{equation}
The transformations $\Phi_\infty, \Phi_\infty\inv$ are close to the identity in matrix decay norm,
with
\begin{equation} \label{stima Phi infty}
| \Phi_{\infty} - I |_{s,\Omega_\infty^{2\g}}^{{\rm Lip}(\gamma)}
+ | \Phi_{\infty}^{- 1} - I |_{s,\Omega_\infty^{2\g}}^\Lipg
\leq_s \e^7 \g^{-3} + \e^2 \g^{-1} \| {\mathfrak I}_\delta \|_{s + \sigma}^\Lipg .
\end{equation}
Moreover $\Phi_{\infty}, \Phi_{\infty}^{-1}$  are symplectic, and 
$\mL_\infty $ is a Hamiltonian operator.
\end{theorem}

\begin{proof}
The proof closely follows the one of Theorem 4.1 in \cite{BBM}, 
which is based on Theorem 4.2, Corollaries 4.1, 4.2 and Lemmata 4.1, 4.2 of \cite{BBM}.
Here $\om \in \R^\nu$, while in \cite{BBM} the parameter $\lm \in \R$, 
but Kirszbraun's Theorem on Lipschitz extension also holds in $\R^\nu $.  
The bound \eqref{stima Phi infty} follows by Corollary 4.1 of \cite{BBM} and the estimate of 
$ R_5 $ in Lemma \ref{lemma L6} above.

To adapt the proof of \cite{BBM} to the present case, 
the only changes in the statement of Theorem 4.2 of \cite{BBM} are: 
$\e^{3-2a}$ instead of $\e$ in (4.18) of \cite{BBM}, 
and $\e^{1+b}$ instead of $\e$ in (4.23), (4.25) and (4.26) of \cite{BBM}.
The factor $\e^{1+b}$ comes from the bound for $\pa_i R_5$, 
see Lemma \ref{lemma L6} and \eqref{stima Lip R4}. 
\end{proof}

\begin{remark}
Theorem 4.2 in \cite{BBM} 
also provides the Lipschitz dependence of the (approximate) 
eigenvalues $ \mu_j^n $ with respect to the unknown $ i_0 (\vphi) $,  
which is used for the measure estimate (Lemma \ref{matteo 10}).
 \end{remark}

All the parameters $ \omega \in \Omega_\infty^{2 \gamma} $ satisfy  (specialize \eqref{Omegainfty} for $ k = 0 $) 
\begin{equation}\label{prime di melnikov}
|\ii \omega \cdot l + \mu_j^\infty(\omega)| \geq 
2 \g | j |^3 \langle l \rangle^{-\tau} \, , \quad \forall l \in \Z^\nu , \ j \in S^c, 
\end{equation}
and the diagonal operator $ {\cal L}_\infty $ is invertible.  

In the following theorem we verify the inversion assumption \eqref{tame inverse} for ${\cal L}_\om $. 

\begin{theorem}\label{inversione linearized normale} {\bf (Inversion of $ {\cal L}_\om $)}
Assume the hypotheses of Theorem \ref{teoremadiriducibilita} and \eqref{condizione-kam}. 
Then there exists $ \s_1 := \s_1 ( \t, \nu ) >  0 $ such that, 
$ \forall \omega \in \Omega^{2 \gamma}_\infty(i_\d )$ (see \eqref{Omegainfty}),
for any function $ g \in H^{s+\s_1}_{S^\bot} (\T^{\nu+1})  $ 
the equation  ${\cal L}_\omega h = g$ 
has a solution $h = {\cal L}_\omega^{-1} g \in H^s_{S^\bot} (\T^{\nu+1})$, satisfying
\begin{align}\label{stima inverso linearizzato normale}
\| {\cal L}_\omega^{-1} g \|_s^{{\rm Lip}(\gamma)} 
& \leq_s \gamma^{-1} \big( \| g \|_{s +\sigma_1}^{{\rm Lip}(\gamma)} 
+ \e^2 \gamma^{-1} \| {\mathfrak I}_0 \|_{s + \sigma_1}^\Lipg
\| g \|_{s_0}^{{\rm Lip}(\gamma)} \big) \, .  
\end{align}
\end{theorem}

\begin{proof} See the proof of Theorem 8.16 in \cite{bbm}. 
\end{proof}

\section{The Nash-Moser nonlinear iteration}\label{sec:NM}
 
In this section we prove Theorem \ref{main theorem}. It will be a consequence of the Nash-Moser Theorem \ref{iterazione-non-lineare} below.

Consider the finite-dimensional subspaces
\[
E_n := \big\{ \fracchi (\vphi) = ( \Theta, y, z )(\vphi)  : \, \Theta = \Pi_n \Theta, \ y = \Pi_n y, \ z = \Pi_n z \big\}
\]
where $ N_n := N_0^{\chi^n} $ are introduced  in \eqref{defN}, and $ \Pi_n $ are the projectors 
(which, with a small abuse of notation, we denote with the same symbol)
\begin{equation} \label{Pin def}
\Pi_n \Theta (\ph) :=  \sum_{|l| < N_n} \Theta_l e^{\ii l \cdot \ph}, \quad 
\Pi_n z(\ph,x) := \sum_{|(l,j)| < N_n} z_{lj} e^{\ii (l \cdot \ph + jx)}, 
\end{equation}
where $\Theta (\ph) = \sum_{l \in \Z^\nu} \Theta_l e^{\ii l \cdot \ph}$
and $z(\ph,x) = \sum_{l \in \Z^\nu,  j \in S^c} z_{lj} e^{\ii (l \cdot \ph + jx)}$ 
(for $\Pi_n y(\ph)$ similar definition as for $\Pi_n \Theta(\ph)$).
We define $ \Pi_n^\bot := I - \Pi_n $.  
The classical smoothing properties hold: for all $\alpha , s \geq 0$, 
\begin{equation}\label{smoothing-u1}
\begin{aligned}
\|\Pi_{n} \fracchi \|_{s + \alpha}^\Lipg 
& \leq N_{n}^{\alpha} \| \fracchi \|_{s}^\Lipg  
\quad \forall \fracchi (\om) \in H^{s},   
\\
\|\Pi_{n}^\bot \fracchi \|_{s}^\Lipg 
& \leq N_{n}^{-\alpha} \| \fracchi \|_{s + \alpha}^\Lipg  
\quad \forall \fracchi (\om) \in H^{s + \alpha}.
\end{aligned}
\end{equation}
We define the constants
\begin{alignat}{3} \label{costanti nash moser}
& \mu_1 := 3 \mu + 9\,,\quad &
& \alpha := 3 \mu_1 + 1\,,\quad &
& \alpha_1 := (\alpha - 3 \mu)/2 \,, 
 \\
& \kappa := 3 \big(\mu_1 + \rho^{-1} \big)+ 1\,,\qquad &
& \beta_1 := 6 \mu_1+  3 \rho^{-1} + 3 \, ,  \qquad &
& 0 < \rho < \frac{1 - 3 a}{C_1(2 + 3 a)}\,, \label{def rho}
\end{alignat}
where $ \mu := \mu (\tau, \nu) $ is the ``loss of regularity'' defined in Theorem \ref{thm:stima inverso approssimato} 
(see \eqref{stima inverso approssimato 1}) and $ C_1 $ is fixed below. 

\begin{theorem}\label{iterazione-non-lineare} 
{\bf (Nash-Moser)} Assume that $ f \in C^q $ with 
$ q > s_0 + \b_1 + \mu + 3 $. 
Let $ \tau \geq \nu + 2 $. Then there exist $ C_1 > \max \{ \mu_1 + \a, C_0 \} $
(where $ C_0 := C_0 (\tau, \nu) $  is the one in Theorem \ref{teoremadiriducibilita}),  
$ \delta_0 := \d_0 (\tau, \nu) > 0 $ such that, if
\begin{equation}\label{nash moser smallness condition}  
N_0^{C_1}  \e^{b_* + 2} \gamma^{-2}< \d_0\,, \quad 
\gamma:= \e^{2 + a} = \e^{2b} \,,\quad 
N_0 := (\e^4 \g^{-3})^\rho\,,\quad 
b_* := 5 - 2 b \, , 
\end{equation}
then, for all $ n \geq 0 $: 
\begin{itemize}
\item[$({\cal P}1)_{n}$] 
there exists a function 
$(\fracchi_n, \zeta_n) : {\cal G}_n \subseteq \Omega_\e \to E_{n-1} \times \R^\nu$, 
$\omega \mapsto (\fracchi_n(\omega), \zeta_n(\omega))$,  
$ (\fracchi_0, \zeta_0) := 0 $, $ E_{-1} := \{ 0 \} $,   
satisfying $ | \zeta_n |^\Lipg \leq C \|{\cal F}(U_n) \|_{s_0}^\Lipg $, 
\begin{equation}\label{ansatz induttivi nell'iterazione}
\| \fracchi_n \|_{s_0 + \mu }^{{\rm Lip}(\gamma)} 
\leq C_* \e^{b_*} \gamma^{-1}\,, \quad 
\| {\cal F}(U_n)\|_{s_0 + \mu + 3}^{{\rm Lip}(\gamma)} \leq C_*\e^{b_*} \,, 
\end{equation}
where $U_n := (i_n, \zeta_n)$ with $i_n(\ph) = (\ph,0,0) + \fracchi_n(\ph)$.
The sets ${\cal G}_{n} $ are defined inductively by: 
\begin{multline} \label{def:Gn+1}
\begin{aligned}
{\cal G}_{0} 
& := \Big\{\omega \in \Omega_\e \, : \, |\omega \cdot l| \geq \frac{2 \g}{\langle l \rangle^{\tau}}
\, \ \forall l \in \Z^\nu \setminus \{0\} \Big\} \,, 
\\
{\cal G}_{n+1} 
& :=  \Big\{ \omega  \in {\cal G}_{n} \, : \, |\ii \omega \cdot l + \mu_j^\infty ( i_n) -
\mu_k^\infty ( i_n )| \geq \frac{2\gamma_{n} |j^{3}-k^{3}|}{\left\langle l\right\rangle^{\tau}}
\end{aligned}
\\ 
\forall j , k \in S^c \cup \{0\}, \  l \in \Z^{\nu} \Big\}\,,
\end{multline}
where $ \gamma_{n}:=\gamma (1 + 2^{-n}) $ and $\mu_j^\infty(\omega) := \mu_j^\infty(\omega, i_n(\omega)) $ 
are defined in \eqref{espressione autovalori} (and  $ \mu_0^\infty(\omega) = 0 $).

The difference $\widehat {\mathfrak I}_n := {\mathfrak I}_n - {\mathfrak I}_{n - 1} $ 
(where we set $ \widehat \fracchi_0 := 0 $) is defined on $\mG_n$, and it satisfies
\begin{equation}  \label{Hn}
\| \widehat {\mathfrak I}_1 \|_{ s_0 + \mu}^{\Lipg} \leq C_* \e^{b_*} \gamma^{-1} \, , \quad 
\| \widehat {\mathfrak I}_n \|_{ s_0 + \mu}^{\Lipg} \leq C_* \e^{b_*} \gamma^{-1} N_{n - 1}^{-\alpha_1} \quad \forall n > 1.
\end{equation}

\item[$({\cal P}2)_{n}$]   $ \| {\cal F}(U_n) \|_{ s_{0}}^{{\rm Lip}(\gamma)} \leq C_* \e^{b_*} N_{n - 1}^{- \alpha}$ 
where we set $N_{-1} := 1$.
\item[$({\cal P}3)_{n}$] \emph{(High norms).} 
\ $  \| \fracchi_n \|_{ s_{0}+ \beta_1}^{{\rm Lip}(\gamma)} \leq C_* \e^{b_*} \gamma^{-1}  N_{n - 1}^{\kappa} $ and  
$ \|{\cal F}(U_n ) \|_{ s_{0}+\beta_1}^{{\rm Lip}(\gamma)} \leq C_* \e^{b_*}   N_{n - 1}^{\kappa} $.

\item[$({\cal P}4)_{n}$] \emph{(Measure).} 
The measure of the ``Cantor-like'' sets $ {\cal G}_n $ satisfies
\begin{equation}\label{Gmeasure}
 | \Omega_\e \setminus {\cal G}_0 | \leq C_* \e^{2(\nu - 1)} \g \, , \quad  
\big| {\cal G}_n \setminus {\cal G}_{n+1} \big|  \leq  C_* \e^{2(\nu - 1)} \g N_{n - 1}^{-1}  \, . 
\end{equation}
\end{itemize}
All the Lip norms are defined on $ {\cal G}_{n} $, namely  $\| \ \|_s^{{\rm Lip}(\gamma)} = \| \ \|_{s,\mG_n}^{{\rm Lip}(\gamma)} $.\end{theorem}

\begin{proof}
To simplify notations, in this proof we denote $\| \, \|^{{\rm Lip}(\g)}$ by $\| \, \|$. 

\smallskip

{\sc Step 1:} \emph{Proof of} $({\cal P}1, 2, 3)_0$. 
Recalling \eqref{operatorF} we have $ \| {\cal F}( U_0 ) \|_s$ 
$= \| {\cal F}(\vphi, 0 , 0, 0 ) \|_s$ 
$= \| X_P(\vphi, 0 , 0 ) \|_s \leq_s \e^{5-2b} $ 
by Lemma \ref{lemma quantitativo forma normale}. 
Hence (recall that $ b_* := 5 - 2 b $) the smallness conditions in
$({\cal P}1)_0$-$({\cal P}3)_0$ hold taking $ C_* := C_* (s_0 + \b_1) $ large enough.

\smallskip

{\sc Step 2:} \emph{Assume that $({\cal P}1,2,3)_n$ hold for some $n \geq 0$, and prove $({\cal P}1,2,3)_{n+1}$.}
The proof of this step closely follows Step 2 in the proof of Theorem 9.1 of \cite{bbm}. 
We just mention the main changes: here it is convenient to define
\begin{equation}\label{riscalamenti nash moser}
w_n := \e^2 \g^{-2}  \|{\cal F}(U_n) \|_{s_0}\,,\quad 
B_n := \e^2 \g^{-1}\| \fracchi_n \|_{s_0 + \beta_1} + \e^2 \g^{-2}  \|{\cal F}(U_n) \|_{s_0 + \beta_1} \,,
\end{equation}
while the corresponding quantities defined in (9.18) of \cite{bbm} have $\e$ instead of $\e^2$ 
(and then, with definition \eqref{riscalamenti nash moser}, 
the bounds (9.19) of \cite{bbm} are also valid here without changes).
In the present case, the estimates (9.20)-(9.21) of \cite{bbm} for the quadratic Taylor remainder 
have to be adapted by replacing the factor $\e$ with $\e^2$. 
The reason for this improvement is that the nonlinearity in the mKdV equation is cubic, 
whereas in the KdV equation considered in \cite{bbm} the nonlinearity is just quadratic. 

\begin{remark} \label{rem:one power less} 
Since the KdV, respectively mKdV,  
nonlinearity is quadratic, respectively cubic, 
the smallness condition required in \cite{bbm} for the convergence of the Nash-Moser scheme is stronger than 
for Theorem \ref{iterazione-non-lineare}: it is 
 $ \e \| {\cal F}(\vphi, 0, 0 ) \|_{s_0+ \mu} \g^{-2} \ll 1 $ instead of  $ \e^2 \| {\cal F}(\vphi, 0, 0 ) \|_{s_0+ \mu} \g^{-2} \ll 1 $. 
As a consequence less steps of Birkhoff normal form are required
(namely less monomials to work out in the original Hamiltonian) 
to reach the sufficient smallness  $\mF(U_0) = O( \e^{5-2b}) $ 
to make the Nash-Moser scheme to converge
(in \cite{bbm} it is needed $\mF(U_0) = O( \e^{6-2b}) $).
\end{remark}

{\sc Step 3:} \emph{Prove $({\cal P}4)_n$ for all $n \geq 0$.} 
For all $n \geq 0$, the difference $\mG_n \setminus \mG_{n+1}$ 
is the union over $l \in \Z^\nu$, $j,k \in S^c \cup \{ 0 \}$ 
of the sets $R_{ljk}(i_n)$,
where 
\begin{equation}\label{resonant sets}
R_{ljk}(i_n) := \big\{ \omega \in {\cal G}_n \, : \, |\ii \omega \cdot l + \mu_j^\infty (i_{n}) -
\mu_k^\infty (i_{n})| < 2\gamma_{n} |j^{3}-k^{3}|\left\langle l\right\rangle^{- \tau}\big\}\,.
\end{equation}
Since $R_{ljk}(i_n) = \emptyset$ for $j = k$, 
in the sequel we assume that $j \neq k$.

\begin{lemma} \label{matteo 10}
For $n \geq 1$, $|l| \leq N_{n - 1}$, one has the inclusion 
$R_{ljk}(i_n) \subseteq R_{ljk}(i_{n - 1}) $. 
\end{lemma}

\begin{proof}
The proof closely follows the one of Lemma 5.2 in \cite{BBM}. 
The differences are that here the vector $\omega$ is not confined along a fixed direction, 
here we have $N_{n-1}$ instead of $N_n$,
and the factor $\e$ in (5.28) and (5.33) of \cite{BBM} is replaced here by $\e^7 \g^{-2} = \e^{3-2a}$. 

In the proof we use \eqref{Hn}, \eqref{decay R6}, \eqref{stima m3}, \eqref{stima m1}, 
and the bounds (4.25), (4.26), (4.34) of \cite{BBM} adapted to the present case 
(the bounds (4.25), (4.26) of \cite{BBM} hold here with $\e^{1+b}$ instead of $\e$, 
as already pointed out in the proof of Theorem \ref{teoremadiriducibilita}; 
the bound (4.34) of \cite{BBM} holds here with no change). 
\end{proof}

By definition, $ R_{ljk} (i_n) \subseteq {\cal G}_n $ (see \eqref{resonant sets}).
By Lemma \ref{matteo 10}, for $n \geq 1$ and $ |l| \leq N_{n-1} $ 
we also have $R_{ljk}(i_n) \subseteq R_{ljk}(i_{n - 1}) $. 
On the other hand, $ R_{ljk}(i_{n-1}) \cap {\cal G}_{n} = \emptyset $ (see \eqref{def:Gn+1}). 
As a consequence, $ R_{ljk} (i_n) = \emptyset $ for all $ |l| \leq N_{n-1} $,
and
\begin{equation} \label{Gn Gn+1}
{\cal G}_n \setminus {\cal G}_{n+1} \subseteq 
\bigcup_{\begin{subarray}{c} j, k \in S^c \cup \{0\} \\ 
|l| > N_{n-1} \end{subarray}} 
R_{ljk} ( i_n)  \quad \forall n \geq 1. 
\end{equation}

\begin{lemma}\label{matteo 4}
Let  $n \geq 0$. If $R_{ljk}(i_n) \neq \emptyset$, 
then $|l| \geq C_1 |j^3 - k^3| \geq \frac12 C_1 (j^2 + k^2) $ for some constant $C_1 > 0$
(independent of $l,j,k,n,i_n,\om$).
\end{lemma}

\begin{proof}
Follow the proof of Lemma 5.3 of \cite{BBM}, also using \eqref{autofinali}.  
Note that $|\om| \leq 2 |\bar\om|$ for all $\om \in \Omega_\e$, for $\e$ small enough, 
by \eqref{Omega epsilon} and \eqref{mappa freq amp}.
\end{proof}

Now we study the measure of the resonant sets $R_{ljk}(i_n)$ defined in \eqref{resonant sets}. 
We have to analyze in more details the sublevels of the function 
\be\label{def:phi}
\om \mapsto \phi(\om) :=
\ii \omega \cdot l + \mu_j^\infty (\om) - \mu_k^\infty (\om),
\ee
appearing in \eqref{resonant sets} 
($\phi$ also depends on $l,j,k,i_n$). 

\begin{lemma} \label{lemma:res alti}
There exists $C_0 > 0$ such that for all $j \neq k$, with $j^2 + k^2 > C_0$,
the set $R_{ljk}(i_n)$ has Lebesgue measure
$|R_{ljk}(i_n)| \leq C \e^{2(\nu-1)} \g \la l \ra^{-\t}$.
\end{lemma}

\begin{proof}
For $l \neq 0$, decompose $\om = s \hat l + v$, where $\hat l := l / |l|$, $s \in \R$, and $l \cdot v = 0$ (so that $\om \cdot l = s |l|$). 
Let $\psi(s) := \phi(s \hat l + v)$. 
The eigenvalues $\mu_j^\infty$ are given in \eqref{espressione autovalori}.
By \eqref{c(xi)} and \eqref{linkxiomega}, 
$
\e^2 |c(\xi)|^\lip \leq C_2 
$
for some constant $C_2 >  0 $ depending only on the set $S$ of the tangential sites.
Then, by \eqref{autofinali} and \eqref{def norma Lipg}, 
\begin{align*}
|\tilde m_3(s_1) - \tilde m_3(s_2)| & \leq C \e^3 \g^{-1} |s_1 - s_2|,  
\\  
|\tilde m_1(s_1) - \tilde m_1(s_2)| 
& 
\leq (C_2 + C \e^5 \g^{-2})|s_1 - s_2| 
\leq 2 C_2 |s_1 - s_2|,
\\
|r_j^\infty(s_1) - r_j^\infty(s_2)| 
& \leq C \e^{3-2a} \g^{-1} |s_1 - s_2|
\end{align*} 
for some $C > 0$ and $\e$ small enough, where, with a slight abuse of notations, we have written 
$$
\tilde m_i(s) = \tilde m_i (s \hat l + v)\,, \quad i = 1, 3 \quad \text{and} \quad r_j^\infty(s) = r_j^\infty(s \hat l + v)\,, \quad j \in S^c\,.
$$
By \eqref{espressione autovalori} and Lemma \ref{matteo 4},  
\begin{align*}
|\psi(s_1) - \psi(s_2)| 
& \geq \big( |l| - C \e^3 \g^{-1} |j^3 - k^3| - 2 C_2 |j-k| - 2C \e^{3-2a} \g^{-1} \big) |s_1 - s_2|
\\
& \geq |j^3 - k^3| \Big( C_1 - C \e^3 \g^{-1} - \frac{2 C_2 |j-k|}{|j^3 - k^3|} \, 
- \frac{2C \e^{3-2a} \g^{-1}}{|j^3 - k^3|} \Big) |s_1 - s_2|
\\
& \geq \frac{C_1}{2} \,|j^3 - k^3| |s_1 - s_2|
\end{align*}
for $\e$ small enough and  
$j^2 + k^2 + jk > C_0 := 12 C_2 / C_1$. 
As a consequence, the set $\Delta_{ljk}(i_n) := \{ s : s \hat l + v \in R_{ljk}(i_n) \}$ 
has Lebesgue measure 
\[
| \Delta_{ljk}(i_n) | 
\leq \frac{2}{C_1 |j^3 - k^3|} \, \frac{4 \g_n |j^3 - k^3|}{\la l \ra^\t}
\leq \frac{C \g}{\la l \ra^\t}
\]
for some $C > 0$. The lemma follows by Fubini's Theorem.  
\end{proof}

\begin{remark} \label{rem:lm M2-6}
When $ K = H + \lambda M^2 $, $ \lambda = 3 / 4 $,  
using \eqref{m1 IF},  the conclusion of Lemma \ref{lemma:res alti} holds without restrictions on $j,k$.
\end{remark}

It remains to estimate the measure of the finitely many resonant sets $R_{ljk}(i_n)$ for $j^2 + k^2 \leq C_0$. 
Recalling \eqref{c(xi)} and the parity $\xi_{-j} = \xi_j$, 
we write $c(\xi) = 6 \varsigma \vec{1} \cdot \xi$
where $\vec{1}$ is the vector $(1, \ldots, 1) \in \R^\nu$
and $\xi = (\xi_j)_{j \in S^+} \in \R^\nu$.  
Hence, by \eqref{linkxiomega}, 
\begin{equation} \label{alg}
\e^2 c(\xi) 
= 6 \varsigma \vec{1} \cdot \mathbb{A}^{-1} [\om - \bar\om] 
= 6 \varsigma \mathbb{A}^{-T} \vec{1} \cdot [\om - \bar\om]
\end{equation} 
where $\mathbb{A}^{-T}$ is the transpose of $\mathbb{A}^{-1}$. 
We write the function $ \phi (\om)  $ in \eqref{def:phi} as 
\[
\phi(\om) = a_{jk} + b_{ljk} \cdot \om + q_{jk}(\om) \, ,
\]
where 
\begin{align*}
a_{jk} & := - \ii \big(  j^3 - k^3 + 6 \varsigma (j-k) \vec{1} \cdot \mathbb{A}^{-1} \bar\om \big) , \\
b_{ljk} &  := \ii \big( l + 6 \varsigma (j-k) \mathbb{A}^{-T} \vec{1} \big) ,
\\
q_{jk}(\om) 
& := - \ii (\tilde m_3 -1) (j^3 - k^3) + \ii (\tilde m_1 -\e^2 c(\xi)) (j-k) 
+ r_j^\infty - r_k^\infty 
\end{align*}
(and $\tilde m_3, \tilde m_1, \xi, r_j^\infty, r_k^\infty$ all depend on $\om$). 
By \eqref{autofinali} and since $ j^2 + k^2 \leq C_0 $ we deduce that $ |q_{jk}|^\Lipg \leq C \e^{3-2a} $.
Recalling \eqref{def norma Lipg} we get
\be\label{stima:qjl}
|q_{jk}|^{\rm sup} \leq C \e^{3-2a} \, , \quad
|q_{jk}|^\lip 
\leq \g^{-1} |q_{jk}|^\Lipg
\leq C \e^{1-3a} 
\ee
so that $ \phi (\om) $
is a small perturbation of the affine function $ \om \mapsto a_{jk} + b_{ljk} \cdot \om $. 
By the next lemma, the hypothesis \eqref{scelta siti} on the tangential sites $ S$ allows to verify that such function does not vanish identically. 

\begin{lemma} \label{lemma:a jk} Assume  \eqref{scelta siti}. 
Then, for all $j \neq k$, $j^2 + k^2 \leq C_0$ it results  $a_{jk} \neq 0$. 
\end{lemma}

\begin{proof} Using formulae \eqref{bar omega} and \eqref{AA -1}, we calculate
\[
\vec{1} \cdot \mathbb{A}^{-1} \bar\om = - \frac{1}{3\varsigma (2\nu -1)} \, 
\sum_{i=1}^\nu \bar \jmath_i^{\,2}.   
\]
Hence 
\[
a_{jk} = - \ii (j-k) \Big(  j^2 + jk + k^2 - \frac{2}{2\nu -1} \, 
\sum_{i=1}^\nu \bar \jmath_i^{\,2} \Big) \neq 0
\]
by assumption \eqref{scelta siti} on the set $S$.  
\end{proof}

Lemma \ref{lemma:a jk} implies that $\d := \min\{ |a_{jk}| : j^2 + k^2 \leq C_0, \ j \neq k \} > 0$. 

\begin{lemma} \label{lemma:res bassi}
Assume  \eqref{scelta siti}.  If $j^2 + k^2 \leq C_0$, then 
$|R_{ljk}(i_n)| \leq C \e^{2(\nu-1)} \g \la l \ra^{-\t}$.
\end{lemma}

\begin{proof} 
Denote $b := b_{ljk}$ for brevity.
For $ j^2 + k^2 \leq C_0$, $\om \in R_{ljk}(i_n)$, one has, by \eqref{resonant sets}, \eqref{stima:qjl},  
\[
| b \cdot \om |
\geq |a_{jk}| - |\phi(\om)| - |q_{jk}(\om)| 
\geq \d - 2 \g_n |j^3 - k^3| \la l \ra^{-\t} - C \e^{3-2a} 
\geq \d/2
\]
for $\e$ small enough. 
On the other hand, $| b \cdot \om | \leq 2 | \bar \om| |b|$
because $|\om| \leq 2 |\bar\om|$ (see \eqref{Omega epsilon} and \eqref{mappa freq amp}).
Hence $|b| \geq \d_1$ where $\d_1 := \d / (4 |\bar\om|) > 0$. 
Split $\om = s \hat b + v$ where $\hat b := b / |b|$ 
and $v \cdot b = 0$. Let $\psi(s) := \phi( s \hat b + v )$. By \eqref{stima:qjl}, for $\e$ small enough,  
we get 
$$
| \psi(s_1) - \psi(s_2)| \geq (|b| - |q_{jk}|^\lip) |s_1 - s_2|
\geq \frac{\d_1}{2}  |s_1 - s_2| \, .
$$ 
Then we proceed similarly as in the proof of Lemma \ref{lemma:res alti}.  
\end{proof}

The proof of \eqref{Gmeasure} follows from 
the lemmata \ref{matteo 10}, \ldots, \ref{lemma:res bassi},  
proceeding like in \cite{BBM} 
(see the conclusion of the proof of Theorem 5.1 in \cite{BBM}).  
\end{proof}

\noindent
{\bf Proof of Theorem \ref{main theorem} concluded.}  
The conclusion of the proof of Theorem \ref{main theorem} follows exactly like in \cite{bbm} 
(see ``Proof of Theorem 5.1 concluded'' in \cite{bbm}). 

\begin{remark} \label{rem:lm M2-7}
By remark \ref{rem:lm M2-6}, 
Lemma \ref{lemma:a jk} (which is the only point in the paper
where assumption \eqref{scelta siti} is used)
is not needed any more. 
Thus Theorem \ref{thm:mKdV}   
applies to $ K = H + (3\varsigma/4) M^2$ without 
assuming hypothesis \eqref{scelta siti}.
\end{remark}

\bigskip

\noindent
\textbf{Acknowledgements.}
This research was supported by the European Research Council under FP7 
and PRIN 2012 ``Variational and perturbative aspects of nonlinear differential problems''.
This research was carried out in the frame of Programme STAR, financially
supported by UniNA and Compagnia di San Paolo.

\begin{footnotesize}

\end{footnotesize}

\bigskip

\noindent
\textbf{Pietro Baldi}\\
Dipartimento di Matematica e Applicazioni ``R. Caccioppoli'' \\
Universit\`a di Napoli Federico II\\  
Via Cintia, Monte S. Angelo, 80126 Napoli, Italy \\
\emph{Email:} {pietro.baldi@unina.it} 

\bigskip

\noindent
\textbf{Massimiliano Berti}\\
SISSA\\ 
Via Bonomea 265, 34136 Trieste, Italy \\
\emph{Email:} {berti@sissa.it} 

\bigskip

\noindent
\textbf{Riccardo Montalto}\\
Institut f\"ur Mathematik\\ 
Universit\"at Z\"urich\\
Winterthurerstrasse 190, CH-8057 Z\"urich \\
\emph{Email:} {riccardo.montalto@math.uzh.ch}


\begin{thebibliography}{10}

\bibitem{Alazard-Baldi} 
Alazard T., Baldi P., 
\emph{Gravity capillary standing water waves}, 
Arch. Ration. Mech. Anal. 217 (2015), no.\,3, 741-830.

\bibitem{Baldi-Benj-Ono} 
Baldi P., 
\emph{Periodic solutions of fully nonlinear autonomous equations of Benjamin-Ono type}, 
Ann. Inst. H. Poincar\'e (C) Anal. Non Lin\'eaire 30 (2013), 33-77.

\bibitem{BBM} 
Baldi P., Berti M., Montalto R., 
\emph{KAM for quasi-linear and fully nonlinear forced perturbations of Airy equation}, 
Math. Annalen 359, 471-536 (2014). 

\bibitem{BBM1} 
Baldi P., Berti M., Montalto R., 
\emph{KAM for quasi-linear KdV},  
C. R. Acad. Sci. Paris, Ser. I 352 (2014) 603-607.

\bibitem{bbm} 
Baldi P., Berti M., Montalto R., 
\emph{KAM for autonomous quasi-linear perturbations of KdV}, 
to appear on Ann. Inst. H. Poincar\'e (C) Anal. Non Lin\'eaire. 

\bibitem{Baldi-Floridia-Haus} 
Baldi P., Floridia G., Haus E., 
\emph{Exact controllability for quasi-linear perturbations of KdV},
preprint. 

\bibitem{BBiP1} 
Berti M., Biasco P., Procesi M., 
{\it KAM theory for the Hamiltonian DNLW}, 
Ann. Sci. \'Ec. Norm. Sup\'er. (4), Vol. 46, fascicule 2 (2013), 301-373.

\bibitem{BBiP2}
Berti M., Biasco P., Procesi M., {\it KAM theory for the reversible derivative wave equation},
Arch. Rational Mech. Anal., 212 (2014), 905-955. 

\bibitem{BB13JEMS} Berti M., Bolle P.,  
{\it Quasi-periodic solutions with Sobolev regularity of NLS on $ \T^d $ with a multiplicative potential}, 
J. Eur. Math. Soc. 15 (2013), 229-286.

\bibitem{BB13} Berti M., Bolle P., {\it A Nash-Moser approach to KAM theory}, 
Fields Institute Communications, special volume ``Hamiltonian PDEs and Applications'', to appear.

\bibitem{Berti-Montalto} Berti M., Montalto R., \emph{KAM for gravity capillary water waves},
preprint. 

\bibitem{B96}  Bourgain J.,  {\it Gibbs measures and quasi-periodic solutions for nonlinear Hamiltonian partial differential equations}, 
23-43, Gelfand Math. Sem., Birkh\"auser Boston, Boston, MA, 1996.

\bibitem{FP} Feola R., Procesi M. {\it Quasi-periodic solutions for fully nonlinear forced
reversible Schr\"odinger equations}, J. Diff. Eq., 259, no. 7, 3389-3447, 2015.

\bibitem{K13} Guan H., Kuksin S.,   
{\it The KdV equation under periodic boundary conditions and its perturbations}, 
Nonlinearity 27 (2014), no.\,9, R61-R88.

\bibitem{IP09}  Iooss G.,   Plotnikov P.I.,  
{\it Small divisor problem in the theory of three-dimensional water gravity waves}, 
Mem. Amer. Math. Soc. 200, no. 940 (2009).

\bibitem{Ioo-Plo-Tol} Iooss G.,   Plotnikov P.I., Toland J.F., 
{\it Standing waves on an infinitely deep perfect fluid under gravity}, 
Arch. Rational Mech. Anal. 177  no. 3, (2005), 367-478. 

\bibitem{Lax} Lax P.,
{\it Development of singularities of solutions of nonlinear
hyperbolic partial differential equations}, J. Mathematical Phys. 5 (1964), 611-613.

\bibitem{LY}  Liu J., Yuan X.,
{\it A KAM Theorem for Hamiltonian Partial Differential
Equations with Unbounded Perturbations}, Comm. Math. Phys, 307 (3) (2011), 629-673.

\bibitem{KaP} Kappeler T., P\"{o}schel J., {\it KAM and KdV}, Springer, 2003.

\bibitem{KaT}  Kappeler, T., Topalov, P. 
{\it Global well-posedness of mKdV in $ L^2 (T,R)$}, 
Comm. Partial Differential Equations 30 (2005), no. 1-3, 435-449.

\bibitem{KM} Klainerman S., Majda A.,
{\it Formation of singularities for wave equations including the
nonlinear vibrating string}, Comm. Pure Appl. Math., 33,  (1980), 241-263.

\bibitem{Ku}  Kuksin S.,
{\it Hamiltonian perturbations of infinite-dimensional linear
systems with imaginary spectrum},
Funktsional. Anal. i Prilozhen. 21, no. 3, 22--37, 95, 1987.

\bibitem{K2} Kuksin S., {\it A KAM theorem for equations of the Korteweg-de Vries type},
Rev. Math. Phys., 10, 3, (1998), 1-64.

\bibitem{k1}  Kuksin S.,
{\it Analysis of Hamiltonian PDEs},
Oxford Lecture Series in Mathematics and its Applications, 19.
Oxford University Press (2000).

\bibitem{KP} Kuksin S., P\"oschel J., {\it
Invariant Cantor manifolds of quasi-periodic oscillations
for a nonlinear Schr\"{o}dinger equation}, Annals of Math. 2  143, (1996), 149-179.

\bibitem{Po3} P\"oschel J., {\it Quasi-periodic solutions for
a nonlinear wave equation}, Comment. Math. Helv.,  71, no. 2, (1996)
269-296.

\bibitem{PP1}
Procesi M., Procesi C., 
{\it A normal form for the Schr\"odinger equation with analytic non-linearities}, Comm. Math. Phys. 312 (2012), 501-557. 

\bibitem{Taylor}  Taylor M. E., {\it Pseudodifferential Operators and Nonlinear PDEs}, 
Progress in Mathematics, Birkh\"auser, 1991.

\bibitem{ZGY} Zhang J.,  Gao M.,  Yuan X.
{\it KAM tori for reversible partial differential equations},
Nonlinearity 24 (2011), 1189-1228. 

\bibitem{Z1} Zehnder E., 
{\it Generalized implicit function theorems with applications to some small divisors problems I-II}, 
Comm. Pure Appl. Math. 28 (1975), 91-140, and 29 (1976), 49-113.
\end{thebibliography}
\end{document}